\documentclass{article}
\usepackage[utf8]{inputenc}
\usepackage{amsmath}
\usepackage{amssymb}
\usepackage{amsthm}
\usepackage{subfigure}
\usepackage{threeparttable}
\usepackage{rotating}
\usepackage{multirow}
\usepackage{array}
\usepackage{adjustbox}
\usepackage{textcomp}
\usepackage[title]{appendix}
\usepackage[ruled,vlined,linesnumbered]{algorithm2e}
\usepackage[allbordercolors={0.9 0.9 0.9}]{hyperref}
\usepackage{apacite}
\usepackage{geometry}
\usepackage{enumitem}
\setitemize[1]{itemsep=0pt,partopsep=0pt,parsep=\parskip,topsep=5pt}
\usepackage{url}

\newtheorem{definition}{Definition}
\newtheorem{proposition}{Proposition}
\newtheorem{lemma}{Lemma}
\newtheorem{remark}{Remark}

\title{An approximate dynamic programming approach to {the admission control of elective patients}}
\date{}

\begin{document}

\maketitle
\begin{center}
	Jian Zhang$^{a,b,*}$, Mahjoub Dridi$^b$, Abdellah El Moudni$^b$\par
\end{center}\vspace{5pt}

\noindent$^a$ School of Industrial Engineering, Eindhoven University of Technology, 5600MB Eindhoven, Netherlands\newline\newline
$^b$ Laboratoire Nanomédecine, Imagerie et Thérapeutiques, Université Bourgogne Franche-Comté, 
UTBM, Rue Thierry Mieg, 90010 Belfort cedex, France\newline\newline
* Corresponding author \newline\newline  
Email addresses: j.zhang4@tue.nl (J. Zhang) 
\par\setlength\parindent{7.7em} mahjoub.dridi@utbm.fr (M. Dridi)
\par\setlength\parindent{7.7em} abdellah.el-moudni@utbm.fr (A. El Moudni)
\newpage

\begin{abstract}
\normalsize\noindent
In this paper, we propose an approximate dynamic programming {(ADP)} {algorithm} to solve a Markov decision process (MDP) formulation for {the admission control of elective patients}. To manage the elective patients from multiple specialties equitably and efficiently, we establish a waiting list and assign each patient a time-dependent dynamic priority {score}. Then, {taking the random arrivals of patients} into account, sequential decisions are made on a weekly basis. At the end of each week, we select the patients to be treated in the following week from the waiting list. By minimizing the cost function of {the} MDP over {an} infinite horizon, we seek {to achieve} the best trade-off between the patients' waiting times and the over-utilization of {surgical} resources. Considering the curses of dimensionality resulting from the large {scale} of {realistically sized} problems, we first analyze the structural properties of the MDP and propose an algorithm that facilitates the search for {best} actions. {We then} develop a novel {reinforcement-learning-based ADP algorithm} as the solution technique. Experimental results reveal that the proposed algorithms consume much less computation time in comparison with {that required by} conventional dynamic programming methods. {Additionally, the algorithms are shown to be} capable of {computing high-quality} near-optimal policies for {realistically sized} problems. 

\noindent\textbf{\emph{Keywords}}: {patient admission control}, operating theatre planning, Markov decision process, approximate dynamic programming, recursive least-squares temporal difference learning
\end{abstract}

\setlength\parindent{1.75em}
\section{Introduction}

{Globally, the aging population and a rising quality of life are driving the demand for health services to increase rapidly, leading to shortages of medical resources and imposing a heavy financial burden on national governments. In the United States, the healthcare expenditure has been increasing continually, reaching \$3.5 trillion in 2017 and accounting for 17.9\% of the gross domestic product \shortcite{CMS}. The same figures for Australia in 2015–2016 were \$170 billion and 10.3\%, respectively, while health expenditure increased (by 50\%) much faster than the population growth (17\%) \shortcite{AIHW}. During the period from 2010 to 2017, China increased its government spending on health and the number of medical personnel employed by 163\% and 43\%, respectively. However, its health expenditure per capita was still much lower than that of developed countries and the increase in its health service capacity was not enough to cope with the rising demands on the service \shortcite{xiao2016stochastic,NBSC}. The same challenges have also been faced by European countries. For example, Portugal's surgical demand grew by 43.7\% from 2006 to 2014, while the median waiting time for surgery reached 3.0 months and 12\% of patients waited longer than their clinically recommended maximum waiting time \shortcite{marques2017different}.

In order to satisfy a growing demand for health services and to slow down the increase in health expenditure, hospital managers should improve the efficiency and quality of healthcare activities while reducing the hospitals' expenditures as much as possible. In a hospital, the operating theatre (OT) is generally considered to be both the main revenue center as well as the most expensive department since providing surgical services consumes more than 40\% of the hospital's budget and contributes a similarly large proportion to the hospital's total revenues \shortcite{denton2007optimization}. Therefore, the management of OT and the scheduling of surgeries have drawn much attention from both researchers and practitioners.

The complexity of OT management problems is the result of many factors including the expensiveness and shortage of OT resources \shortcite{rath2017integrated}, the conflicting interests of different stakeholders (e.g., patients and hospitals) \shortcite{marques2017different}, and the {uncertainty} associated with surgical activities \shortcite{banditori2013combined,jebali2015stochastic}. For these reasons, researchers have developed and applied various operations research methodologies to cope with OT management problems. In the various relevant works, OT management decisions} are generally divided into three hierarchical levels \shortcite{guerriero2011operational,zhu2019operating}: {the} strategic level determines the distribution of OR capacity among different specialties on a long-term basis; {the} tactical level involves the development of a master surgery schedule (MSS) for one or several months; {the} operational level concerns the assignment of a definite date {and a specific OR} for each {elective} surgery (advance scheduling) and the intra-day sequencing of { the scheduled surgeries} (allocation scheduling). The three decision levels are interrelated since each level depends on the decisions made at the {higher} level \shortcite{koppka2018optimal}.

This paper addresses {the admission control of elective patients} at the operational level. It is assumed that the allocation of OT resources among specialties {is} already determined by the strategic level and {that} an MSS has been fixed {at} the tactical level. A dynamic waiting list is built to manage the elective patients according to their specialties, urgency {coefficients} and {actual} waiting times. {At the end of each week, the surgery planner updates the waiting list by adding the newly arrived patients and removing the postoperative ones, then determines the number and type of elective cases that will be performed in the next week.} {Our objective is to shorten patients' waiting times and to optimize the utilization of OT resources{, including ORs and the recovery beds in a surgical intensive care unit (SICU)}.} Emergency cases are not involved in {the problem studied in this paper because} they can be treated by dedicated facilities \shortcite{dios2015decision,guido2017hybrid}. {Since} uncertainty is a key feature of surgical activities and significantly impacts the efficiency and quality of surgical services \shortcite{guerriero2011operational,jebali2015stochastic}, three major sources of {uncertainty} are considered in {our mathematical model and/or our numerical experiments: new patient arrivals, surgery durations, and the length of stay (LOS) of postoperative patients in} {the SICU}. 

{The patient admission control problem studied in this paper can be regarded as a subproblem of advance scheduling. In a typical advance scheduling problem, the {surgery planner} determines the patients to be treated in the current planning period (usually one week) and assigns these patients to specific ORs and surgery dates \shortcite<e.g., >{min2010scheduling,jebali2015stochastic,neyshabouri2017two,marques2017different,moosavi2018scheduling}. Such a problem is usually formulated as a pure mathematical programming model that only optimizes the cost of one single planning period without considering the {impact on} subsequent periods. However, two consecutive planning periods are always correlated since the postponed surgeries of the present period will continue to incur waiting costs or surgery costs in the next period. Therefore, optimizing the costs of each period separately cannot guarantee global optimality. In contrast, a patient admission control problem is usually formulated as a Markov decision process (MDP) \shortcite<e.g., >{patrick2008dynamic,min2010elective,min2014managing,astaraky2015simulation,truong2015optimal}. It focuses on the management of the patient waiting list and determines the selection of patients to be treated in each planning period (usually one week or one day). Though the MDP model does not address intra-week or intra-day assignment decisions, it optimizes the flow of patients and minimizes the expected total costs over an infinite horizon. Hence, the optimal policy of {an} MDP offers a better long-term performance than the myopic policy provided by a pure mathematical programming model. 

In this paper, we propose an MDP model for patient admission control that considers the needs of multiple specialties, the limited capacity of ORs and {SICU}, as well as time-dependent dynamic patient {priority scores}. Sequential decisions are made at the end of each week that determine the patients to be treated in the following week. Since the MDP model does not capture intra-week scheduling, hospital-related costs (incurred by overusing ORs and the SICU) cannot be exactly computed. Hence, for a given selection of patients, the cost function of {the MDP model} computes the exact patient-related costs incurred by performing and postponing surgeries and estimates the hospital-related costs based on total resource availability and the scheduled patients' total expected surgery durations and LOSs. While taking into account the random arrivals of new patients in each week, the MDP model provides an optimal policy that leads to the lowest expected total costs over an infinite horizon. Once the selection of patients is determined, these patients' assignments to surgical blocks (made up of a combination of ORs and dates) can be optimized using a stochastic programming model. For example, \shortciteA{zhang2019two} propose a two-level optimization model that combines MDP and stochastic programming to optimize all the decisions required by an advance scheduling problem. Their model addresses simple instances with a single specialty and two surgical blocks, while the MDP model proposed in this paper can be combined with a more complicated stochastic programming model in future research and can thereby be used to solve more realistic advance scheduling problems.

Compared to existing research into patient admission control \shortcite{patrick2008dynamic,min2010elective,min2014managing,astaraky2015simulation,truong2015optimal}, this work presents two main innovations{. First}, it proposes a comprehensive patient admission control model that simultaneously incorporates time-dependent dynamic patient {priority scores}, the needs of multiple specialties, and the capacity constraints of ORs and SICU{. Second}, in order to tackle the curses of dimensionality {for} realistically sized problems, this paper investigates the mathematical properties of the proposed model in depth and develops a novel approximate dynamic programming (ADP) algorithm based on recursive least-squares temporal difference (RLS–TD($\lambda$)) learning. The proposed algorithm is tested through extensive numerical experiments and its capability to solve realistically sized problems is validated.
}
 
The {remainder} of this paper is organized as follows: {Section \ref{SecLiterature} provides a comprehensive review of the existing literature on patient admission control and MDP}. Section \ref{SecModel} describes the studied {patient admission control} problem and presents the infinite-horizon MDP formulation. Section \ref{SecAna} addresses the structural analysis for the MDP model and proposes an algorithm to predigest the exploration of action space. The RLS–TD($\lambda$)-based ADP algorithm is introduced in Section \ref{SecSol}, and the results of numerical experiments are presented and discussed in Section \ref{SecExp}. Finally, the conclusions and several possible future extensions of this work are given by Section \ref{SecCon}. 

{
\section{Literature Review}\label{SecLiterature}
In this section, we review the existing literature relevant to patient admission control and the MDP.

\subsection{Patient admission control}\label{SubSecPAC}
Typically, in a patient admission control problem, the surgery planner manages a dynamic waiting list of elective patients and makes sequential decisions at the beginning (or end) of each planning period (day or week) to determine which patients are treated in the present (or the next) period. Since the patient-to-block assignments are not considered, patient admission control can be regarded as a subproblem of advance surgery scheduling. Among the existing research on advance surgery scheduling, mathematical programming (MP) models are commonly used to optimize the cost of one single planning period \shortcite<e.g., >{lamiri2008stochastic,min2010scheduling,jebali2015stochastic,neyshabouri2017two,moosavi2018scheduling}{. However, these MP models minimize the short-term cost myopically and may lead to high costs-to-go in the future. In contrast, patient admission control is usually formulated as an MDP model in order to minimize the expected total costs over a long planning horizon}. Though the MDP model does not address the issue of patient-to-block assignments, it can be combined with an MP model to optimize all the advance scheduling decisions. Such a combination may lead to a much better long-term performance in comparison to that achieved by a pure MP model \shortcite{zhang2019two}. This paper aims to formulate and efficiently solve a comprehensive MDP model for realistic patient admission control problems, while the optimization of patient-to-block assignments is left for further research. 

In the literature on {patient admission control and} surgery scheduling, elective patients and non-elective patients are two commonly addressed patient groups. The former group is usually added onto a waiting list before being treated{. These} patients' surgeries can be postponed for some period of time. In comparison, the latter group is made up of emergent patients (emergencies) that should be treated as soon as possible. Among the relevant research, two policies are commonly adopted to deal with the two different patient groups. The first one is a dedicated policy under which all emergent patients are channeled to dedicated surgical facilities, so that scheduled elective surgeries are not interrupted. When this policy is employed, the surgery planner focuses on the planning of elective surgeries using non-dedicated resources, while emergent patients {can usually be ignored} since they are treated in dedicated facilities on a first-come-first-served (FCFS) basis \shortcite<e.g., >{fei2009solving,denton2010optimal,jebali2015stochastic,addis2016operating,marques2017different,neyshabouri2017two,roshanaei2017propagating,roshanaei2020reformulation,zhang2020mitigating}. The second policy is more flexible{. It} allows emergent surgeries to be performed in any unoccupied OR. Scheduling surgeries under {this} flexible policy is more complex than doing so under the dedicated policy since emergencies introduce more uncertainty into the scheduling problem. In research into the employment of the flexible policy, the authors usually assume the emergency demand to be a stochastic parameter and optimize the surgical resources (e.g., OR capacity and recovery beds) reserved for emergencies in a way that balances the satisfaction of elective patients and the quick access to care of emergent ones \shortcite<e.g., >{wang2014column,truong2015optimal,rachuba2017fuzzy}. In addition to the two policies discussed above, some papers plan elective and non-elective surgeries using a hybrid policy that maintains both dedicated and versatile ORs \shortcite<e.g., >{tancrez2013assessing,hosseini2014evaluation,ferrand2014partially}. \shortciteA{ferrand2010comparing} present a study of a two-OR planning problem{. They} conclude that adopting the dedicated policy reduces the elective patients' waiting times and the overtime of ORs, but that it leads to longer waiting times for emergent patients. \shortciteA{duma2019management} provide a detailed comparison that reveals that the dedicated policy results in fewer elective surgery cancellations, higher resource utilization rates, and shorter waiting lists, and that the flexible policy brings about a better trade-off between the interests of elective and non-elective patients. Moreover, they establish that a further improvement can be achieved through a mixture of the two policies (i.e., hybrid policies). The authors also claim that the performances of different policies depend on the scenario and the operative conditions at hand{. Hence}, any policy could be the best one under a specific problem setting. In this paper, we {assume a} dedicated policy {meaning that} emergencies are excluded from our patient admission control problem.

Effectively managing the waiting list of elective patients is the key to patient admission control. To determine the relative priority of elective patients on the waiting list, it is important to adopt a proper prioritization system \shortcite{guerriero2011operational,van2015trade}. \shortciteA{patrick2008dynamic} classify elective patients into several urgency-related groups (URGs) with static {priority scores} and specify a priority-related waiting-time target for each group. However, \shortciteA{testi2007three} and \shortciteA{testi2008prioritizing} suggest that classifying patients into URGs is not enough to guarantee the efficiency and equity of schedules. They propose a time-dependent prioritization scoring system that uses the product of urgency {coefficient} and waiting time to capture a patient's priority. Further to this, \shortciteA{valente2009model} introduce a pre-admission model that has been put into practice in an Italian hospital. This model assesses the urgency level of each new patient and thereby determines a corresponding maximum time before treatment (MTBT). It then prioritizes the patient in real-time according to his/her urgency level and actual waiting time. Hospital data show that this pre-admission model allows homogeneous and standardized prioritization and enhances transparency, efficiency and equity. \shortciteA{min2014managing} conduct numerical simulations that compare the effects of static and time-dependent dynamic prioritization on patients' waiting time. The results indicate that the waiting time for non-urgent patients is much longer than that of urgent ones when the static priority {score} is applied, whereas the dynamic priority {score} minimizes the total weighted waiting time of all patients, so that the gap in the length of waiting time between urgent patients and non-urgent ones is significantly reduced. {In the relevant research on advance surgery scheduling and patient admission control, time-dependent priority scores are widely used as multipliers of the patient-related costs \shortcite<including surgery costs and waiting costs, e.g.,>{addis2016operating,agnetis2012long,agnetis2014decomposition,aringhieri2015two,jebali2015stochastic,jebali2017chance,lamiri2008column,lamiri2008stochastic,marques2017different,min2010scheduling,rachuba2017fuzzy,roshanaei2017propagating,roshanaei2020reformulation,tanfani2010pre,testi2009tactical,vancroonenburg2019chance,wang2016discrete,zhang2019markov,zhang2019two,zhang2020column}. With this setting, patients with the longer waiting times are given the higher priority scores, and the surgery planner is encouraged to schedule the surgeries as early as possible because the unscheduled patients will incur higher costs in the future. In addition to the time-dependent priority scores,} \shortciteA{neyshabouri2017two} introduce a {new multiplier} that addresses the relative importance of different specialties. This feature gives surgical cases from more important specialties (e.g., cardiology) higher priority than those from less important specialties (e.g., otolaryngology). Referring to these prioritization methods, we use the product of urgency {coefficient}, waiting time, and specialty-related importance factor to describe patient priority in this paper. {Moreover, we assume that every patient is given a maximum {recommended} waiting time (dependent on his/her specialty and urgency coefficient) and that any patient that is not scheduled in the currently considered planning period incurs a waiting cost (multiplied by his/her time-dependent priority score, as mentioned above). Thus, the surgery planner gives priority to the patients who join the waiting list earlier {with higher urgency coefficients}.}

{While ORs are considered to be the most important surgical resources and their utilization is optimized in almost all the relevant works, recently, more and more researchers take the supporting facilities of surgical activities into consideration and address multi-resource patient admission control and surgery scheduling problems \shortcite<e.g., >{min2010scheduling,gocgun2012lagrangian,huh2013multiresource,astaraky2015simulation,truong2015optimal,jebali2015stochastic,jebali2017chance,samudra2016scheduling,neyshabouri2017two,zhang2019two,zhang2020column}. Among the various supporting facilities, the downstream recovery units, such as SICU, have drawn much attention from the researchers and practitioners, because the unavailability of SICU beds may block the postoperative patients in ORs and affect the feasibility of the ongoing surgery schedule \shortcite{min2010scheduling,jebali2015stochastic}. Moreover, \shortciteA{jonnalagadda2005evaluation} point out that the unavailability of recovery room is the cause of 15\% of surgery cancellations in the hospital they have studied, and \shortciteA{utzolino2010unplanned} report that the readmission rate to SICU of the patients that are discharged due to the lack of recovery beds is almost 3 times that of the patients discharged electively. Given that SICU is an important surgical facility and usually becomes the bottleneck resource of an OT, in this work, we believe that considering the capacities of ORs and SICU together results in better resource utilization than only considering the OR capacity. By jointly planning ORs and SICU, the inconvenience and extra costs caused by exceeding the regular capacity of any of the two resources can be minimized, meanwhile the success rate of the intra-week surgery scheduling can be improved.

With different problem settings, patient admission control problems have been studied by many researchers.} \shortciteA{gerchak1996reservation} address a patient admission control problem in which decisions are made at the beginning of each day that determine the optimal number of patients to be treated that day. Similarly, \shortciteA{green2006managing} present a finite-horizon MDP formulation that deals with an intra-day diagnostic facility scheduling problem for multiple classes of patients. Further to this, \shortciteA{patrick2008dynamic} use an infinite-horizon MDP model to dynamically assign patients to diagnostic facilities. Their model differs from that of \shortciteA{gerchak1996reservation} in that it considers multiple priority classes and quantifies the actual waiting times of elective patients. They also introduce a booking horizon that consists of a number of future days, so that patients can be assigned to anyone of the days within the booking horizon at each decision epoch. A similar booking horizon is adopted by \shortciteA{liu2010dynamic} to schedule single-priority outpatient appointments. \shortciteA{min2010elective} extend the model developed by \shortciteA{gerchak1996reservation} to apply to a multi-priority case; they then evaluate the impact of patient priority and cost settings on the resulting surgery schedules. \shortciteA{min2014managing} go one step further and consider dynamic patient priority: each patient is initially assigned a fixed urgency {coefficient} when he/she joins the waiting list, then the product of the urgency {coefficient} and actual waiting time, which dynamically increases over time, determines his/her priority {score}. \shortciteA{astaraky2015simulation} and \shortciteA{truong2015optimal} extend the model developed by \shortciteA{patrick2008dynamic} to incorporate multiple surgical resources (ORs, recovery beds, etc.); however, the multi-priority setting is dropped in \shortciteA{truong2015optimal}.

In this paper, we extend the aforementioned extant research by taking into account time-dependent dynamic patient priority {scores}, the capacity constraints of ORs and the SICU, and the needs of multiple specialties and the different patient characteristics with which they are presented. To the best of our knowledge, this is the first time that these factors are collated into a comprehensive MDP model in order to handle realistically large-sized patient admission control problems.

\subsection{Markov decision process}
MDP is a discrete mathematical model used to facilitate sequential decision-making in the presence of uncertainty \shortcite{White2001}{. It} is the most commonly used mathematical model by existing research into patient admission control \shortcite<e.g., >{patrick2008dynamic,min2010elective,min2014managing,astaraky2015simulation,truong2015optimal}. In an MDP model, the status of the problem under consideration is described by the {state variables} from a set named state space{. The} possible actions that can be taken are included in another set {called the} action space. At every discrete time point (i.e., decision epoch), the agent observes the state of the problem and selects an action from the action space. The current state and the selected action determine the instant cost (or revenue) of the problem and the probability distribution of the subsequent state to which the problem transfers at the next decision epoch \shortcite{puterman1994markov}. The objective of an MDP model is to find the optimal policy that specifies a mapping from the state space to the action space and optimizes the value function (i.e., minimizes the expected total costs or maximizes the expected total revenues {over all} decision epochs) \shortcite{kolobov2012planning}.

Most algorithms used to solve {an MDP} are based on dynamic programming (DP), such as policy iteration (PI), value iteration (VI), real-time dynamic programming (RTDP) and its variants. VI and PI evaluate the entire state space iteratively, updating the value function until the improvement is lower than a given threshold. These algorithms employ brute-force search strategies and the computational resources (memory and time) they need are exponential to the problem size. As a result, solving realistically sized MDP models using VI or PI is usually intractable. In order to improve the computational efficiency, \shortciteA{barto1995learning} propose a RTDP algorithm that only evaluates a subset of the state space. At every decision epoch, RTDP explores the state space along several sampled trajectories that are rooted at the state that describes the problem's current status (i.e., the initial state). Specifically, at each visited state (including the initial state), RTDP updates the value function and policy then randomly samples the next state to visit according to the probability distribution determined by the current state and the corresponding action under the current policy. When a certain number of states are evaluated, RTDP returns to the initial state and repeats the above procedure. Since RTDP only visits those states that are reachable from the initial state, it outperforms PI and VI in terms of computational efficiency. However, the original version of RTDP lacks a mechanism for convergence detection and a proper criterion for terminating the computation; {hence}, it may waste computational resources on those states where the value function {has} already converged, or terminate so early that the value function is still far from {convergence}. To overcome these drawbacks, several extensions of RTDP have been developed in the literature. \shortciteA{bonet2003labeled} propose labelled-RTDP (LRTDP), which incorporates a labelling scheme into the original RTDP. They label the states where the value function no longer improves; in this way, the computational resources can be concentrated on the non-labelled states where the value function {has} not converged. \shortciteA{mcmahan2005bounded} and \shortciteA{Smith2006Focused} propose bounded-RTDP (BRTDP) and focused-RTDP (FRTDP), respectively. These extended RTDP algorithms compute both an upper bound and {a} lower bound on the optimal value function; the gap between the two bounds can then be used to judge whether the value function {has} converged and the exploration of the state space can be guided towards the poorly understood states (i.e., the states where the gap between the two bounds is large). Further to this, \shortciteA{Sanner2009VPI} improve BRTDP by introducing a value of perfect information (VPI) analysis to detect whether the policy at a state {has} converged. The VPI–RTDP algorithm they propose saves computational resources by not visiting the states where the value function {has not} converged but the policy {has} already converged.

Though the DP-based algorithms {have advanced} in different ways, their computational performance is highly dependent on the problem scale; i.e., the CPU time and memory consumed by the DP-based algorithms increase exponentially as the size of {the} MDP model increases. Specifically, \shortciteA{zhang2019markov} employ VI and the above-mentioned RTDP algorithms to solve patient admission control problems. Their experimental results show that VPI–RTDP outperforms VI significantly and that it is the best DP-based algorithm in terms of accuracy and efficiency. However, the problem size that VPI–RTDP can cope with is still fairly limited since the CPU time consumed by VPI–RTDP increases greatly as the size of {the} MDP model becomes slightly larger. Therefore, more advanced algorithms should be developed that can solve realistically sized patient admission control problems efficiently. 

In short, the challenges faced by DP-based algorithms are the so-called three curses of dimensionality: state space, action space, and outcome space (containing all the possible subsequent states of a state–action pair) \shortcite{powell2007approximate}. Regarding the large action spaces, many researchers addressing patient admission control perform structural analyses of their MDP models to reduce the number of actions that should be evaluated \shortcite <e.g.,>{min2010elective,min2014managing,truong2015optimal}. Moreover, approximate dynamic programming (ADP) provides a set of powerful algorithms that are able to cope with large state spaces and outcome spaces. Firstly and the most importantly, since the dimension of the value function depends on the cardinality of the state space, the computational complexity caused by a large state space can be reduced by estimating the high-dimensional value function with a low-dimensional approximator{. Secondly}, the enumeration of a large outcome space can be avoided by randomly sampling a subsequent state to evaluate. To the best of our knowledge, the application of ADP in the literature on patient admission control is very limited, although we note that \shortciteA{patrick2008dynamic} and \shortciteA{astaraky2015simulation} have developed a linear programming–based ADP algorithm and a simulation–based ADP algorithm, respectively.

In this paper, we formulate the patient admission control problem being studied as an MDP model with an infinite horizon. Recognizing that the problems involved in this paper are much larger than those addressed in the existing research \shortcite<e.g., >{patrick2008dynamic,min2010elective,min2014managing,astaraky2015simulation,truong2015optimal}, we reduce the computational complexity of solving the MDP model in two ways{. First}, we perform an in-depth analysis to investigate the optimal policy's properties, based on which we develop an efficient algorithm to simplify the exploration of the action space and to accelerate the solution approaches employed{. Second}, we develop an ADP algorithm based on RLS–TD($\lambda$), an efficient reinforcement learning method proposed by \shortciteA{xu2002efficient}. RLS–TD($\lambda$) {updates} a low-dimensional linear function to estimate the high-dimensional value function; {thus}, the RLS–TD($\lambda$)-based ADP algorithm is capable of providing near-optimal policies for large-scale patient admission control problems and consumes much less CPU time than do conventional DP-based algorithms. In the literature, a similar RLS–TD($\lambda$)-based ADP algorithm has been proposed by \shortciteA{yin2017recursive} to solve a real-time traffic signal control problem{. Their} experimental results show that the proposed ADP algorithm provides high-quality policies and is much more efficient than DP-based algorithms.}

\section{Problem description and formulation}\label{SecModel}

{This section describes the patient admission control problem in detail and presents the infinite-horizon MDP model}. We manage an OT that {provides surgical service for} elective patients from $J$ specialties. Every specialty $j$ has a relative importance factor denoted by $v_j$, and the patients in specialty $j$ are divided into $U_j$ urgency groups{, each associated with a maximum recommended waiting time $W_{ju}$.} {An MSS has been fixed at the tactical level and specifies the surgical blocks (i.e., OR-days) preallocated to each specialty, and the SICU is shared by all the $J$ specialties. Patients'} surgery durations and LOSs are subject to uncertainty, whereas the regular capacities of the ORs and SICU are fixed. Extra use of an OR leads to a cost of $c_o$ per hour, while a penalty of ${c_e}$ per bed/per day is incurred if the recovery beds in SICU are insufficient for the actual demand {(since some patients in need of intensive care have to be transferred to other recovery units with a lower level of care, or extra beds have to be added to SICU)}. {It should be mentioned that the actual LOSs of some patients can be 0, which means that not all the patients need intensive care after surgery.} 

{In addition to the hospital-related costs, we consider two types of patient-related costs in this paper: 1) surgery costs incurred by the booked surgeries which require surgical staff's work, utilization of surgical facilities, etc.; 2) waiting costs incurred by the {postponed} surgeries which cause patient dissatisfaction and may lead to deterioration in patients' health}. {We define} ${c_b}$ and ${c_d}$ {as} the unit surgery cost and the unit waiting cost, respectively. {In the literature, patient priority scores are commonly used as multipliers in the patient-related costs} {\shortcite<e.g., >{lamiri2008column,min2010elective,min2014managing,marques2017different,neyshabouri2017two,roshanaei2017propagating,roshanaei2020reformulation}. Hence, considering the dynamic patient priority {score} $Pr=v_juw$, we assume that every scheduled patient produces a surgery cost ${c_b}v_juw$, and that every unscheduled patient incurs a waiting cost ${c_d}v_juw$.} {The two types of time-dependent patient-related costs are both incorporated into the cost function, so that patients' waiting times can be minimized.} {It can be observed that scheduling a patient reduces the patient-related costs by $(c_d-c_b)v_juw$, which is proportionate to the relative importance factor $v_j$. Hence, the ORs preallocated to the specialties with higher $v_j$ are more likely to be overused, and when the patients of $J$ specialties compete for SICU recovery beds, priority is given to the specialties with higher $v_j$. We note that the priority score $P_r=v_juw$ may be simplistic from the perspective of practitioners, since it imposes a relationship among the three criteria and may not be sufficiently accurate. In fact, the prioritization systems adopted by hospitals tend to be more comprehensive and usually vary from one specialty to another \shortcite<e.g.,>{mullen2003prioritising,las2010can,montoya2014study,oliveira2020assessing}. As this paper focuses on mathematical optimization instead of prioritization strategies, we define $P_r$ in a simple format, which is well accepted by the research community (as discussed in Section \ref{SubSecPAC}), to keep our model concise.}

{From the existing research studying advance scheduling, we find that elective surgeries are usually planned over a fixed planning horizon (e.g., one week) \shortcite<e.g., >{fei2008solving,fei2009endoscopy,fei2009solving,fei2010planning,denton2010optimal,min2010scheduling,wang2014column,hashemi2016constraint,marques2017different,neyshabouri2017two,roshanaei2017collaborative,roshanaei2017propagating,moosavi2018scheduling,pang2018surgery}. In such studies, decisions to assign the patients on the waiting list to specific ORs are made at the beginning (or the end) of each horizon, with the dates of such assignations being within the present (or the next) horizon. When new elective patients arrive, the surgery schedule of the present horizon has already been determined{. As} a result, they are left unscheduled until the beginning of the next horizon. The patient admission control problem studied in this paper is a subproblem of advance scheduling; therefore, we assume that the newly arrived patients cannot be scheduled during the week of their arrival (i.e., direct admission is not allowed). The same assumption is also adopted by other research that addresses patient admission control, such as \shortciteA{min2010elective,min2014managing} and \shortciteA{zhang2019two}.}

At the end of each week, we update the patient waiting list by removing the treated patients and adding the newly arrived ones. We then determine which patients will be treated during {the week after the current week}. Let $n_{juw}$ be the total number of type-$\{juw\}$ patients (i.e., {specialty-$j$ patients with urgency coefficient $u$ and actual waiting time $w$}) on the waiting list; the state of the MDP {model} can then be defined as the {vector} of $n_{juw}$: $s=\{n_{juw}{:j=1,2,...,J;u=1,2,...,U_j;w=1,2,...,W_{ju}}\}$. Similarly, the action of the MDP {model} is defined as {vector} $a=\{m_{juw}{:j=1,2,...,J;u=1,2,...,U_j;w=1,2,...,W_{ju}}\}$, where $m_{juw}$ denotes the number of scheduled type-$\{juw\}$ patients. {In reality, it may be impossible to strictly comply with the maximum recommended waiting times, but \shortciteA{patrick2008dynamic} demonstrate that, in their patient admission control problem, a suboptimal policy can schedule the vast majority of patients earlier than their waiting-time targets. Hence, we assume in this paper that $W_{ju}$ are ``achievable'', i.e., they are determined by hospital managers according to the urgency of surgeries as well as the scarcity of surgical resources, such that it is not difficult to find good policies under which the proportion of the deferred patients ($w>W_{ju}$) is sufficiently small and can be ignored. This assumption may result in a departure from the real practice because the primary criterion for determining $W_{ju}$ is usually the urgency only \shortcite<e.g.,>{testi2008prioritizing,valente2009model}, but it can avoid incorporating into our MDP model the deferred patients with $w=W_{ju}+1,W_{ju}+2,...$, which lead to an infinite-dimensional state space. Under this assumption, all patients are supposed to be treated before their waiting times exceeding $W_{ju}$, and the deferred patients are thus excluded from any state $s$ and action $a$, as defined above.} Let $A(s)$ be the set of feasible actions for {state} $s$, then, for any $a\in A(s)$:
\begin{equation}\label{FeasibleAction}
\left\{
\begin{array}{ll}
m_{juw}\leqslant n_{juw} & \text{if } w<W_{ju}\\
m_{juw}=n_{juw} & \text{if } w=W_{ju}\\
\end{array}
\right.
\end{equation}

Let $\tilde{n}_{ju}^\tau$ be the number of newly arrived { specialty-$j$} patients with urgency {coefficient} $u$ during week $\tau$; the transition of state from week $\tau$ to week $\tau+1$ can then be written as
\begin{equation}
\left\{
\begin{array}{ll}
n^{\tau+1}_{ju,1}=\tilde{n}_{ju}^\tau & \forall\ j,\ u\\
n^{\tau+1}_{ju,w+1}=n^\tau_{juw}-m^\tau_{juw} & \forall\ j,\ u,\ w<W_{iu}\\
\end{array}
\right.
\end{equation}
Then{,} we can obtain the transition probability function of {the MDP model} as follows:
\begin{equation}\label{TransDef}
p(s_\tau,a_\tau,s_{\tau+1})=\prod_{j=1}^{J}\prod_{u=1}^{U_j}\left[p(\tilde{n}_{ju}^\tau=n^{\tau+1}_{ju,1})\prod_{w=1}^{W_{ju}}p(n^\tau_{juw}-m^\tau_{juw}=n^{\tau+1}_{ju{,w+1}})\right]
\end{equation}
{In (\ref{TransDef}), $p(\tilde{n}_{ju}^\tau=n^{\tau+1}_{ju,1})\in(0,1)$ is the probability that the number of newly arrived type-$\{ju,0\}$ patients in week $\tau$ is equal to the total number of type-$\{ju,1\}$ patients in state $s_{\tau+1}$. Its value is computed by the probability mass function of a Poisson distribution with an arrival rate $\bar{n}_{ju}$. $p(n^\tau_{juw}-m^\tau_{juw}=n^{\tau+1}_{ju{,w+1}})\in\{0,1\}$ is a deterministic term that indicates if the number of unscheduled type-$\{juw\}$ patients in week $\tau$ matches the total number of type-$\{ju,w+1\}$ patients in state $s_{\tau+1}$. Its value can either be 0 or 1.}

{In a patient admission control problem, we cannot compute the exact utilizations of surgical resources, which depend on the intra-week schedule (i.e., patient-to-OR assignments) and the realizations of stochastic parameters. It is straightforward to estimate the expected resource utilizations using the total capacities of ORs and SICU as well as the distributional information of stochastic parameters \shortcite<e.g.,>{astaraky2015simulation,gerchak1996reservation,gocgun2012lagrangian,huh2013multiresource,jebali2017chance,lamiri2008stochastic,min2014managing,truong2015optimal}. Even so, for the problem studied in this paper, computing the exact expectations of OR overtime and SICU bed shortage is a considerable challenge, because this requires computing the expectations of $|J|+1$ truncated probability distributions (lognormal distributions are commonly adopted) for each state–action pair. In order to reduce the computational complexity, we approximate the expected OR overtime and SICU bed shortage by taking the sum of the expected surgery durations and LOSs of the scheduled patients, subtracting the available OR and SICU capacity, and taking the positive parts. {This estimation method is adopted in many relevant studies} \shortcite<e.g,>{astaraky2015simulation,molina2015integrated,molina2015new,wang2016discrete,roshanaei2017collaborative,roshanaei2017propagating,roshanaei2020reformulation} {and its effectiveness in estimating the over-utilization of ORs has been validated:} \shortciteA{molina2015integrated} show that the stochasticity of surgery durations does not influence the average OR utilization; \shortciteA{wang2016discrete} reveal that the under- and over-estimations of surgery durations are mostly counterbalanced, and that stochasticity has little influence on the quality of their deterministically optimized surgery schedules; \shortciteA{roshanaei2017collaborative,roshanaei2020reformulation} anticipate that ignoring uncertainty does not drastically diminish the cost-savings achieved by their deterministic optimization approach. {However, we admit that estimating the SICU utilization in the same way may be inaccurate, mainly because we assume that the demand for intensive care is smoothed out, but in reality, the postoperative patients are likely to be clumped due to the restrictions imposed by the MSS, rather than being evenly distributed throughout the week. In addition, we need to assume that all patients currently in the SICU will finish their stays before the beginning of the next week. This assumption is relatively reasonable because most patients spend no more than two days in the SICU (see the test instances of \shortciteA{min2010scheduling} and \shortciteA{neyshabouri2017two}).} In the numerical experiments of this work, we compute the overtime of ORs and the excess of SICU capacity using large numbers of randomly sampled surgery durations and LOSs, thus the performance and robustness of our policies under uncertainty can be effectively evaluated.} 

{Then, the cost function of the MDP model for state–action pair $(s,a)$ is defined as}
\begin{equation}\label{CostDef}
\begin{split}
C(s,a)=&{c_b}\sum_{j=1}^{J}\sum_{u=1}^{U_j}\sum_{w=1}^{W_{ju}}v_juwm_{juw}+{c_d}\sum_{j=1}^{J}\sum_{u=1}^{U_j}\sum_{w=1}^{W_{ju}}v_juw(n_{juw}-m_{juw})\\
&+c_o\sum_{j=1}^{J}\left(\sum_{u=1}^{U_j}\sum_{w=1}^{W_{ju}}m_{juw}\bar{d}_{j}-\rho_1B_j\right)^++{c_e}\left(\sum_{j=1}^{J}\sum_{u=1}^{U_j}\sum_{w=1}^{W_{ju}}m_{juw}\bar{l}_{j}-\rho_2R\right)^+\\
\end{split}
\end{equation}
where $(\pmb\cdot)^+=\max\{\pmb\cdot,0\}$; $\bar{d}_{j}$ and $\bar{l}_{j}$ are the expectations of surgery duration and LOS for {specialty-$j$} patients {(we assume that patients from the same specialty have the same lognormal distributions for surgery duration and LOS)}; $B_j$ is the regular OR capacity reserved for specialty $j$; $R$ is the regular capacity of SICU; $\rho_1,\rho_2\in(0,1]$ are the estimated availability rates of the ORs and recovery beds, respectively. $\rho_1$ and $\rho_2$ are incorporated into the cost function, since the intra-week scheduling cannot always fully utilize the surgical resources. 

{It can be observed that cost function (\ref{CostDef}) is independent of week $\tau$. Moreover, since the number and type of elective surgeries do not change too much over weeks \shortcite{guerriero2011operational}, we assume that the arrival rates $\bar{n}_{ju}$ of new patients are constant. Thus, the transition probability function (\ref{TransDef}) is also independent of $\tau$ and the MDP model is stationary \shortcite{kolobov2012planning}. In order to capture more dynamics of the problem, one might want to shorten the time interval between two consecutive decision epochs from one week to one day. However, making decisions on a daily basis requires us to consider the two following issues. First, elective surgeries cannot be scheduled at weekends{, implying that} making the same decision on different weekdays leads to different expected waiting times of patients and different costs-to-go. Second, the regular OR capacity preallocated to each specialty by an MSS usually varies from Monday to Friday. {Consequently}, the resource constraints and cost function are no longer independent of $\tau$, and the state space of the non-stationary MDP model will be five times larger than that of the stationary one.} {Given that the planning horizons of most advance surgery scheduling problems are one week \shortcite<e.g.,>{min2010scheduling,jebali2015stochastic,jebali2017chance,marques2017different,neyshabouri2017two,moosavi2018scheduling,molina2018stochastic,zhu2019operating,roshanaei2020reformulation}, our MDP model for patient admission control (a subproblem of advance surgery scheduling) is formulated on a weekly basis and is compatible with longer planning periods that are adopted by some hospitals: the decision epochs $\tau$ can be changed from weekly to bi-weekly or monthly without the need for changing the structure of the MDP model.}

{Figure \ref{FigEg} demonstrates the evolution of the MDP model for an example patient admission control problem with two specialties ($j=1,2$) {and} two urgency {groups} ($u=1,2$) in each specialty.} 

\begin{figure}[!htb]
	\centering
	\includegraphics[width=6in]{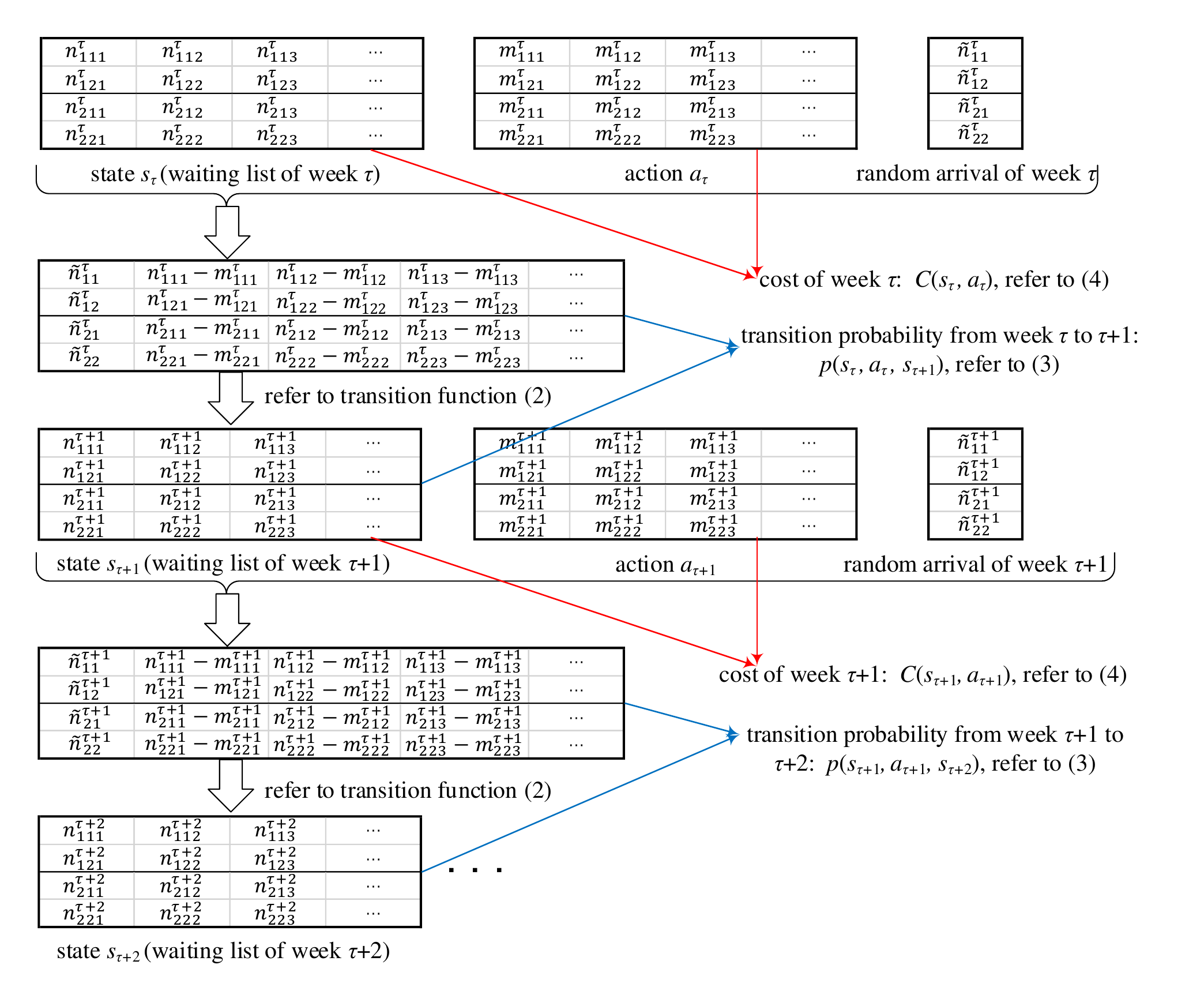}
	\vspace{-20pt}
	\caption{ An example MDP model for a two-specialty, two-urgency-{group} problem}\label{FigEg}
\end{figure}
 
Finally, the objective of the MDP model is to find the optimal policy $\pi^*$ that minimizes the discounted expected costs over the infinite horizon:
\begin{equation}
\pi^*=\mathop{\arg\min}_{\pi}\mathbb{E}\left\{\sum_{\tau=1}^{+\infty}\gamma^{\tau-1}C[s_\tau,\pi(s_\tau)]\right\}
\end{equation}
where $\gamma\in[0,1)$ is the discount factor and $\pi(s_\tau)$ is the action corresponding to $s_\tau$ under policy $\pi$.

\section{Structural analysis}\label{SecAna}
{The MDP model presented in Section \ref{SecModel} can be regarded as a substantial extension of the MDP models formulated by \shortciteA{gerchak1996reservation} and \shortciteA{min2010elective,min2014managing}. In these previous works, the authors explore the properties of their MDP models to facilitate the solution procedures. In this paper, as we take time-dependent patient priority scores, multiple resource constraints (ORs and SICU), and multiple specialties with different patient characteristics into consideration, our MDP model is significantly more complex and most of the model properties proved by \shortciteA{gerchak1996reservation} and \shortciteA{min2010elective,min2014managing} do not apply to our MDP model. Hence, we investigate the structural properties of our MDP model in this section, and these properties will help us to improve computational efficiency when evaluating the action space.} 

We first create a vector $\Delta_{j'u'w'}=\{\delta_{juw}\}\in S\cup A$ ($S$ denotes the state space and $A$ denotes the action space) in which $\delta_{j'u'w'}=1$ and the other elements are 0. Then, two partial orders on $S\cup A$ are defined as follows:
\begin{definition}\label{poset}
	Let $\textsc{x}=\{x_{juw}\}\in S\cup A$ and $\textsc{y}=\{y_{juw}\}\in S\cup A$.
	\begin{itemize}
		\item[\textup{(i)}] If $\forall j,u,w$: $x_{juw}\leqslant y_{juw}$ and $\exists j',u',w'$ s.t. $x_{j'u'w'}<y_{j'u'w'}$, then $\textsc{x}<\textsc{y}$.
		\item[\textup{(ii)}] If $\forall j,u,w$: $x_{juw}=y_{juw}$, then $\textsc{x}=\textsc{y}$.
		\item[\textup{(iii)}] Otherwise, $\textsc{x}$ and $\textsc{y}$ are incomparable.
	\end{itemize}
\end{definition}
\begin{definition}\label{pp}
	Let $\textsc{x}=\{x_{juw}\}\in S\cup A$ and $\textsc{y}=\{y_{juw}\}\in S\cup A$. The patient priority {scores} of $\textsc{x}$ and $\textsc{y}$ are denoted by $P(\textsc{x})$ and $P(\textsc{y})$, respectively.
	\begin{itemize}
	\item[\textup{(i)}] If $\textsc{x}=\textsc{y}+\Delta_{j'u'w'}-\Delta_{j'u''w''}$ and $u'w'<u''w''$, then $P(\textsc{x})<P(\textsc{y})$.
	\item[\textup{(ii)}] If $\textsc{x}=\textsc{y}$, then $P(\textsc{x})=P(\textsc{y})$.
	\item[\textup{(iii)}] Otherwise, $P(\textsc{x})$ and $P(\textsc{y})$ are incomparable.
	\end{itemize}
\end{definition}
Next, we propose Lemma \ref{Lemmamin} that will be used in the proofs of structural properties.
\begin{lemma}\label{Lemmamin}
	For functions $f,g:D\rightarrow\mathbb{R}$ with $D\subseteq\mathbb{R}$\textup{:} $\min f(x)-\min g(x)\geqslant\min[f(x)-g(x)]$.
\end{lemma} 
\begin{proof}
	See Appendix \ref{ProofLemmamin}.
\end{proof}
{Let $C_0(s)=\min_{a\in A(s)}C(s,a)$ be the optimal single-period cost and $a_0(s)=\mathop{\arg\min}_{a\in A(s)}C(s,a)$ be the optimal single-period action}. For the sake of simplicity, we use $C_p(s,a)$ to denote the patient-related cost, i.e., the sum of the first term and the second term of (\ref{CostDef}), and use $C_h(a)$ to denote the hospital-related cost, i.e., the sum of the third term and the fourth term of (\ref{CostDef}). Using Definitions \ref{poset}, \ref{pp}, and Lemma \ref{Lemmamin}, the following statements can be proved to be true:
\begin{proposition}\label{ConvexC}
	\begin{itemize}
		\item[\textup{(i)}] $C_0(s)$ is increasing in $s$.
		\item[\textup{(ii)}] If $P(a)$ and $P(a')$ are comparable, then $C_h(a)=C_h(a')$ holds.
		\item[\textup{(iii)}] $C_0(s)$ is increasing in $P(s)$.
	\end{itemize}
\end{proposition}
\begin{proof}
	See Appendix \ref{ProofConvexC}.
\end{proof}

{ From (i) and (iii) of Proposition \ref{ConvexC}, we know that the optimal single-period cost $C_0(s)$ is monotonically increasing in $n_{juw}$ and the {priority scores} of the patients on the waiting list. In addition, (ii) of Proposition \ref{ConvexC} indicates that the expected hospital-related costs depend on the number of {scheduled} patients in each specialty and are independent of patients' urgency {coefficients} and waiting times. With Proposition \ref{ConvexC}, we can further analyze the optimal value function (i.e., the value function under the optimal policy $\pi^*$) which is given by}
\begin{equation}\label{DefV}
	V^{\pi^*}(s)=\mathbb{E}\left\{\sum_{n=1}^{+\infty}\gamma^{n-1}C[s_n,\pi^*(s_n)]\right\}=\sum_{n=1}^{+\infty}\gamma^{n-1}\mathbb{E}\{C[s_n,\pi^*(s_n)]\}
\end{equation}
where $s_1=s$. Since $C[s_n,\pi^*(s_n)]$ is finite and $\gamma<1$, then $\gamma^{n-1}\mathbb{E}\{C[s_n,\pi^*(s_n)]\}\rightarrow0$ as $n$ goes to infinity. We can thereby assume $\gamma^{n-1}\mathbb{E}\{C[s_n,\pi^*(s_n)]\}=0$ for any $n>N\gg0$, then
\begin{equation}\label{EquaV}
	V_n^{\pi^*}(s)=\left\{
		\begin{array}{ll}
			C[s,\pi^*(s)]+\gamma\sum\limits_{s'\in S}p[s,\pi^*(s),s']V_{n+1}^{\pi^*}(s')&\text{if } n=1,2,...,N\\
			0&\text{if } n=N+1,N+2,...\\
		\end{array}
	\right.
\end{equation}
where $V_1^{\pi^*}(s)=V^{\pi^*}(s)$. Further, the Q-value of state-action pair $(s,a)$ {under $\pi^*$} is defined as 
\begin{equation}
	Q^{\pi^*}(s,a)=C(s,a)+\gamma\sum_{s'\in S}p(s,a,s')V^{\pi^*}(s')
\end{equation}
$Q^{\pi^*}(s,a)$ can also be written as the following recursive formula:
\begin{equation}\label{EquaQ}
Q^{\pi^*}_n(s,a)=C(s,a)+\gamma\sum\limits_{s'\in S}p(s,a,s')V_{n+1}^{\pi^*}(s')
\end{equation}
where $Q^{\pi^*}_1(s,a)=Q^{\pi^*}(s,a)$ and $V_{n+1}^{\pi^*}(s')=\min_{a\in A(s')}Q^{\pi^*}_{n+1}(s',a)$. 
 
Before exploring the properties of $V^{\pi^*}(s)$, we present the following notations and {analysis}, which will be useful when proving our propositions. Let {vector} $G^s_a=\{g_{juw}\}$ be the post-action state of $(s,a)$:
\begin{equation}
\left\{
\begin{array}{ll}
g_{ju,1}=0,&\forall {\ j,}\ \ u=1,2,...,U_j\\
g_{ju,w+1}=n_{juw}-m_{juw},&\forall {\ j,}\ \ u=1,2,...,U_j,\ \ w=1,2,...,W_{ju}-1\\
\end{array}
\right.
\end{equation}
and $\Psi=\{\psi_{juw}\}$ be the {vector} of newly arrived patients:
\begin{equation}
\left\{
\begin{array}{ll}
\psi_{ju,1}=\tilde{n}_{ju},&\forall {\ j,}\ \ u=1,2,...,U_j\\
\psi_{ju,w+1}=0,&\forall {\ j,}\ \ u=1,2,...,U_j,\quad w=1,2,...,W_{ju}-1\\
\end{array}
\right.
\end{equation}
Then, the subsequent state can be written as $s'=G^s_a+\Psi$. To keep the notations simple, we introduce $P_{\Psi}$ in the following equation: 
\begin{equation}
p(s,a,s')=\left\{
\begin{array}{ll}
\prod\limits_{u=1}^{U_j}p(\tilde{n}_{ju}=\psi_{ju,1})=P_{\Psi}, &\text{if }s'=G^{s}_a+\Psi\\
0,&\text{if }s'\neq G^{s}_a+\Psi\\
\end{array}
\right.
\end{equation}
Then, we have
\begin{equation}
\sum_{s'\in S}p(s,a,s')V^{\pi^*}(s')=\sum_{\Psi=\pmb{0}}^{+\infty}P_{\Psi}V^{\pi^*}(G^s_a+\Psi)
\end{equation}\par 
With these notations and analysis, we propose the following properties of $V^{\pi^*}(s)$:
\begin{proposition}\label{V}
	\begin{itemize}
		\item [\textup{(i)}] $V^{\pi^*}(s)$ is increasing in $s$.
		\item [\textup{(ii)}] $V^{\pi^*}(s)$ is increasing in $P(s)$.
	\end{itemize}
\end{proposition}
\begin{proof}
	See Appendix \ref{ProofV}.
\end{proof}
{ Proposition \ref{V} indicates that the properties of the single-period optimal cost $C_0(s)$ proposed in (i) and (iii) of Proposition \ref{ConvexC} also hold for the optimal value function $V^{\pi^*}(s)$.} Considering that $V^{\pi^*}(s)$ is monotonically increasing in $s$ and $P(s)$, we can prove the following properties of $\pi^*(s)$ and $Q^{\pi^*}(s,a)$ that enable us to simplify the exploration of the action space.
\begin{proposition}\label{pi}
	For any patient type $\{juw\}$: if $({c_d}-{c_b})v_{j}uw>c_o\bar{d}_{j}+{c_e}\bar{l}_{j}$ holds, then all the type-$\{juw\}$ patients are scheduled by $\pi^*(s)$.
\end{proposition}
\begin{proof}
	See Appendix \ref{Proofpi}.
\end{proof}
\begin{proposition}\label{PropVQ}
$\forall s\in S$: $Q^{\pi^*}(s,a)$ is decreasing in $P(a)$.
\end{proposition}
\begin{proof}
	See Appendix \ref{ProofPropVQ}.
\end{proof}
{Proposition \ref{pi} suggests that, if the reduction $(c_d-c_b)v_{j}uw$ in the patient-related costs induced by scheduling a type-$\{juw\}$ patient is greater than the maximum possible increase $(c_o\bar{d}_{j}+c_e\bar{l}_{j})$ in the hospital-related costs, all the type-$\{juw\}$ patients should be scheduled in the current week without further postponement. Hence, the best action $\pi^*(s)$ under the optimal policy $\pi^*$ should hold the following property:
\begin{remark}\label{Remark1}
	For any patient type $\{juw\}$: if $(c_d-c_b)v_{j}uw>c_o\bar{d}_{j}+c_e\bar{l}_{j}$, then $m_{juw}=n_{juw}$ holds in $\pi^*(s)$.
\end{remark}
Proposition \ref{PropVQ} reveals another structural property that the best action $\pi^*(s)$ should hold: in every week and every specialty, the priority score of any scheduled patient should not be lower than that of any unscheduled patient. More specifically, if $\forall j\in\{1,...,J\}$: the total numbers of specialty-$j$ patients scheduled by actions $a$ and $\pi^*(s)$ are identical, i.e., $P(a)$ and $P(\pi^*(s))$ are comparable (as defined in Definition \ref{pp}), then $P(a)\leqslant P(\pi^*(s))$.
\begin{remark}\label{Remark2}
	$\forall s\in S$: there exists no action $a\in A(s)$ s.t. $P(a)>P(\pi^*(s))$.
\end{remark}
By excluding the actions that do not hold the properties summarized in Remarks \ref{Remark1} and \ref{Remark2}, $\pi^*(s)$ can be searched efficiently in a subset $A^*(s)$ of the full action set $A(s)$ ($A^*(s)\subseteq A(s)$). As $\|A^*(s)\|<\|A(s)\|$ holds for most states in our MDP model, the efficiency of exploring the action space can be significantly improved. The procedure of determining $A^*(s)$ is provided in Algorithm \ref{GreedyAction} and can be integrated into most MDP solution approaches, including DP-based exact algorithms and simulation-based ADP algorithms. }

\begin{algorithm}[!htb]
	\caption{Procedure of determining $A^*(s)$}\label{GreedyAction}
	\KwIn{state $s=\{n_{juw}\}$}
	\KwOut{action set $A^*(s)$}
	Initialize $a'=\{m'_{juw}\}=\pmb{0}$;\\
	\For{$j=1,2,...,J$}
	{
		\For{$u=1,2,...,U_j$}
		{
			\For{$w=1,2,...,W_{ju}$}
			{
				\lIf{$(({c_d}-{c_b})v_juw>c_o\bar{d}_j+{c_e}\bar{l}_j)\vee (w=W_{u})$}{$m'_{juw}=n_{juw}$}
			}
		}
	}
	Let $A^*(s)=\{a'\}$ and $I=\|A^*(s)\|=1$;\quad//\textit{$I$ is the size of $A^*(s)$}\\
	\For{$j=1,2,...,J$}
	{
		$N_j=\sum_{u=1}^{U_j}\sum_{w=1}^{W_{ju}}(n_{juw}-m'_{juw})$;\quad//\textit{number of specialty-$j$ patients with $({c_d}-{c_b})v_juw\leqslant c_o\bar{d}_j+{c_e}\bar{l}_j$}\\
		$M_j=0$;\quad//\textit{number of scheduled specialty-$j$ patients with $({c_d}-{c_b})v_juw\leqslant c_o\bar{d}_j+{c_e}\bar{l}_j$}\\
		$I=I\times (N_j+1)$;\quad//\textit{compute the size of $A^*(s)$}\\
	}
	\For{$i=1,2,...,I$}
	{
		$a_i=\{m_{juw}\}=a'$;\\
		//\textit{exhaust all the possible combinations of $M_j$}\\
		\For{$j=1,2,...,J$}
		{
			\If{$M_j<N_j$}
			{
				$M_j=M_j+1$;\\
				\lForEach{\textup{integer} $k\in[1,j-1]$}{$M_k=0$}
				break;
			}
		}
		//\textit{for each combination of $M_j$, find the action with the maximum value of $P(a)$}\\
		\For{$j=1,2,...,J$}
		{
			$\{u',w'\}=\arg\max_{(({c_d}-{c_b})v_juw\leqslant c_o\bar{d}_j+{c_e}\bar{l}_j)\wedge(w\neq W_u)} v_juw$;\\
			$x=M_j$;\\
			\While{true}
			{
				\uIf{$n_{ju'w'}\geqslant x$}
				{
					$m_{ju'w'}=x$;\\
					break;
				}
				\Else
				{
					$m_{ju'w'}=n_{ju'w'}$;\\
					$x=x-n_{ju'w'}$;\\
					$\{u',w'\}=\arg\max_{(uw\leqslant u'w')\wedge(w\neq W_u)\wedge((u\neq u')\vee(w\neq w'))} v_juw$;
				}
			}
		}
		$A^*(s)=A^*(s)\cup\{a_i\}$;\\
	}
\end{algorithm}

\section{Solution approach}\label{SecSol}
{The algorithm proposed in the previous section (Algorithm \ref{GreedyAction}) can reduce the number of actions to be evaluated for each state, thus facilitating the exploration of the action space. However, it cannot be independently employed to solve the MDP model proposed in Section \ref{SecModel}. As an acceleration technique, Algorithm \ref{GreedyAction} should always be integrated into a solution approach of MDP, such as PI, VI, and RTDP. For the patient admission control problem studied in this paper, there are $\sum_{j=1}^J\sum_{u=1}^{U_j}W_{ju}$ different patient types, and the dimension of the MDP model's state space is increasing in $\sum_{j=1}^J\sum_{u=1}^{U_j}W_{ju}$ at an exponential rate. Moreover, as there can be large numbers of $\sum_{j=1}^JU_j$ different types of newly arrived patients in each week, the number of possible subsequent states for any state–action pair is enormous, i.e., the outcome space is huge as well (exponentially increasing in $\sum_{j=1}^JU_j$). Therefore, even with the acceleration of Algorithm \ref{GreedyAction}, solving the MDP model using traditional DP-based algorithms is still computationally intractable. To cope with the curses of dimensionality in the state space and the outcome space, in this section, we propose an ADP algorithm based on reinforcement learning. 
	
The basic idea of the ADP algorithm is to reduce the model dimensionality by approximating the value function. It is well known that value function approximation can either be parametric or non-parametric: the former adopts a low-dimensional linear function to approximate the true optimal value function $V^{\pi^*}(s)$, while the latter does not rely on functional structures but directly approximates the values of observed instances \shortcite{ulmer2020meso}. A comprehensive overview of the two value function approximation methods is provided by \shortciteA{powell2007approximate}. Non-parametric methods often need to aggregate the state space and use lookup tables which may suffer from the curse of dimensionality \shortcite{powell2015tutorial,voelkel2020aggregation}. In comparison, parametric methods are easier to compute and more suitable for online computation \shortcite{powell2015tutorial}. In the ADP algorithm proposed in this paper, we adopt a parametric method that approximates the true optimal value function $V^{\pi^*}(s)$ using a classical feature-based linear function:
\begin{equation}\label{approx}
\hat{V}(s,\Theta)=\Phi^T(s)\Theta=\sum_{\xi=1}^{\Xi}\phi_\xi(s)\theta_\xi\approx V^{\pi^*}(s)
\end{equation}
In (\ref{approx}), $\phi_\xi(s)$ captures the $\xi$th feature of state $s$ and each state $s\in S$ has a corresponding feature vector $\Phi(s)=[\phi_1(s),...,\phi_\xi(s),...,\phi_\Xi(s)]^T$. For our MDP model, it is straightforward to define that each component $n_{juw}$ of state $s$ corresponds to a feature $\phi_\xi(s)$, thus the length of feature vector $\Phi(s)$ is $\Xi=\sum_{j=1}^{J}\sum_{u=1}^{U_j}W_{ju}$. $\Theta=[\theta_1,...,\theta_\xi,...,\theta_\Xi]^T$ in (\ref{approx}) is a parameter vector of length $\Xi$ and each component $\theta_\xi$ is a tunable parameter. We want to compute the optimal parameter vector $\Theta^*$ which minimizes the gaps between $\hat{V}(s,\Theta)$ and $V^{\pi^*}(s)$:
\begin{equation}
\Theta^*=\mathop{\arg\min}_{\Theta\in\mathbb{R}^\Xi}\sum_{s\in S}\left[\hat{V}(s,\Theta)-V^{\pi^*}(s)\right]^2
\end{equation}
With $\Theta^*$, we can compute a near-optimal policy based on the approximate value function $\hat{V}(s,\Theta^*)$. As the computation of the complex value function $V^{\pi^*}(s)$ in the high-dimensional state space $S$ is replaced by the computation of the low-dimensional parameter vector $\Theta^*$, the curse of dimensionality in the state space can be drastically alleviated.

Many reinforcement learning methods can be used to compute $\Theta^*$ (by iteratively updating $\Theta$), such as temporal difference learning \shortcite<TD($\lambda$):>{sutton1988learning,tsitsiklis1997analysis}, least-squares temporal difference learning \shortcite<LS-TD($\lambda$):>{barto1995learning,boyan2002technical}, and recursive least-squares temporal difference learning \shortcite<RLS-TD($\lambda$):>{barto1995learning,xu2002efficient}. In order to combine reinforcement learning with ADP, we adopt an on-policy learning strategy introduced by \shortciteA{powell2007approximate}, under which the parameter vector $\Theta$ and the policy based on it are updated simultaneously and all the computation is performed online. 

Before making any decision, we first initialize $\Theta$ arbitrarily, e.g., $\Theta_0=\pmb0$. Then, at each decision epoch (week) $\tau$, we randomly sample the stochastic information (i.e., patient arrivals) of the next $N$ epochs. At each simulated epoch $n$, we select action $a_n$ for the current state $s_n$ (note that $s_1=s_\tau$) according to the approximate value function with the latest parameter vector $\Theta_{n-1}$:
\begin{equation}\label{EqOnpolicySimu}
	a_n=\mathop{\arg\min}_{a\in A^*(s_n)}\left[C(s_n,a)+\gamma\Phi^T(G^{s_n}_{a}+\Psi_n^a)\Theta_{n-1}\right]
\end{equation}
where $G^{s_n}_a$ is the post-action state of state-action pair $(s_n,a)$; $\Psi_n^a$ is a vector of the numbers of new patients (defined in Section \ref{SecAna}) and it is independently and identically sampled from Poisson distributions for each action $a$ and each epoch $n$. Then, we update $\Theta_{n-1}$ to $\Theta_n$ and compute the next state to be visited $s_{n+1}=G^{s_n}_{a_n}+\Psi_n^{a_n}$. In this way, we can obtain a trajectory $(s_1,a_1,\Psi_1^{a_1},s_2,a_2,\Psi_2^{a_2},...,s_N,a_N,\Psi_N^{a_N})$ with $s_1=s_\tau$.

The process described above can be repeated multiple times for the same state $s_\tau$, until some termination criterion is satisfied. That is, at each real decision epoch $\tau$, we generate $M\geqslant1$ trajectories with independently and identically sampled vectors $\Psi$. Thus, parameter vector $\Theta$ and the policy based on it are updated $M\times N$ times before we determine the action to be executed in week $\tau$:
\begin{equation}\label{EqOnpolicyReal}
a_\tau=\mathop{\arg\min}_{a\in A^*(s_\tau)}\left[C(s_\tau,a)+\gamma\Phi^T(G^{s_\tau}_{a}+\bar{\Psi})\Theta_\tau\right]
\end{equation}
where vector $\bar{\Psi}$ consists of the expected arrival numbers $\bar{n}_{ju}$ of $\sum_{j=1}^JU_j$ patient types, and $\Theta_\tau$ is the latest updated parameter vector. From (\ref{EqOnpolicySimu}) and (\ref{EqOnpolicyReal}), it can be seen that all the actions we take in simulations ($a_n$) and in the real world ($a_\tau$) are selected according to the approximate value function $\hat{V}(s,\Theta)$. This is the so-called on-policy learning strategy. An advantage of this strategy is that we explore the outcome space along randomly sampled trajectories, rather than evaluating all the possible subsequent states of each state-action pair. In this way, the curse of dimensionality in the outcome space can be effectively solved.

Next, we briefly present the reinforcement learning method integrated into the ADP algorithm for updating $\Theta$. After action $a_n$ is computed by (\ref{EqOnpolicySimu}) at simulated epoch $n$, we compute the gap between two successive value function approximations, i.e., the temporal difference or TD error:
\begin{equation}\label{TDerror}
e_n=C(s_n,a_n)+\gamma\hat{V}(s_{n+1},\Theta_{n-1})-\hat{V}(s_n,\Theta_{n-1})\\
=C(s_n,a_n)-[\Phi(s_n)-\gamma\Phi(s_{n+1})]^T\Theta_{n-1}
\end{equation}

The TD($\lambda$) method introduced by \shortciteA{sutton1988learning} updates parameter vector $\Theta$ using gradient descent:
\begin{equation}\label{GradientDescent}
\Theta_n=\Theta_{n-1}+\eta_ne_n\sum_{k=1}^{n}(\gamma\lambda)^{n-k}\nabla\hat{V}(s_k,\Theta_{n-1})
\end{equation}
where $\eta_n$ is a sequence of step-size parameters, $\lambda\in[0,1]$ is a decaying factor, and gradient $\nabla\hat{V}(s_k,\Theta_{n-1})=\Phi(s_k)$ is the vector of partial derivatives with respect to the components of $\Theta_{n-1}$. Equation (\ref{GradientDescent}) can be simplified by defining $z_n$ as the eligibility trace which implicitly stores the history of a trajectory  \shortcite{tsitsiklis1997analysis}:
\begin{equation}\label{eligibility}
z_n=\sum_{k=1}^{n}(\gamma\lambda)^{n-k}\nabla\hat{V}(s_k,\Theta_{n-1})=\sum_{k=1}^{n}(\gamma\lambda)^{n-k}\Phi(s_k)
\end{equation}
Thus, TD($\lambda$) can be represented by the following recursive equations:
\begin{equation}\label{TDlambda}
\left\{\begin{split}
&z_n=\Phi(s_n)+\gamma\lambda z_{n-1}\\
&\Theta_{n}=\Theta_{n-1}+\eta_ne_nz_n
\end{split}\right.
\end{equation}
where $z$ is often arbitrarily initialized, e.g., $z_0=\pmb0$. The convergence with probability one of TD($\lambda$) has been established by \shortciteA{tsitsiklis1997analysis}. However, TD($\lambda$) is particularly impacted by the step-size parameters $\eta_n$, as a poor choice of $\eta_n$ can drastically slow down the convergence \shortcite{barto1995learning}. To eliminate $\eta_n$, \shortciteA{boyan2002technical} extends TD($\lambda$) to LS-TD($\lambda$) based on the fact that the changes of $\Theta$ made by TD($\lambda$) are in the following form:
\begin{equation}\label{LS-TD}
\Theta_n=\left\{\sum_{k=1}^nz_k[\Phi(s_k)-\gamma\Phi(s_{k+1})]^T\right\}^{-1}\left[\sum_{k=1}^nC(s_k,a_k)z_k\right]=\mathcal{A}_n^{-1}b_n
\end{equation}
It can be seen from (\ref{LS-TD}) that step-size parameters $\eta_n$ are eliminated, but the inverse of a $\Xi\times\Xi$ matrix $\mathcal{A}_n$ has to be computed at each time epoch, hence the computational complexity of LS-TD($\lambda$) is $O(\Xi^3)$. 

Further, \shortciteA{xu2002efficient} improve LS-TD($\lambda$) using the recursive least-squares (RLS) method of \shortciteA{ljung1983theory}. Since the RLS-TD($\lambda$) method proposed by \shortciteA{xu2002efficient} has a lower computational complexity $O(\Xi^2)$, we employ RLS-TD($\lambda$) to update $\Theta$ in this paper. To facilitate the computation, RLS-TD($\lambda$) defines a $\Xi\times \Xi$ variance matrix $\mathcal{P}_n=\mathcal{A}_n^{-1}=\left\{\mathcal{A}_{n-1}+z_n[\Phi(s_n)-\gamma\Phi(s_{n+1})]^T\right\}^{-1}$. By Lemma 2.1 of \shortciteA{ljung1983theory}, $\mathcal{P}_n$ can be written in the following recursive form: 
\begin{equation}\label{VarianceMatrix}
\mathcal{P}_n=\mathcal{P}_{n-1}-\mathcal{P}_{n-1}z_n\left\{1+[\Phi(s_n)-\gamma\Phi(s_{n+1})]^T\mathcal{P}_{n-1}z_n\right\}^{-1}[\Phi(s_n)-\gamma\Phi(s_{n+1})]^T\mathcal{P}_{n-1}\\
\end{equation}
Then, using the RLS method with (\ref{LS-TD}) and (\ref{VarianceMatrix}), $\Theta$ can be updated in the following way:
\begin{equation}\label{ThetaRLSTD}
\Theta_n=\Theta_{n-1}+\frac{\mathcal{P}_{n-1}z_ne_n}{1+[\Phi(s_n)-\gamma\Phi(s_{n+1})]^T\mathcal{P}_{n-1}z_n}
\end{equation}
Combining (\ref{TDerror}), the first line of (\ref{TDlambda}), (\ref{VarianceMatrix}), and (\ref{ThetaRLSTD}), we have the following update rules of RLS-TD($\lambda$):
\begin{equation}\label{RLSTD}
\left\{
	\begin{split}
		&e_n=C(s_n,a_n)-[\Phi(s_n)-\gamma\Phi(s_{n+1})]^T\Theta_{n-1}\\
		&z_n=\gamma\lambda z_{n-1}+\Phi(s_n)\\
		&\mathcal{P}_n=\mathcal{P}_{n-1}-\frac{\mathcal{P}_{n-1}z_n[\Phi(s_n)-\gamma\Phi(s_{n+1})]^T\mathcal{P}_{n-1}}{1+[\Phi(s_n)-\gamma\Phi(s_{n+1})]^T\mathcal{P}_{n-1}z_n}\\
		&\Theta_n=\Theta_{n-1}+\frac{\mathcal{P}_{n-1}z_ne_n}{1+[\Phi(s_n)-\gamma\Phi(s_{n+1})]^T\mathcal{P}_{n-1}z_n}
	\end{split}
\right.
\end{equation}
where $\mathcal{P}_n$ is initialized as $\mathcal{P}_0=\beta \mathcal{I}$ ($\beta$ is a positive real number and $\mathcal{I}$ is the identity matrix).

Based on the methods discussed above, including parametric value function approximation (\ref{approx}), on-policy learning, and RLS-TD($\lambda$), we develop the ADP algorithm presented in Algorithm \ref{ADP}. The depth of each sampled trajectory is fixed to a positive integer $N$. In each week, the process of sampling new trajectories terminates when the relative improvement of $\Theta$ is lower than a threshold $\epsilon>0$ (line \ref{stopcriterion}). Before making the first decision in the first week $\tau=1$, all the components of $z_0$, $\mathcal{P}_0$, and $\Theta_0$ are initialized as $0$. Then, they are continuously updated in all the following weeks. In line \ref{skipstep1} and \ref{skipstep2} of Algorithm \ref{ADP}, we assume that Algorithm \ref{GreedyAction} is used to reduce the actions to be evaluated. If Algorithm \ref{GreedyAction} is not used, line \ref{skipstep1} and \ref{skipstep2} should be skipped and the reduced action spaces $A^*(s_n)$ and $A^*(s_\tau)$ in Algorithm \ref{ADP} should be replaced with the full action spaces $A(s_n)$ and $A(s_\tau)$, respectively.}
\begin{algorithm}[!htb]
	
	\caption{ADP algorithm based on RLS–TD($\lambda$) learning}\label{ADP}
	\KwIn{$s_\tau$, $z_{\tau-1}$, $\mathcal{P}_{\tau-1}$, and $\Theta_{\tau-1}$}
	Initialize $s_1=s_\tau$, $z_0=z_{\tau-1}$, $\mathcal{P}_0=\mathcal{P}_{\tau-1}$, and $\Theta_0=\Theta_{\tau-1}$;\\
	\While{true}
	{
		//\textit{A new trial (trajectory) starts here};\\
		\For{$n=1,2,...,N$}
		{
			Determine $A^*(s_n)$ by Algorithm \ref{GreedyAction}; //\textit{This step can be skipped;}\\\label{skipstep1}
			\ForEach{$a\in A^*(s_n)$}
			{
				Randomly sample $\Psi_n^a$ from Poisson distributions;
			}
			Compute $a_n$ by (\ref{EqOnpolicySimu});\\
			Compute $s_{n+1}=G^{s_n}_{a_n}+\Psi^{a_n}_n$;\\
			Update $z_n$, $\mathcal{P}_n$, and $\Theta_n$ by (\ref{RLSTD});\\
		}
		//\textit{A trial (trajectory) ends here};\\
		\lIf{$\|\frac{\Theta_n-\Theta_0}{\Theta_0}\|<\epsilon$}{break}\label{stopcriterion}
		Let $s_1=s_\tau$, $z_0=z_n$, $\mathcal{P}_0=\mathcal{P}_n$, and $\Theta_0=\Theta_n$;\\
	}
	Determine $A^*(s_\tau)$ by Algorithm \ref{GreedyAction}; //\textit{This step can be skipped.}\\\label{skipstep2}
	Let $z_{\tau}=z_n$, $\mathcal{P}_{\tau}=\mathcal{P}_n$, and $\Theta_{\tau}=\Theta_n$;\\
	Compute $a_\tau$ by (\ref{EqOnpolicyReal});\\
	\KwOut{$a_\tau$, $z_{\tau}$, $\mathcal{P}_{\tau}$, and $\Theta_{\tau}$}
\end{algorithm}

\section{Experimental results}\label{SecExp}
{We conduct numerical experiments on different test problems to validate the efficiency and accuracy of the proposed algorithms. All the programs are coded in C++ and run on a PC with an Intel(R) Core(TM) i7-3770 CPU @3.40GHz processor and a RAM of 8GB. In Section \ref{Exp1}, the algorithms proposed in this paper and several exact methodologies are employed to solve a small-sized test problem. We compare the computational performances of these solution approaches before examining the sensitivity of the resulting policy and the CPU time to the problem parameters. In Section \ref{Exp2}, we solve a large-sized single-specialty test problem and perform sensitivity analyses for the key parameters of the RLS–TD($\lambda$)-based ADP algorithm. Finally, we employ the proposed algorithms to solve a realistically sized multi-specialty patient admission control problem and present the results in Section \ref{Exp3}. 
	
{The quality of each policy computed in this section is evaluated through a numerical simulation with a finite time length $\tau_{max}$. In each simulated week, the numbers of different types of newly arrived patients are randomly sampled from Poisson distributions with constant arrival rates $\bar{n}_{ju}$. The values of $\tau_{max}$ and $\bar{n}_{ju}$ are different for different test problems. When multiple policies are computed for the same test problem, the value of $\tau_{max}$ and the sequences of patient arrivals $(\Psi_1,\Psi_2,...,\Psi_{\tau_{max}})$ are identical for the simulations of these policies. Each patient's surgery duration and LOS are realized when he/she is scheduled for surgery. The cost of each week is computed by taking the average value of cost function (\ref{CostDef}) in a large number of scenarios. Each scenario contains the realizations of all the scheduled patients' surgery durations and LOSs, which are randomly sampled from lognormal distributions. We arbitrarily set the scenario number to 10000.}

The notations used in this section for measuring the computational performances of the employed algorithms are listed in Table \ref{NotationsExpRes}.}

\begin{table}[!htb]
	\centering
	\small
	\caption{Definitions of variables used in Section \ref{SecExp}}\label{NotationsExpRes}
		\begin{tabular}{cl}
			\hline
			Variable & Definition \\
			\hline
			$T$ & total CPU time (s) \\
			$t$ & CPU time consumed in one week (ms) \\
			$\omega_{ju}$ & waiting time of {specialty-$j$} patients {with urgency coefficient} $u$ (week)\\
			{$o$} & {over-utilization of ORs in all specialties during one week (h)} \\
			{$o_j$} & over-utilization of ORs in specialty $j$ during one week (h) \\
			{$e$} & {shortage} of SICU recovery beds during one week (bed-day) \\
			$c$ & {total} cost of one week \\
			{$c_p$} & {{patient-related cost of one week}} \\
			{$c_h$} & {{hospital-related cost of one week}} \\
			$\bar{x}$ & mean of variable $x$ \\
			$\sigma(x)$ & standard deviation of variable $x$ \\
			$\|A\|$ & {total number of feasible actions for all the visited states} \\
			$\|A^*\|$ & total number of actually evaluated actions \\
			\hline
		\end{tabular}
\end{table}

\subsection{Algorithm comparisons and sensitivity analyses of problem parameters}\label{Exp1}
{
In this subsection, we conduct numerical experiments for an arbitrarily generated test problem. The proposed ADP algorithm is compared to two DP-based algorithms: VI and VPI–RTDP, whose complete procedures can be found in \shortciteA{kolobov2012planning} and \shortciteA{Sanner2009VPI}, respectively. For each algorithm employed, we search for the {best} actions in two different ways: by evaluating the action sets $A^*(s)$ determined by Algorithm \ref{GreedyAction} and by enumerating the {full} action sets $A(s)$. {We use a superscript * to highlight the algorithms with $A^*(s)$ adopted.} For example, VI* searches for the {best} actions in $A^*(s)$, while VI evaluates all the actions in $A(s)$. Due to the low computational efficiency of the DP-based algorithms, the test problem solved in this subsection is much smaller than the realistically sized ones. 

The small-sized test problem is a two-specialty patient admission control problem. The relative importance of the two specialties $j=1,2$ are $[v_1,v_2]=[1,2]$. Patients of the two specialties are divided into two urgency groups $u=1,2$, and their maximum {recommended} waiting times are $[W_{11},W_{12},W_{21},W_{21}]=[4,2,3,2]$ weeks. New patient arrivals are assumed to follow Poisson distributions, with the parameters $[\bar{n}_{11},\bar{n}_{12},\bar{n}_{21},\bar{n}_{22}]=[1.0,0.5,0.25,0.25]$. Considering that the DP-based algorithms are unable to tackle infinite state spaces, we truncate the Poisson distributions by omitting the values whose probabilities are lower than 0.005; thus $\tilde{n}_{11}<5$, $\tilde{n}_{12}<4$, $\tilde{n}_{21}<3$, and $\tilde{n}_{22}<3$. The size of the state space of this problem can thereby be calculated as $|S|=5^4\times4^2\times3^3\times3^2=2.43\times10^6$. Surgery durations and LOSs follow lognormal distributions, with $[\bar{d}_1,\bar{d}_2]=[2,4]$ hours and $[\bar{l}_1,\bar{l}_2]=[4,2]$ days, and the standard deviations of surgery durations and the LOS for each specialty are assumed to be equal to their means. The unit costs of the small-sized test problem are determined as $[{c_b},{c_d}]=[50,100]$, $c_o=400$ per hour and ${c_e}=1000$ per patient-day. The regular OR capacities for the two specialties are $[B_1,B_2]=[3,2]$ hours, and the available SICU capacity is $R=7$ bed-days per week. The estimated availability rates of ORs and the SICU are arbitrarily determined as $\rho_1=\rho_2=1$ in the small-sized test problem.

{Using the Monte-Carlo simulation method, we randomly sample arrival information over 1,000 weeks, as well as sampling the surgery durations and LOSs of all the arrived patients. Using these samples, we carry out a numerical simulation for each solution approach for 1,000 consecutive weeks.} The parameters of the ADP algorithm are arbitrarily set as $\lambda=0$, $\beta=1$, $N=5000$ and $\epsilon=0.0001$ in this subsection, while the discount factor of the MDP {model} is $\gamma=0.99$. The computational results are presented in {Table \ref{TabExpSmall}}.

\begin{table}[!htb]
	\centering
	\scriptsize
	
	\caption{Experimental results of the small-sized test problem}\label{TabExpSmall}
	\begin{adjustbox}{center}
	\begin{threeparttable}
		\addtolength\tabcolsep{-0.25em}
		\begin{tabular}{llccccccccccccccccrc}
			\hline
			Algorithm & Type & $\bar{\omega}_{11}$ & $\sigma(\omega_{11})$ & $\bar{\omega}_{12}$ & $\sigma(\omega_{12})$ & $\bar{\omega}_{21}$ & $\sigma(\omega_{21})$ & $\bar{\omega}_{22}$ & $\sigma(\omega_{22})$ & $\bar{o}_{1}$ & $\sigma(o_1)$ & $\bar{o}_2$ & $\sigma(o_2)$ & $\bar{e}$ & $\sigma(e)$ & $\bar{c}$ & $\sigma(c)$ & \multicolumn{1}{c}{$T$} & $\frac{\|A^*\|}{\|A\|}$ \\
			\hline
			VI & off-line & 2.293 & 0.938 & 1.293 & 0.455 & 1.463 & 0.599 & 1.051 & 0.219 & 0.965 & 2.081 & 1.162 & 2.970 & 2.341 & 5.222 &       3,716  &       5,669  & 85,492 & 1.000 \\
			VI* & off-line & 2.280 & 0.922 & 1.304 & 0.460 & 1.474 & 0.606 & 1.051 & 0.219 & 0.961 & 2.099 & 1.160 & 2.971 & 2.340 & 5.234 &       3,711  &       5,706  &  6,835 & 0.063 \\
			VPI-RTDP & on-line & 2.277 & 0.925 & 1.316 & 0.465 & 1.452 & 0.604 & 1.046 & 0.210 & 0.957 & 2.101 & 1.161 & 2.973 & 2.336 & 5.226 &       3,704  &       5,701  &  5,437 & 1.000 \\
			VPI-RTDP* & on-line & 2.277 & 0.925 & 1.316 & 0.465 & 1.452 & 0.604 & 1.046 & 0.210 & 0.957 & 2.101 & 1.161 & 2.973 & 2.336 & 5.226 &       3,704  &       5,701  &  1,450 & 0.216 \\
			ADP & on-line & 2.521 & 1.090 & 1.432 & 0.495 & 1.099 & 0.299 & 1.000 & 0.000 & 0.997 & 2.168 & 1.186 & 3.037 & 2.373 & 5.383 &       3,821  &       5,808  &     14 & 1.000 \\
			ADP* & on-line & 2.521 & 1.092 & 1.432 & 0.495 & 1.096 & 0.294 & 1.000 & 0.000 & 0.998 & 2.166 & 1.186 & 3.026 & 2.371 & 5.384 &       3,820  &       5,802  &     12 & 0.901 \\\hline
		\end{tabular}
		\begin{tablenotes}
			\item Algorithms with superscript * evaluate the action set $A^*(s)$ determined by Algorithm \ref{GreedyAction}, the others enumerate the entire action space $A^*(s)=A(s)$. 
		\end{tablenotes}
	\end{threeparttable}
	\end{adjustbox}
\end{table}

From {Table \ref{TabExpSmall}}, it can be observed that the RLS–TD($\lambda$)-based ADP algorithm is the most efficient solution approach. The CPU time consumed by ADP is only 0.02\% and 0.26\% of that consumed by VI and VPI–RTDP, respectively. Moreover, applying Algorithm \ref{GreedyAction} drastically improves the computational efficiency of all the employed solution approaches, and does not lead to any policy deterioration. To be specific, when applied to VI, VPI–RTDP, and ADP, Algorithm \ref{GreedyAction} reduces CPU time by 92.00\%, 73.33\%, and 14.28\%, respectively. {Table \ref{TabExpSmall}} also indicates that the resulting policies of the algorithms employed are similar. The relative gap between the total costs of ADP/ADP* and those of the DP-based algorithms is below 3.16\%. Among the DP-based algorithms, we can observe that the total costs of VI/VI* are slightly higher than those of VPI–RTDP/VPI–RTDP*, implying that the policies of VI/VI* are {slightly further from convergence} than those of VPI–RTDP/VPI–RTDP*. Obviously, the former can be further improved by modifying the user-defined stopping criterion, but {doing so} requires even more CPU time. In summary, the RLS–TD($\lambda$)-based ADP algorithm provides a high-quality, near-optimal policy for the problem under study, and Algorithm \ref{GreedyAction} is capable of accelerating the employed algorithms without any deterioration of the policy.

Next, we perform sensitivity analyses to evaluate the effects of the problem parameters on the resulting policy and the computational complexity. We use an exact algorithm (VPI–RTDP*) and the proposed ADP* algorithm to solve the small-sized two-specialty test problem with different values for the problem parameters ${c_d}$, $c_o$, ${c_e}$, $B_1$, $B_2$, and $R$. The results of the sensitivity analyses are demonstrated in Figures \ref{FigSensCw} to \ref{FigSensR}. It can be seen from these figures that the ADP* policy is very close to that of VPI–RTDP* in most cases, despite variations in the problem parameters; again, this validates the accuracy of the ADP* algorithm. Specifically, Figure \ref{FigSensCw} shows that changing the unit waiting cost ${c_d}$ (or the ratio of ${c_d}$ to ${c_b}$) only leads to slight changes in the resulting policy. In contrast, Figures \ref{FigSensCo} and \ref{FigSensCe} illustrate that the hospital-related cost $c_h$ and the total cost $c$ are increasing in both the unit OR overtime cost $c_o$ and the unit SICU excess cost ${c_e}$, while Figures \ref{FigSensB1}, \ref{FigSensB2}, and \ref{FigSensR} show that $c_h$ and $c$ are decreasing in the OR capacity $B_1$, $B_2$, and the SICU capacity $R$. From Figures \ref{FigSensCw} to \ref{FigSensR}, we can also observe that the patient-related cost $c_p$ remains relatively stable as the problem parameters change, which implies that the resulting policy is not very sensitive to the problem parameters. In terms of computational complexity, Figures \ref{FigSensCw} to \ref{FigSensR} show that the CPU time consumed by VPI–RTDP* is increasing in $c_o$, ${c_e}$ and decreasing in ${c_d}$, $B_2$, $R$, while the CPU time for ADP* does not change too much as the problem parameters vary. Additionally, the CPU time for ADP* is much less than that of VPI-RTDP* in all cases.
 
\begin{figure}[!htb]
	\includegraphics[width=6.5in]{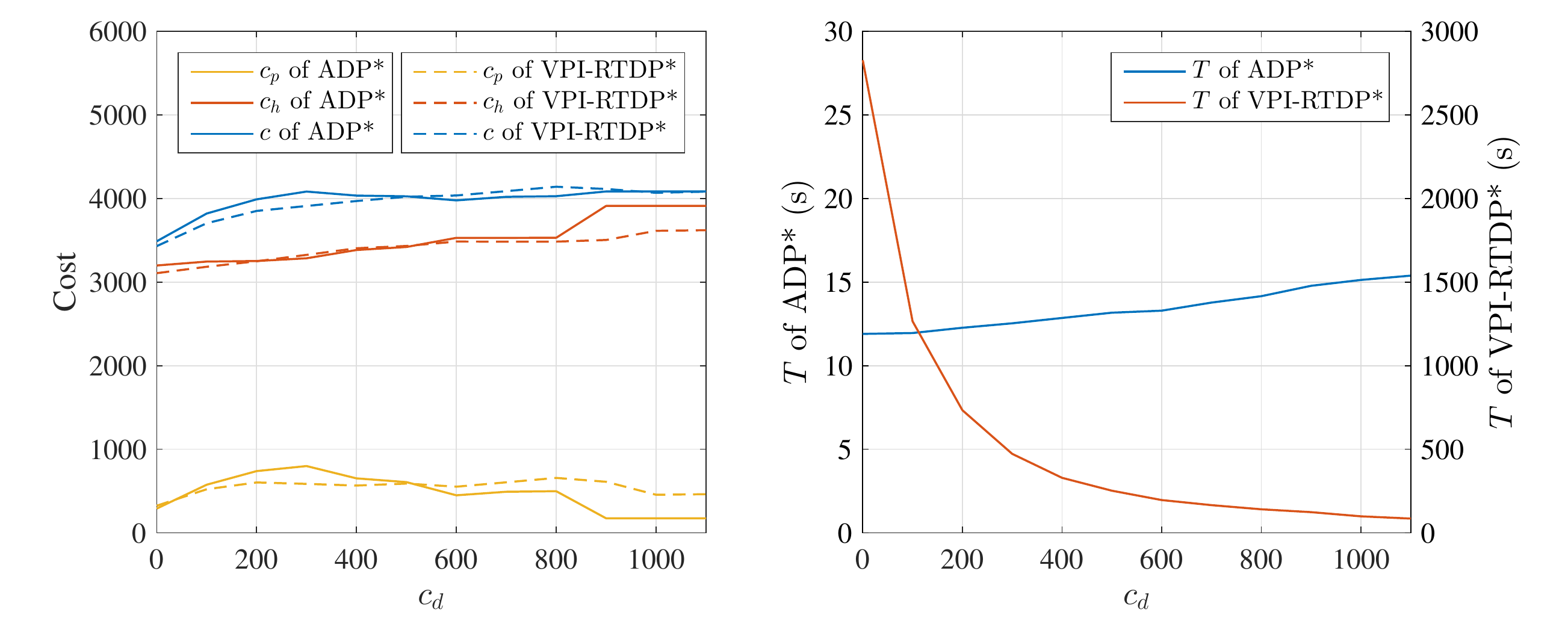}
	\vspace{-0.5cm}
	\caption{Sensitivity analysis for the unit waiting cost ${c_d}$}\label{FigSensCw}
\end{figure} 

\begin{figure}[!htb]
	\includegraphics[width=6.5in]{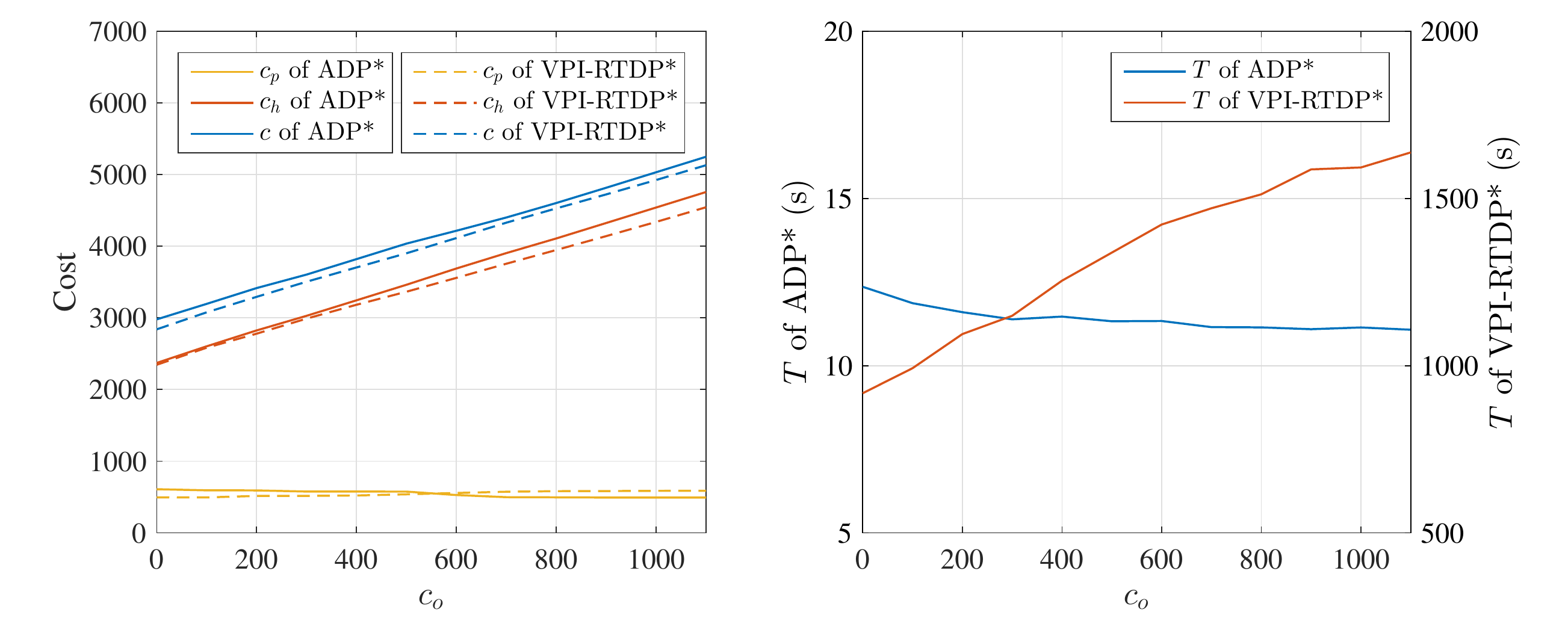}
	\vspace{-0.5cm}
	\caption{Sensitivity analysis for the unit cost $c_o$ of overusing ORs}\label{FigSensCo}
\end{figure} 

\begin{figure}[!htb]
	\includegraphics[width=6.5in]{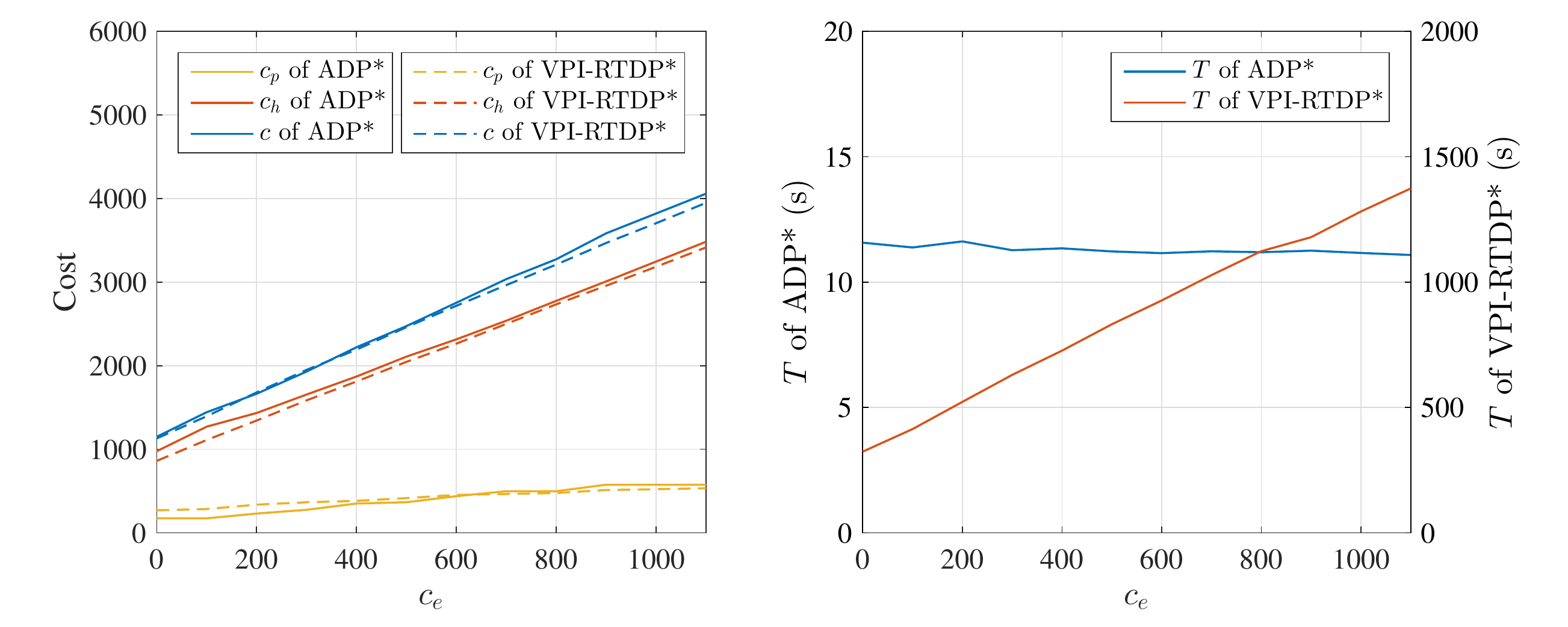}
	\vspace{-0.5cm}
	\caption{Sensitivity analysis for the unit cost ${c_e}$ of exceeding SICU capacity}\label{FigSensCe}
\end{figure} 

\begin{figure}[!htb]
	\includegraphics[width=6.5in]{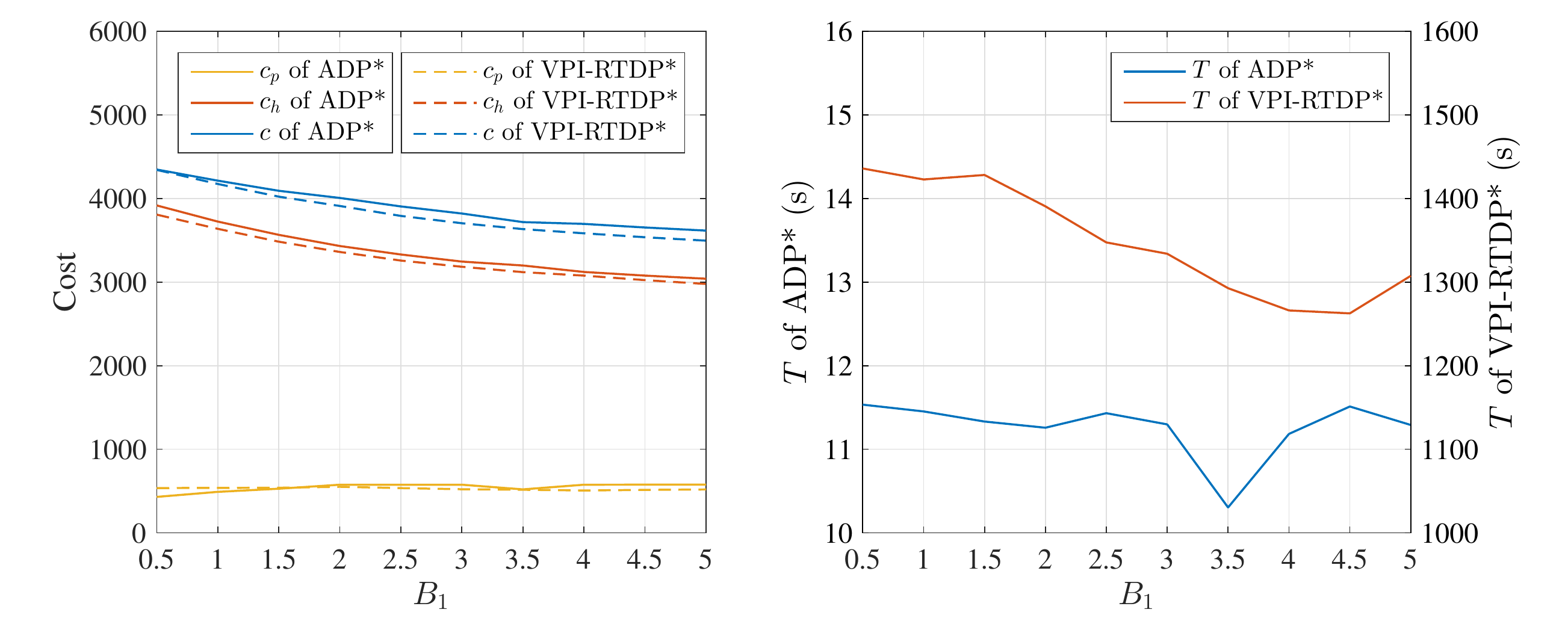}
	\vspace{-0.5cm}
	\caption{Sensitivity analysis for the OR capacity $B_1$ of specialty 1}\label{FigSensB1}
\end{figure} 

\begin{figure}[!htb]
	\includegraphics[width=6.5in]{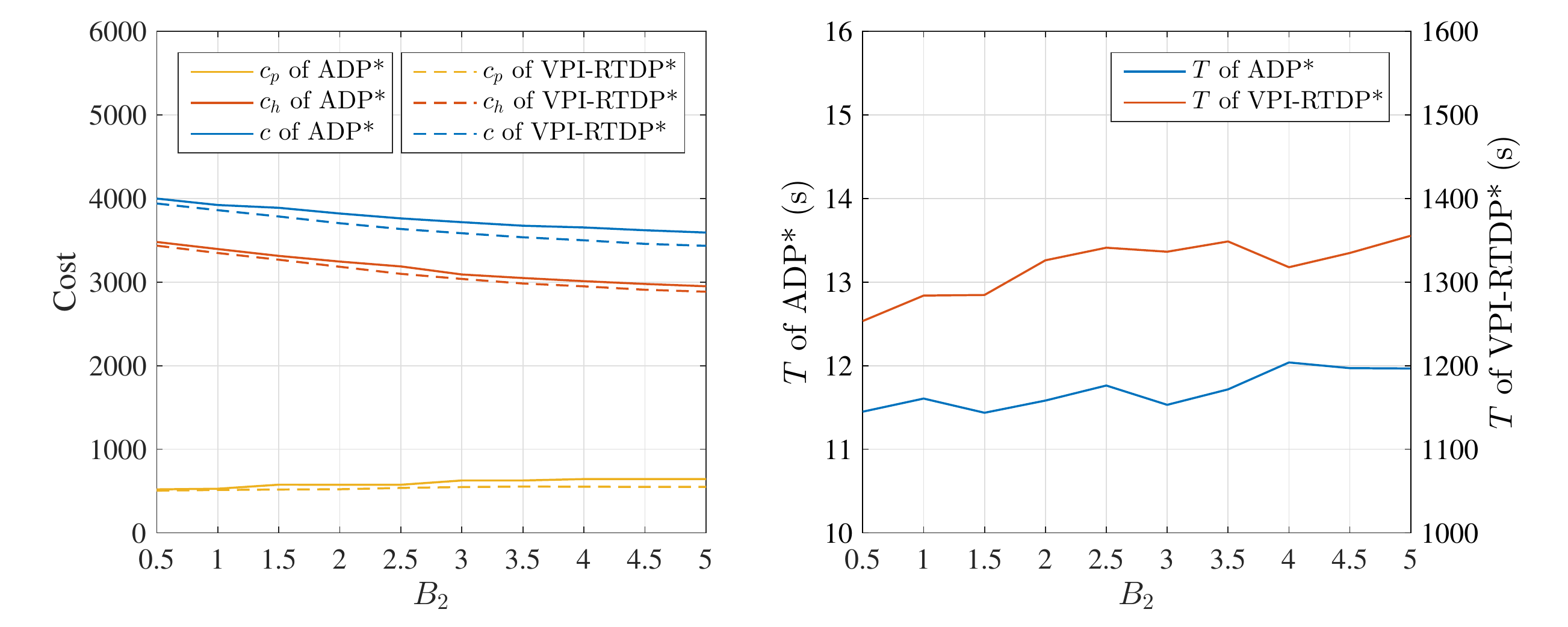}
	\vspace{-0.5cm}
	\caption{Sensitivity analysis for the OR capacity $B_2$ of specialty 2}\label{FigSensB2}
\end{figure} 

\begin{figure}[!htb]
	\includegraphics[width=6.5in]{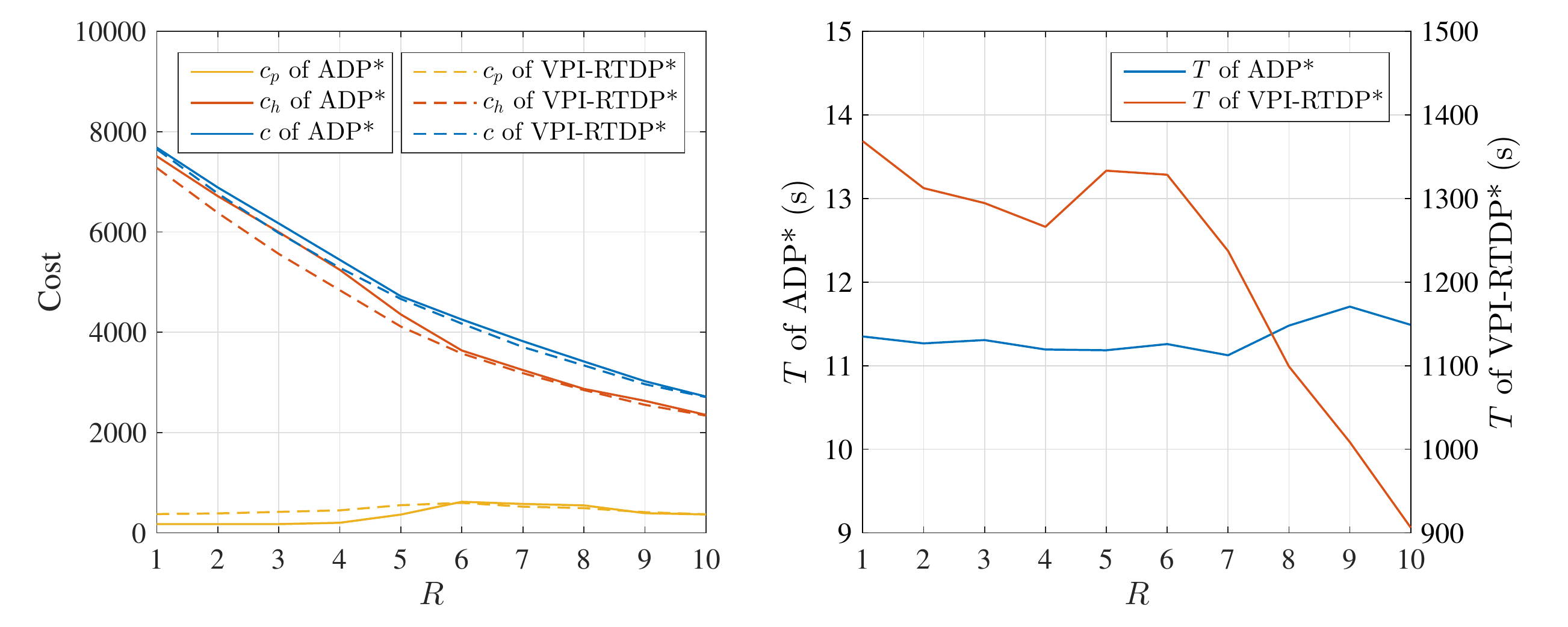}
	\vspace{-0.5cm}
	\caption{Sensitivity analysis for the SICU capacity $R$}\label{FigSensR}
\end{figure} 
}

\subsection{Sensitivity analyses of the ADP algorithm parameters}\label{Exp2}

The experimental results of the previous subsection reveal that the proposed RLS–TD($\lambda$)-based ADP algorithm {is distinctly more efficient than the conventional DP-based algorithms{. At} the same time, the gaps between the resulting costs of these algorithms are insignificant. To further evaluate} the capability of the proposed algorithms to cope with large cases, {we employ the ADP* algorithm (i.e., the RLS–TD($\lambda$)-based ADP algorithm in combination with the action searching method presented by Algorithm \ref{GreedyAction}) to solve a realistically sized single-specialty patient admission control problem, and perform sensitivity analyses for the key parameters $\lambda$ and $\beta$ of the ADP/ADP* algorithm.} 

{The problem to be solved in this subsection is {the} admission control {of} coronary artery bypass grafting (CABG) surgeries, which has previously been studied by \shortciteA{min2010elective}}. Our work differs from \shortciteA{min2010elective}, however, as it incorporates patient waiting times, dynamic patient {priority scores,} and due dates into the MDP model. The problem settings are based on the configurations of \shortciteA{min2010elective} and the real-life data provided by \shortciteA{sobolev2008analysis}. Patients from the same specialty ($j=1$) are divided into three urgency groups, with $u=1,2,6$, $[W_{11},W_{12},W_{16}]=[12,6,2]$, $[\bar{n}_{11},\bar{n}_{12},\bar{n}_{16}]=[3,5,1]$, $\tilde{n}_{11}\leqslant9$, $\tilde{n}_{12}\leqslant13$ and $\tilde{n}_{16}\leqslant5$. The size of the state space is as large as $10^{12}\times14^6\times6^2\approx2.71\times10^{20}$. Parameters of lognormal distributions for surgery durations and LOSs are $[\bar{d}_1,\sigma(d_1)=[4, 1.72]$ hours and $[\bar{l}_1,\sigma(l_1)]=[2, 2]$ days, respectively. Costs are determined as $[{c_b},{c_d},c_o,{c_e}]=[100,150,1500,1500]$. Available surgical resources are $B_1=40$ hours and $R=25$ bed-days, with estimated availability rates $[\rho_1,\rho_2]=[0.9,0.72]$.

Simulations for 1,000 consecutive weeks are performed to address the CABG {surgery admission control} problem. As a result of the increasing problem scale, DP-based algorithms are no longer capable of {computing} the optimal policy within an acceptable CPU time. Therefore, current practice is to compute a myopic policy in which each action only minimizes the cost of the present week. This policy serves as the benchmark for the resulting ADP* policy and can easily be obtained by solving the MDP {model} with discount factor $\gamma=0$. { Since the action spaces of realistic problems are considerably large, we employ Algorithm \ref{GreedyAction} to reduce the number of actions to be evaluated.} The simulation results for the myopic policy, as well as the resulting ADP* policies with $N=1000$ and $\epsilon=0.001$, are listed in {Table \ref{tabCABG}}. 

\begin{table}[!htb]
\centering

\footnotesize
\caption{Experimental results of the CABG problem}\label{tabCABG}
\begin{adjustbox}{center}
\begin{threeparttable}
	\addtolength\tabcolsep{-0.25em}
	\begin{tabular}{cccccccccccccccrrc}
		\hline
		$\gamma$ & $\lambda$ & $\beta$ & \multicolumn{1}{c}{$\bar{\omega}_{11}$} & $\sigma(\omega_{11})$ & $\bar{\omega}_{12}$ & $\sigma(\omega_{12})$ & $\bar{\omega}_{16}$ & $\sigma(\omega_{16})$ & $\bar{o}$ & $\sigma(o)$ & $\bar{e}$ & $\sigma(e)$ & $\bar{c}$ & $\sigma(c)$ & \multicolumn{1}{c}{$\bar{t}$} & \multicolumn{1}{c}{$\sigma(t)$} & $\frac{\|A^*\|}{\|A\|}$ (\textperthousand)\\
		\hline
		\vspace{0.2cm}
		0.00 & --- & --- & 4.855 & 2.362 & 2.493 & 1.168 & 1.159 & 0.366 & 1.658 & 2.583 & 1.697 & 2.808 & 19008 & 12651 & 0.011 & 0.104 & $<$0.001 \\
		\vspace{-0.08cm}
		0.99 & 0.0 & $10^{-3}$ & 4.387 & 2.048 & 2.203 & 0.942 & 1.030 & 0.171 & 2.028 & 3.565 & 1.877 & 2.996 & 17008 & 12138 & 62.762  & 152.343 & 0.012 \\
		\vspace{-0.08cm}
		&  & 1 & 4.197 & 1.927 & 2.111 & 0.862 & 1.008 & 0.089 & 2.183 & 3.916 & 1.947 & 3.249 & 16416 & 12813 & 76.698  & 349.497 & 0.929 \\
		\vspace{0.12cm}
		&  & $10^3$ & 3.160 & 1.314 & 1.565 & 0.496 & 1.000 & 0.000 & 2.417 & 4.744 & 1.889 & 3.265 & 12382 & 12177 & 55.780  & 349.204 & 4.936 \\
		\vspace{-0.08cm}
		& 0.5 & $10^{-3}$ & 4.773 & 2.323 & 2.456 & 1.147 & 1.150 & 0.357 & 1.620 & 2.647 & 1.461 & 2.532 & 18185 & 12754 & 68.964  & 157.676 & 0.002 \\
		\vspace{-0.08cm}
		&  & 1 & 4.150 & 1.881 & 2.089 & 0.840 & 1.000 & 0.000 & 2.140 & 4.147 & 1.844 & 3.164 & 15957 & 13295 & 79.743  & 219.785 & 1.299 \\
		\vspace{0.12cm}
		&  & $10^{3}$ & 3.206 & 1.571 & 1.585 & 0.832 & 1.000 & 0.000 & 2.295 & 4.122 & 1.952 & 3.063 & 12468 & 12605 & 78.136 & 165.882 & 1.577 \\
		\vspace{-0.08cm}
		& 1.0 & $10^{-3}$ & 2.035 & 0.857 & 1.001 & 0.011 & 1.000 & 0.000 & 3.195 & 5.697 & 2.258 & 3.775 & 11032 & 13272 & 95.347  & 293.532 & 1.753\\
		\vspace{-0.08cm}
		&  & 1 & 2.031 & 0.841 & 1.000 & 0.000 & 1.000 & 0.000 & 2.965 & 5.522 & 2.241 & 3.666 & 10662 & 13080 & 134.539 & 186.913 & 3.705 \\
		&  & $10^{3}$ & 2.052 & 0.879 & 1.003 & 0.053 & 1.000 & 0.000 & 3.151 & 5.561 & 2.413 & 4.200 & 11199 & 13730 & 94.664  & 198.483 & 4.194\\
		\hline
	\end{tabular}
\end{threeparttable}
\end{adjustbox}
\end{table}

{From the data presented in {Table \ref{tabCABG}}, we can see that all the simulations can be finished within a reasonable CPU time, and that computing the myopic policy ($\gamma=0.00$) is much easier than computing the ADP* policy ($\gamma=0.99$); that is, the CPU time consumed by the former ($\bar{t}=0.011$ms) is much less than that of the latter ($\bar{t}\in(62,135)$ms). However, the ADP* policy significantly outperforms the myopic policy in terms of reducing the total costs and the patients' waiting times. In comparison to the ADP* policy, which minimizes the expected total costs over an infinite horizon, the myopic policy optimizes the cost of each week separately, without considering the inter-week correlations. As {a result}, it generally schedules fewer patients than does the ADP* policy and thereby leads to longer waiting times. To be specific, the average waiting times of patients {with urgency coefficients} 1, 2 and 6 under the ADP* policy are lower than those under the myopic policy by $31.4\pm21.4\%$, $33.1\pm21.7\%$ and $11.9\pm4.0\%$, respectively. As the ADP* policy tends to schedule more patients than does the myopic policy, it increases the over-utilizations of the ORs and SICU by $0.79\pm0.51$ hours per week and $0.29\pm0.27$ patient-days per week, respectively. However, the total cost of the ADP* policy is $26.8\pm14.6\%$ lower than that of the myopic policy. {Table} \ref{tabCABG} also illustrates the effects of the values of $\lambda$ and $\beta$ on the resulting policy. It can be observed that the patients' waiting times and the total cost are decreasing in $\lambda$ and $\beta$; however, the over-utilizations of surgical resources are increasing in these parameters. 
	
We record the sizes of $A(s)$ and $A^*(s)$ for each visited state in the simulations for the CABG surgery admission control problem. We find that the sizes of $A(s)$ of some states can be as large as $1.59\times10^{11}$, while the sizes of $A^*(s)$ determined by Algorithm \ref{GreedyAction} are no larger than 75. {Table \ref{tabCABG}} also shows that the ratio of $\|A^*\|$ to $\|A\|$ in each simulation is {lower} than 5\textperthousand, which is significantly lower than the values of $\|A^*\|/\|A\|$ for the small-sized test problem solved in Section \ref{Exp1}. This fact implies that Algorithm \ref{GreedyAction} leads to {greater} improvements in computational efficiency for larger problems.}

To {further} explore how the scheduling policy can be influenced by the key parameters of the ADP* algorithm, we analyze the sensitivity of several performance measures to the variances of $\lambda$ and $\beta$. Figure \ref{FigSen} demonstrates the experimental results for the CABG problem; it indicates that $\lambda$ varies from 0.0 to 1.0 and $\beta$ varies from $10^{-5}$ to $10^5$. It can be observed that, when $\lambda$ exceeds 0.8, the patients' waiting times and the total cost sharply decrease, regardless of the value of $\beta$; meanwhile, {the over-utilizations of the ORs and SICU increase significantly}. When $\lambda\leqslant0.8$, the larger value of $\beta$ results in lower cost and shorter waiting times for patients{. However}, it also leads to {greater} overuse of the ORs and SICU. More importantly, no matter how the values of $\lambda$ and $\beta$ vary, the average cost per week $\bar{c}$ of the ADP* policy is always { significantly} lower than that of the myopic policy (refer to {Table \ref{tabCABG}}). 

\begin{figure}[!htb]
	\subfigure
	{
		\label{wt1}
		{\includegraphics[width=2.3in]{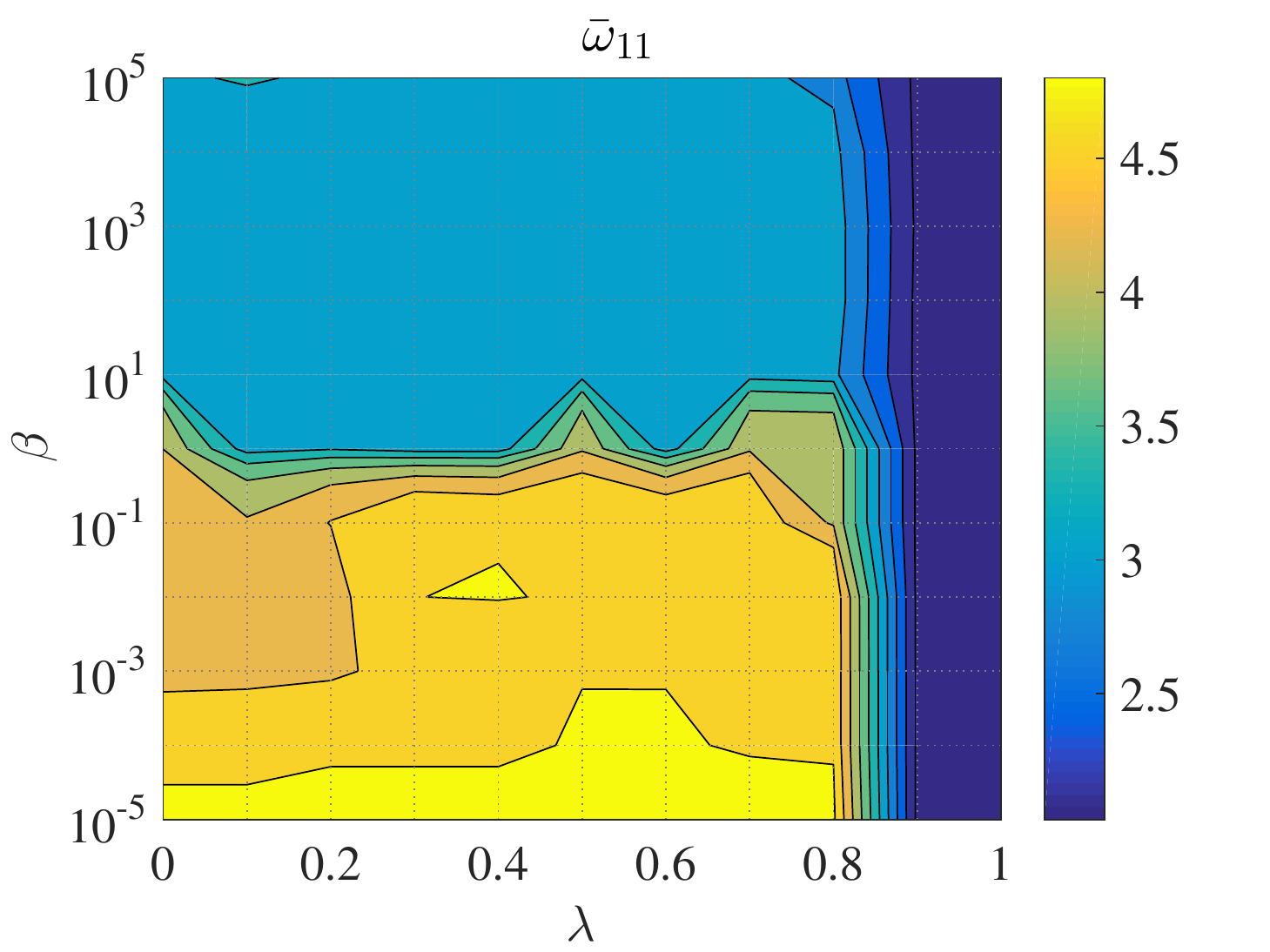}}
	}
	\hskip-0.5cm
	\subfigure
	{
		\label{wt2}
		{\includegraphics[width=2.3in]{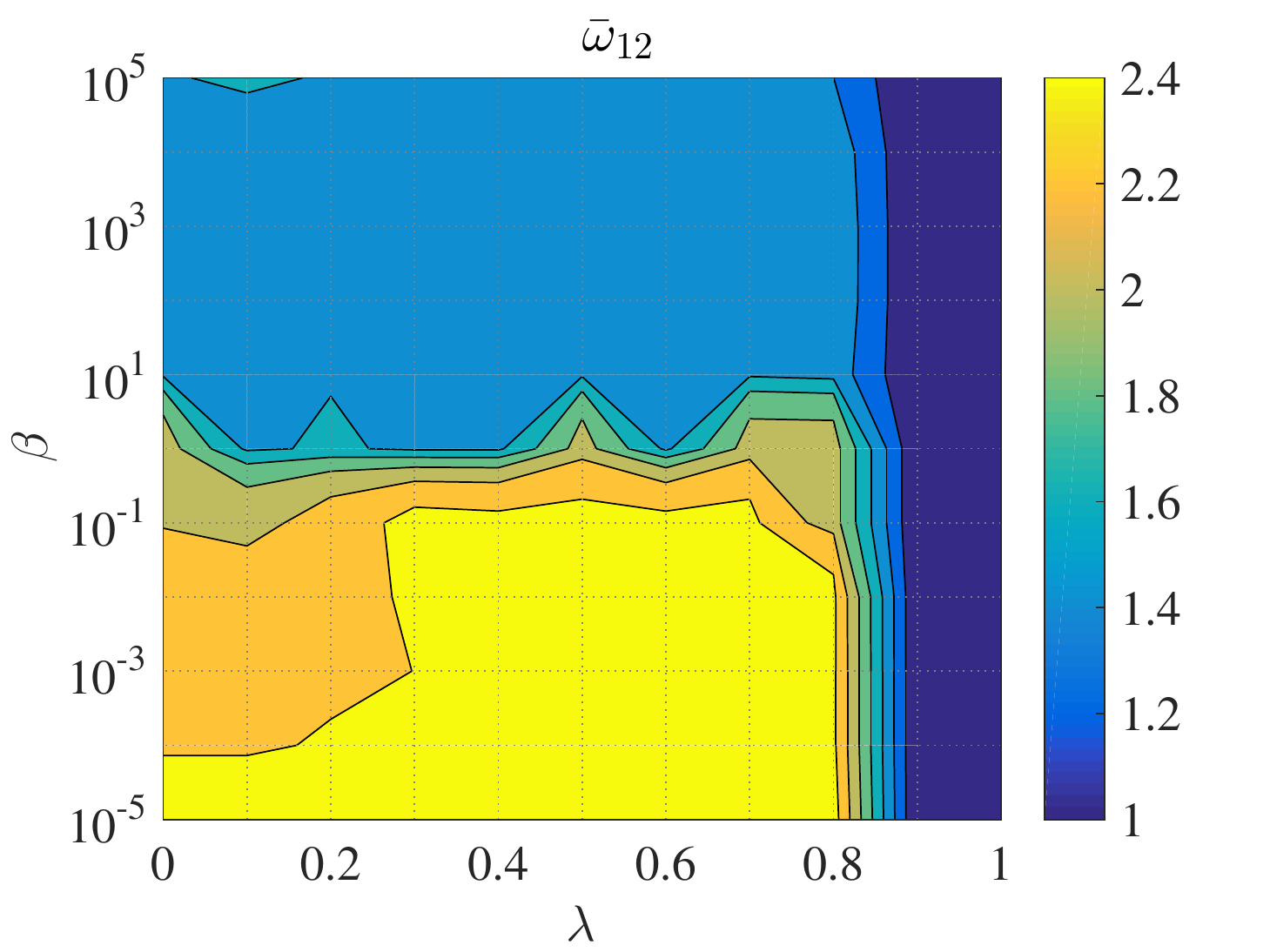}}
	}
 	\hskip-0.5cm
	\subfigure
	{
		\label{wt3}
		{\includegraphics[width=2.3in]{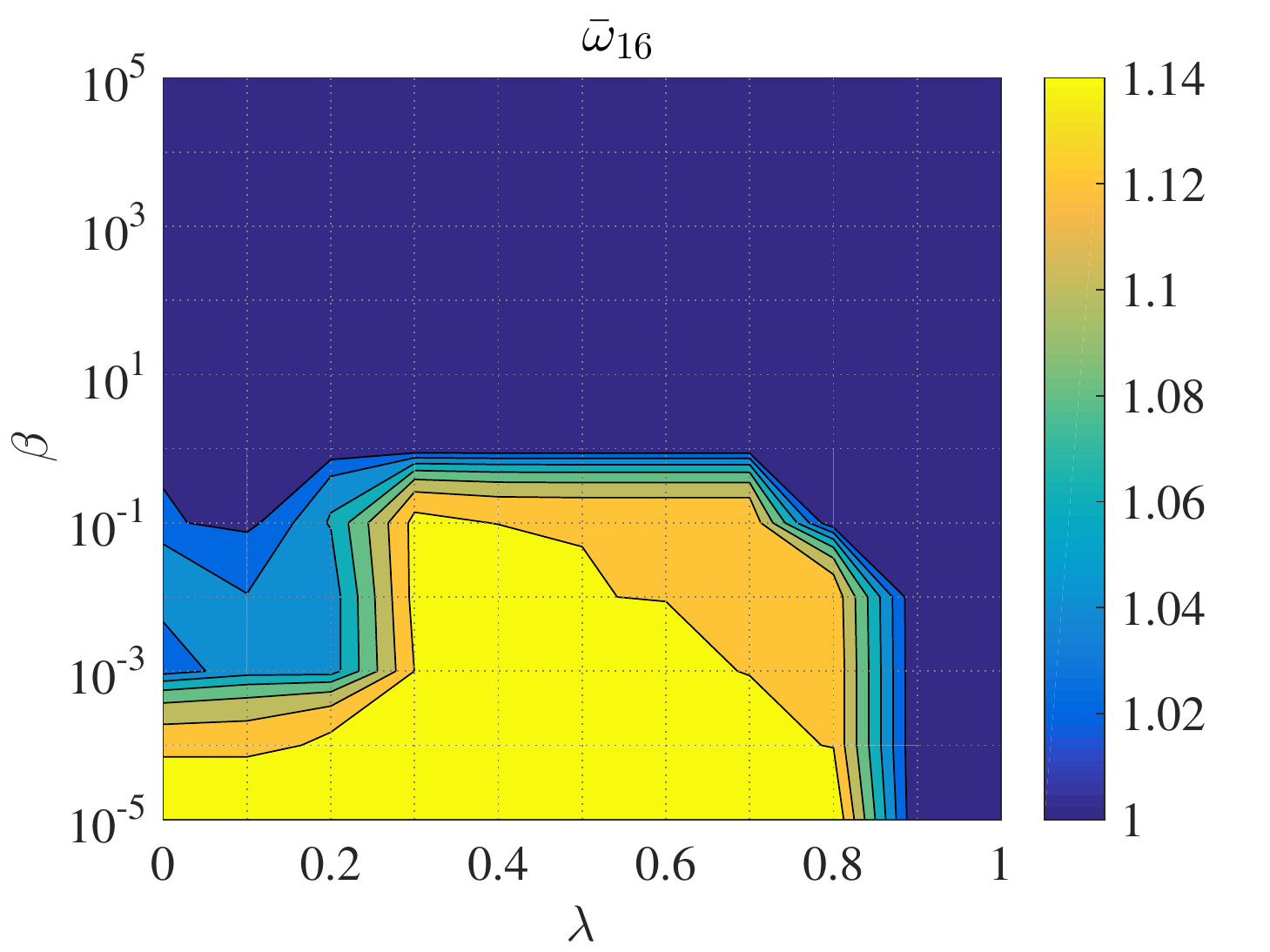}}
	}\\
	\subfigure
	{
		\label{o1}
		{\includegraphics[width=2.3in]{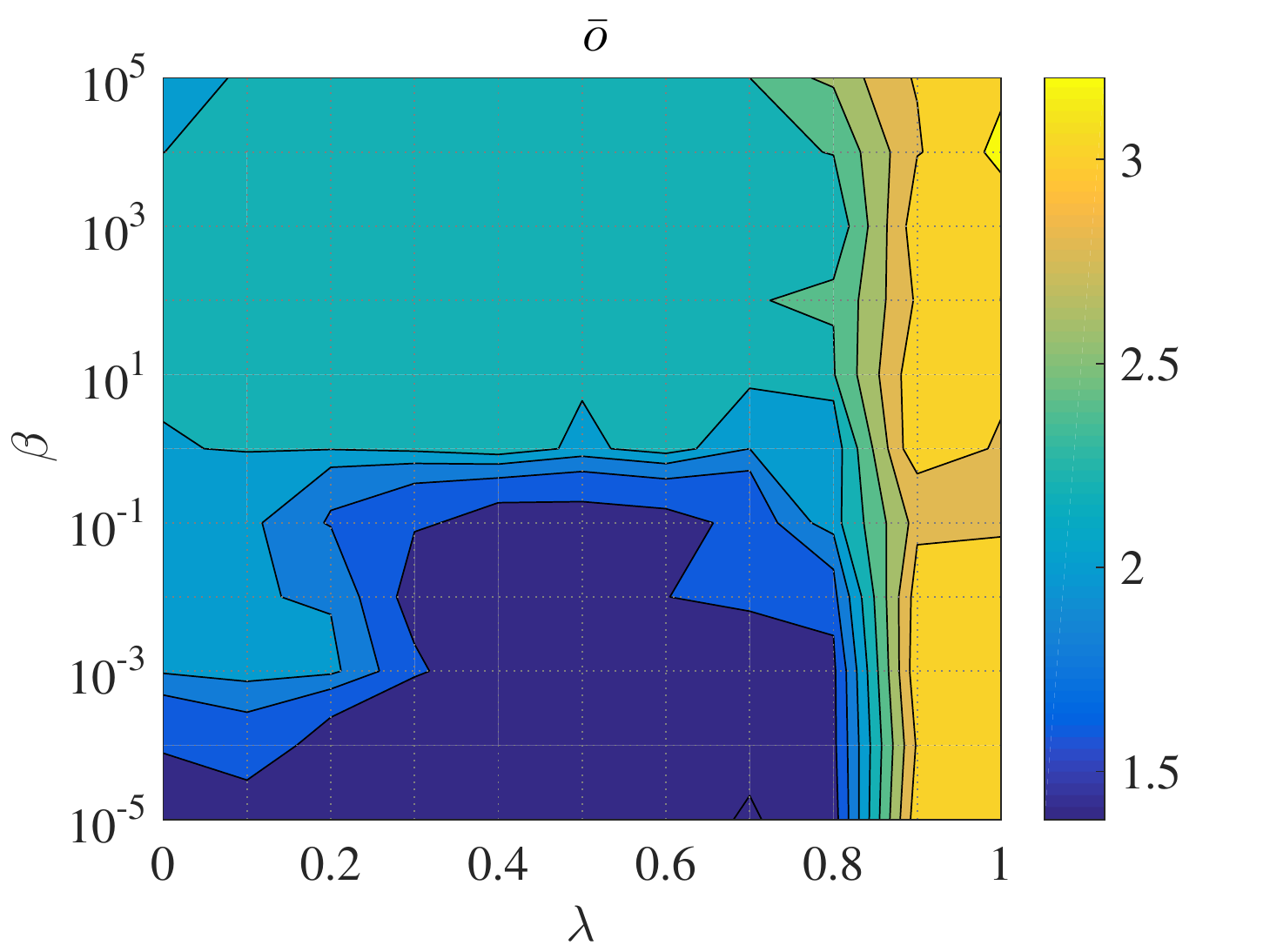}}
	}
	\hskip-0.5cm
	\subfigure
	{
		\label{o2}
		{\includegraphics[width=2.3in]{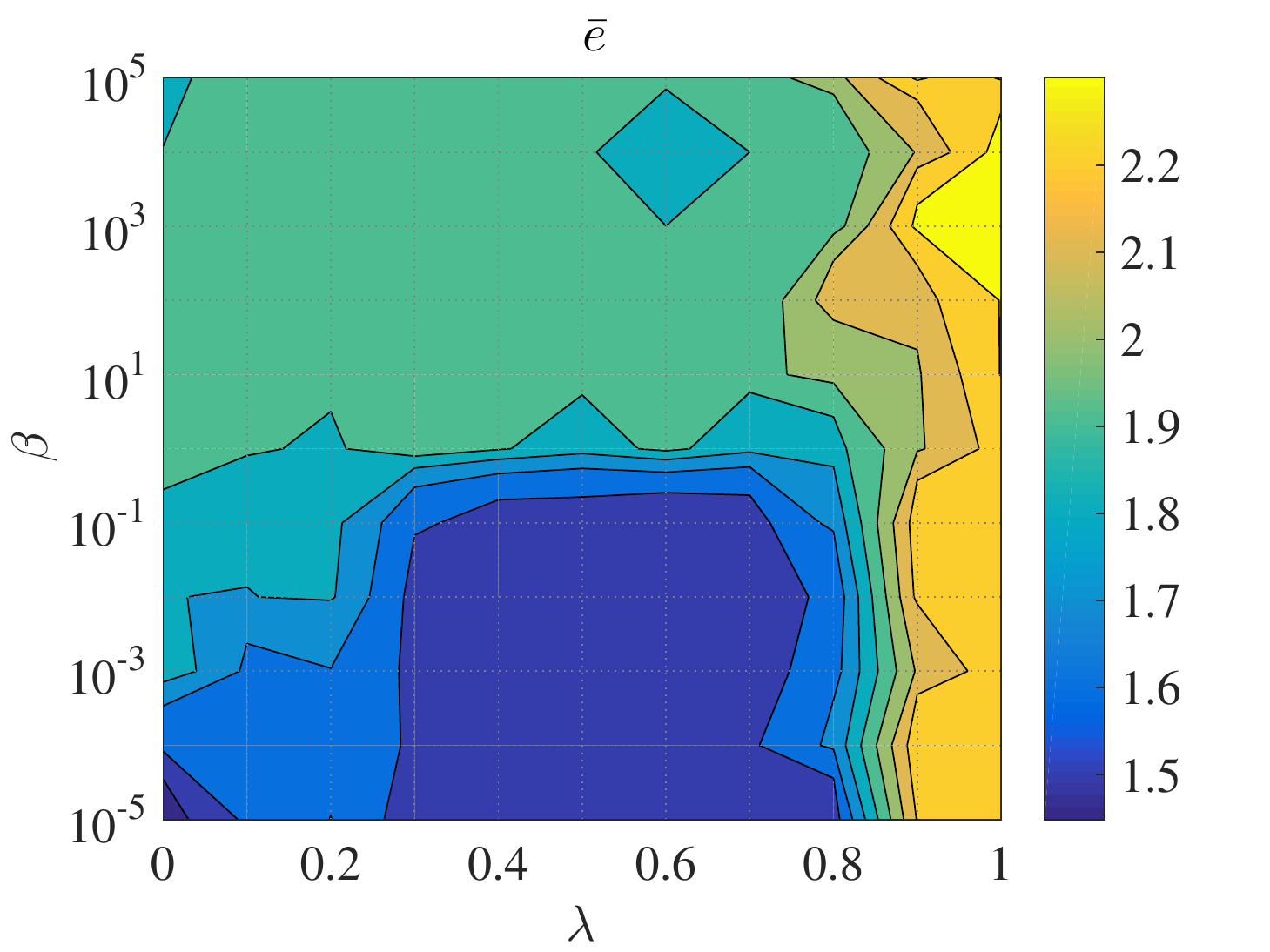}}
	}
	\hskip-0.5cm
	\subfigure
	{
		\label{c}
		{\includegraphics[width=2.3in]{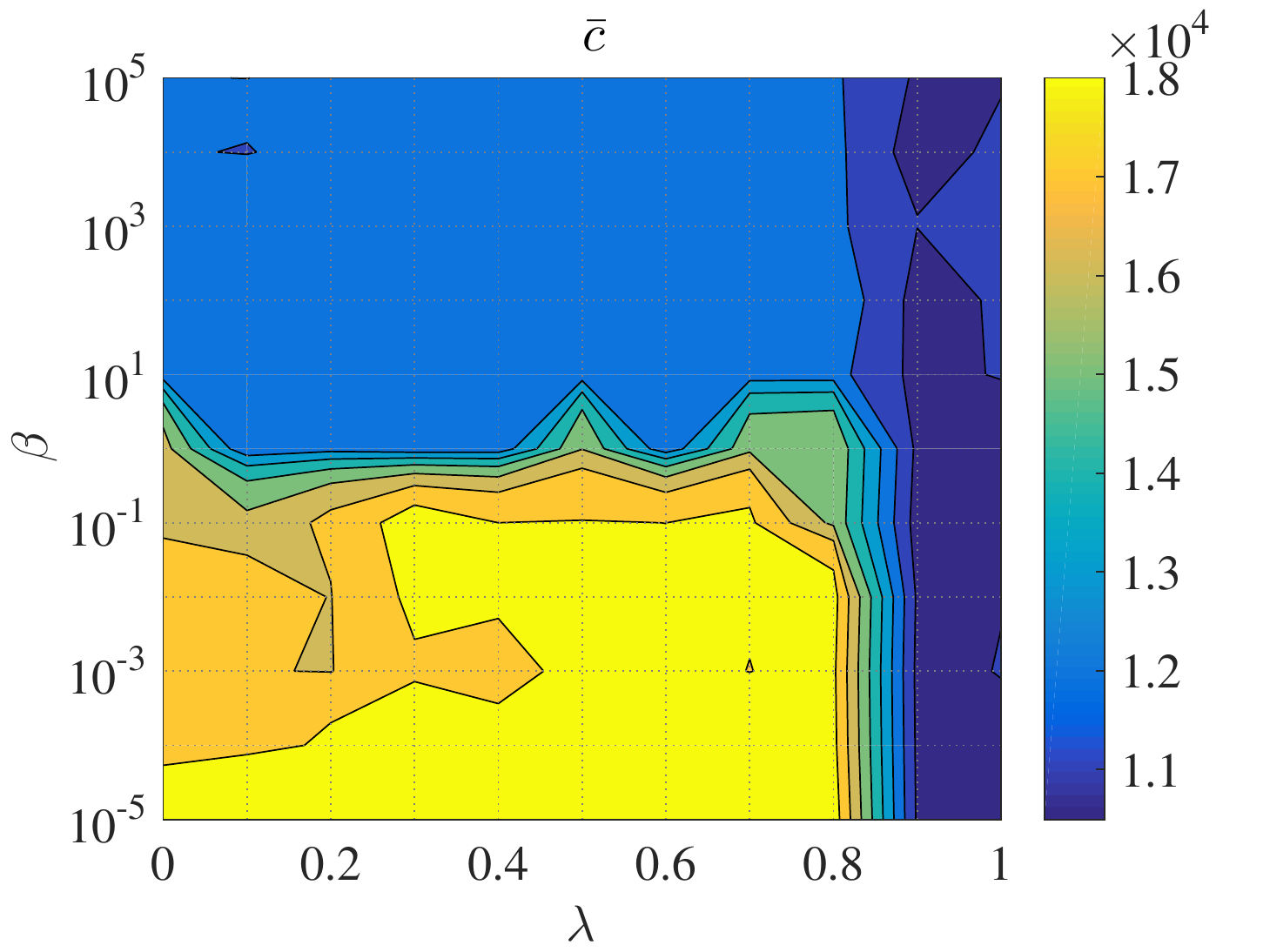}}
	}
	\caption{Experimental results for the CABG test problem with {variations in $\lambda$ and $\beta$}}\label{FigSen}
\end{figure}

Figure \ref{FigTheta} compares the evolutions of the parameter vector $\Theta$ during the simulations, using different values for $\lambda$ and $\beta$. The length of vector $\Theta$ for the CABG problem is $\Xi={20}$; we arbitrarily select the values of $\theta_1$ and $\theta_{11}$ to be demonstrated in Figure \ref{FigTheta}. The first row of Figure \ref{FigTheta} shows that the smaller value of $\beta$ leads to a lower convergence rate, and that the values of $\theta_1$ and $\theta_{11}$ converge especially slowly when $\beta=10^{-5}$. The second row of Figure \ref{FigTheta} reveals that {the values to which parameters $\theta_\xi$ converge} are increasing in line with the value of $\lambda$, which explains why different values of $\lambda$ result in different policies ({refer to} {Table \ref{tabCABG}} and Figure \ref{FigSen}). In terms of computational efficiency, Figure \ref{FigSubCPULambda} shows that the CPU time goes up as $\lambda$ increases from 0.0 to 0.8, which implies that larger values of $\lambda$ require more CPU time before the convergence of $\Theta$ occurs. When $0.8\leqslant\lambda\leqslant1.0$, the resulting policy tends to schedule more patients in each decision; we can see in Figure \ref{FigSen} that patients' waiting times are obviously shorter, hence there are {fewer} patients on the waiting list and the action spaces of the involved states are relatively small. {Consequently}, the CPU time significantly drops when $\lambda$ is larger than 0.8. In addition, Figure \ref{FigSubCPUBeta} shows the influence of $\beta$ on computational efficiency. {As illustrated in Figure \ref{FigSen}, the resulting policies for $10^{-5}\leqslant\beta<1$ and $1\leqslant\beta\leqslant10^5$ are significantly different. Accordingly, the tendencies of $\bar{t}$ in the two intervals are also different: $\bar{t}$ is increasing in $\beta$ for $\beta<1$ and decreasing in $\beta$ for $\beta\geqslant1$.}

\begin{figure}[!htb]
	\centering
	\subfigure
	{\includegraphics[width=2.5in]{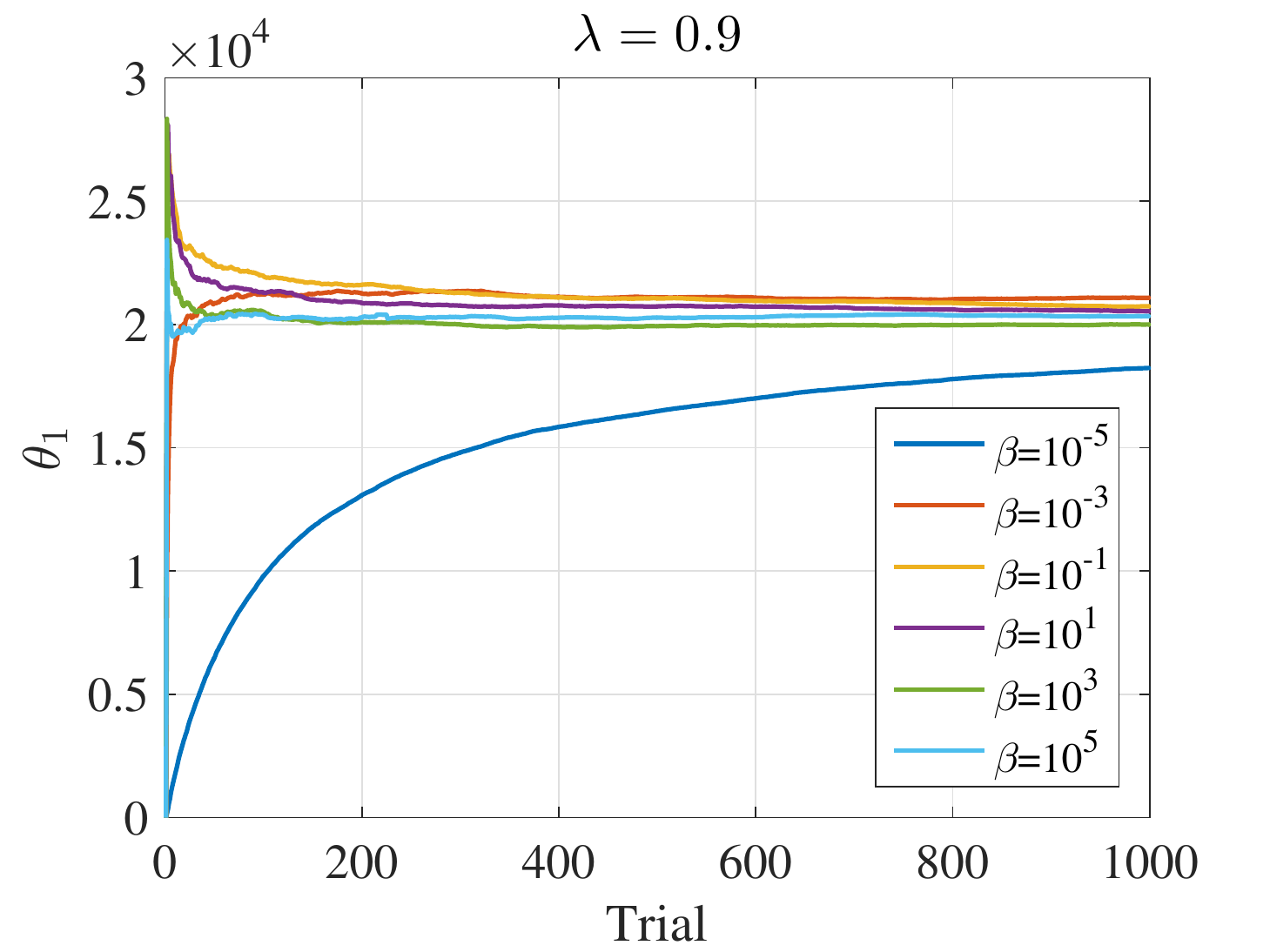}}
	\subfigure
	{\includegraphics[width=2.5in]{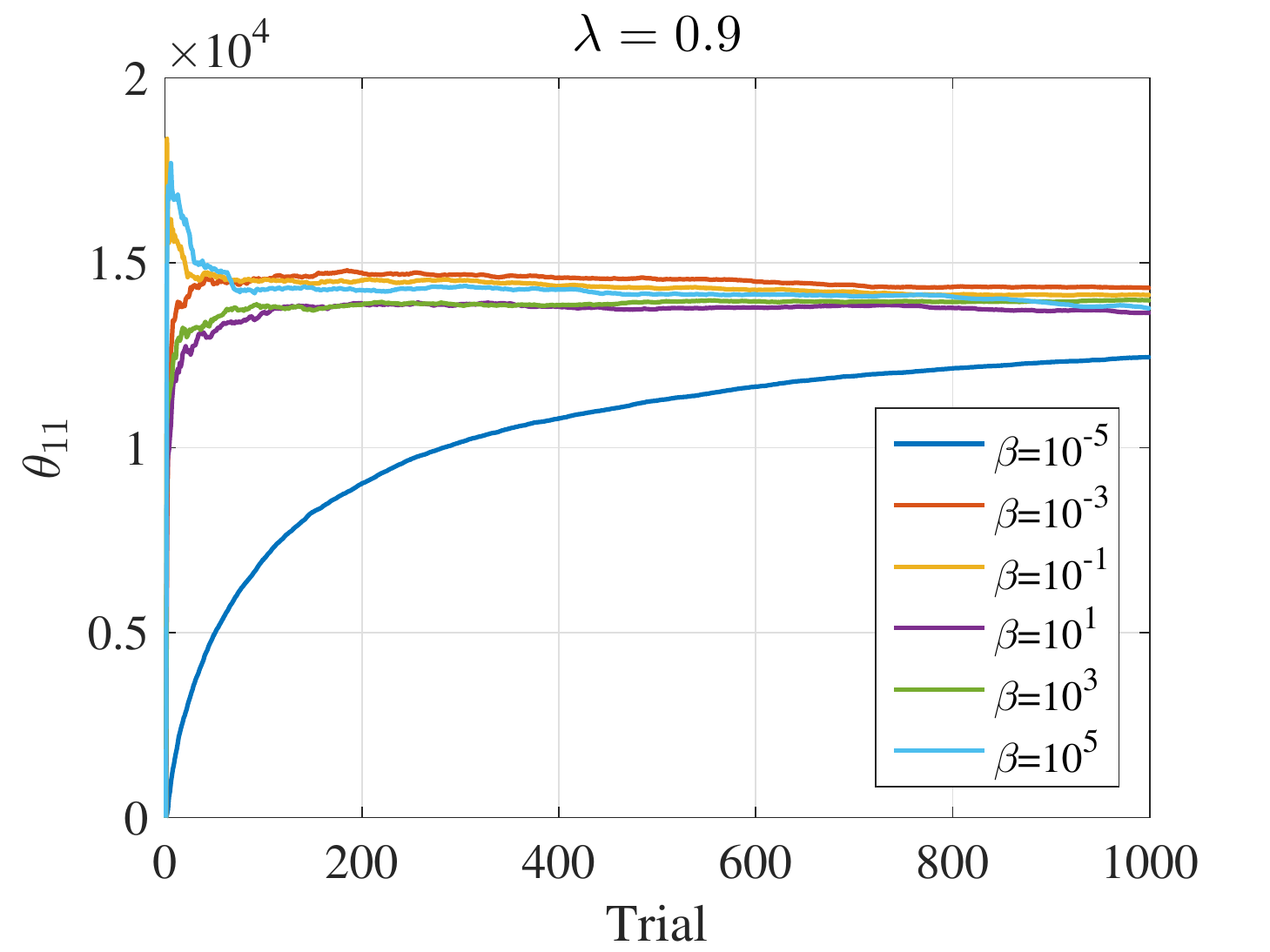}}
	\subfigure
	{\includegraphics[width=2.5in]{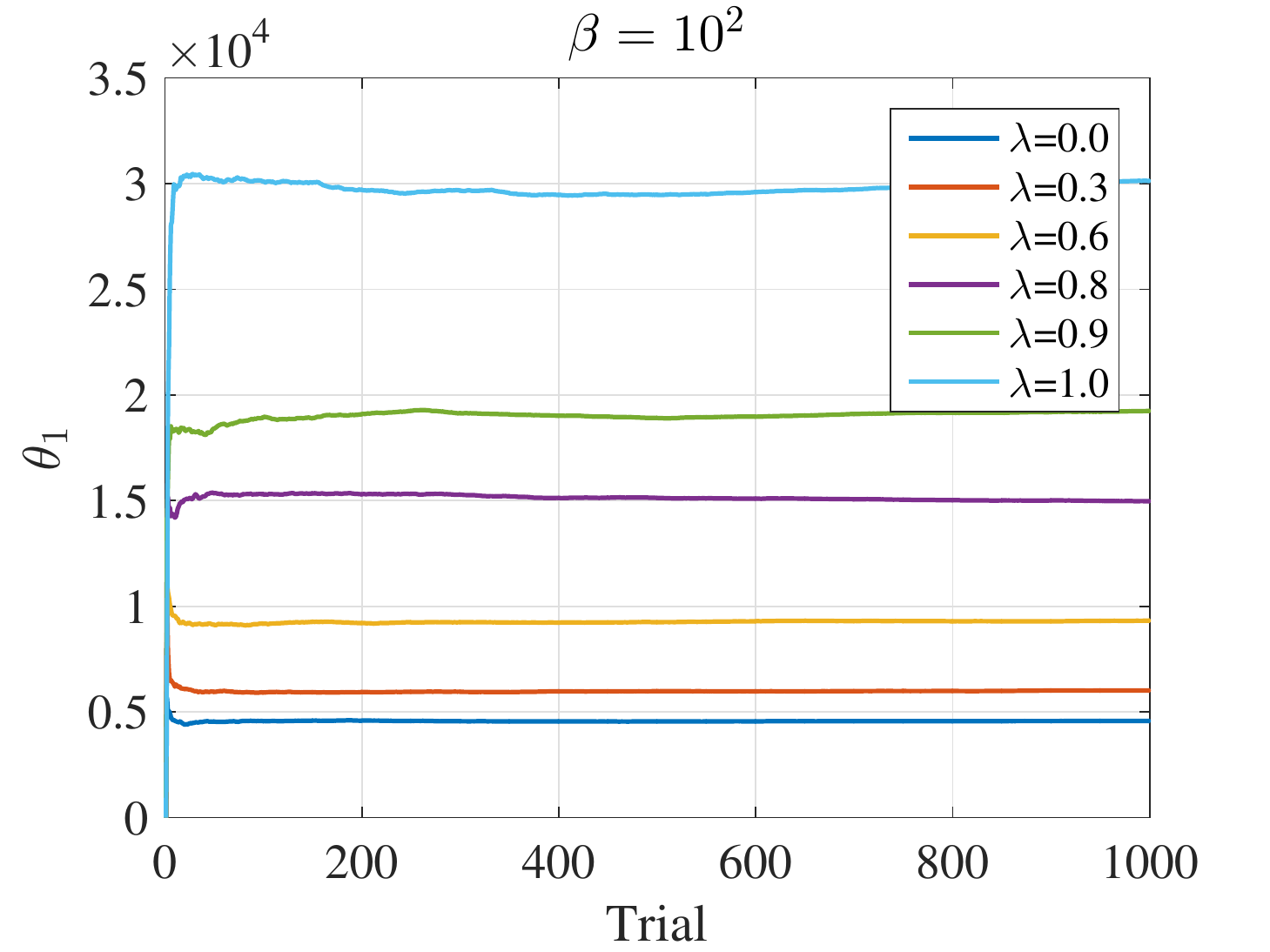}}
	\subfigure
	{\includegraphics[width=2.5in]{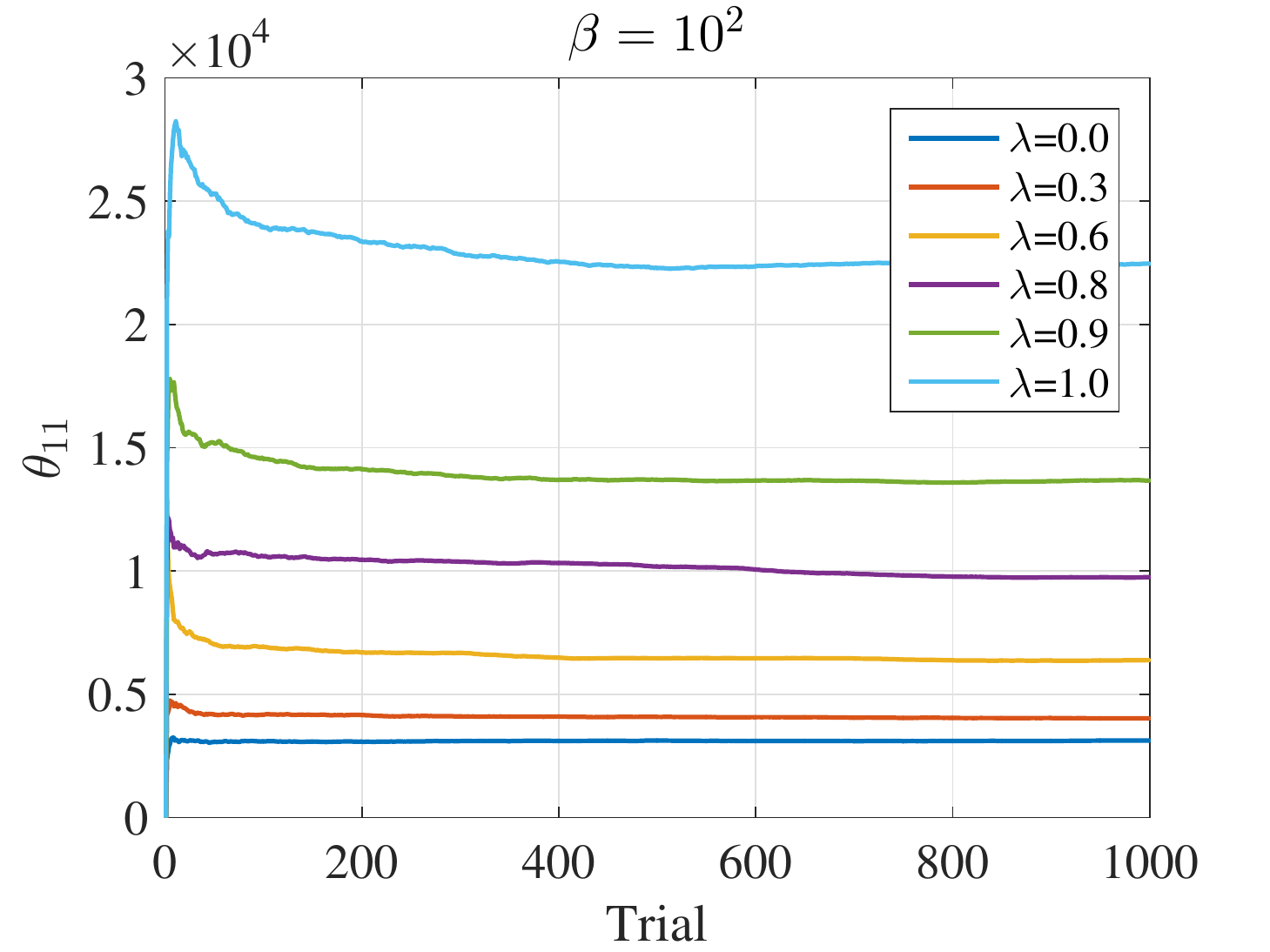}}
	\caption{Evolutions of $\Theta$ {when} solving the CABG problem with different values of $\lambda$ and $\beta$}\label{FigTheta}
\end{figure}
\begin{figure}[!htbp]
	\centering
	\subfigure[]
	{\label{FigSubCPULambda}\includegraphics[width=2.5in]{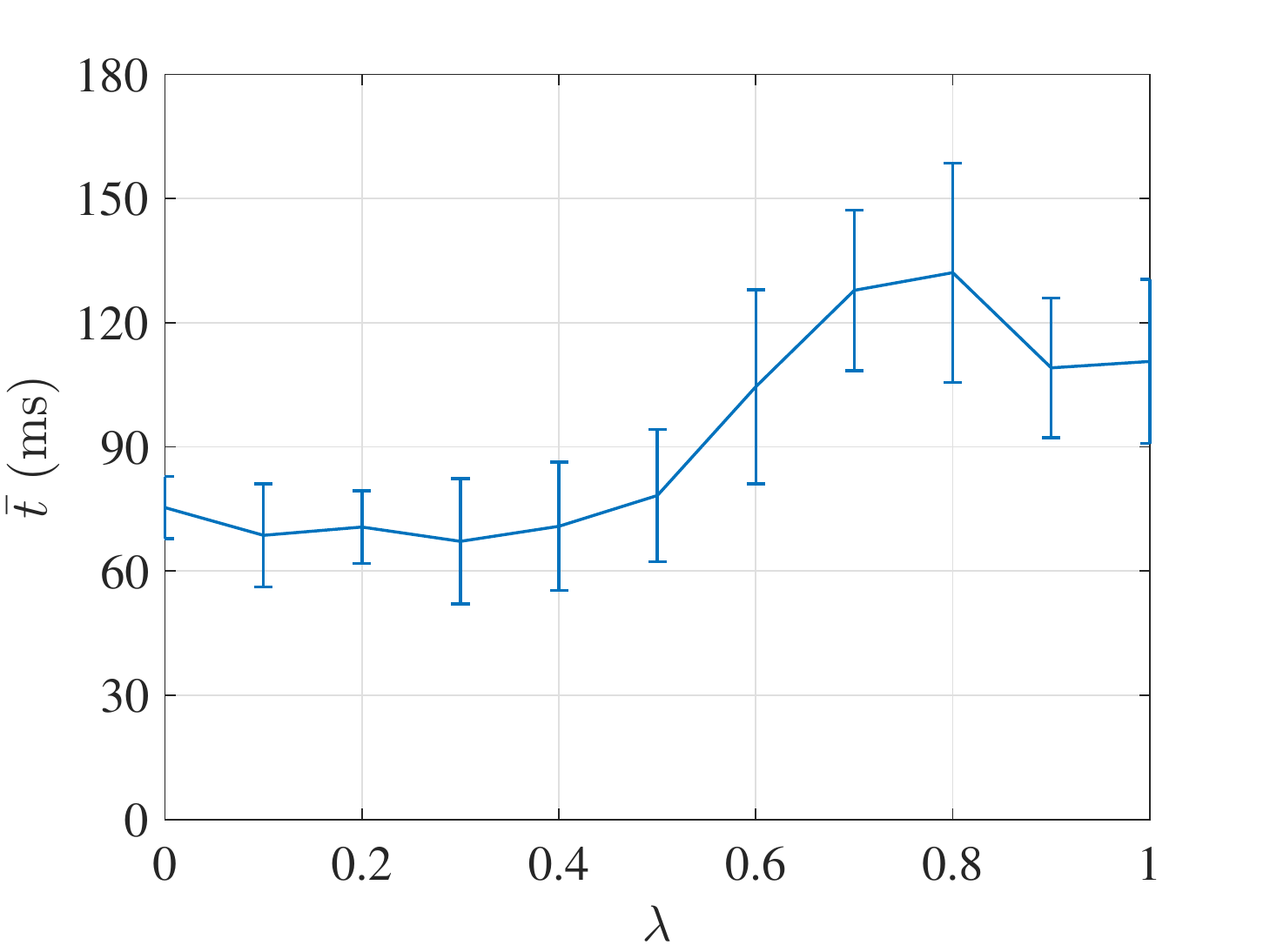}}
	\subfigure[]
	{\label{FigSubCPUBeta}\includegraphics[width=2.5in]{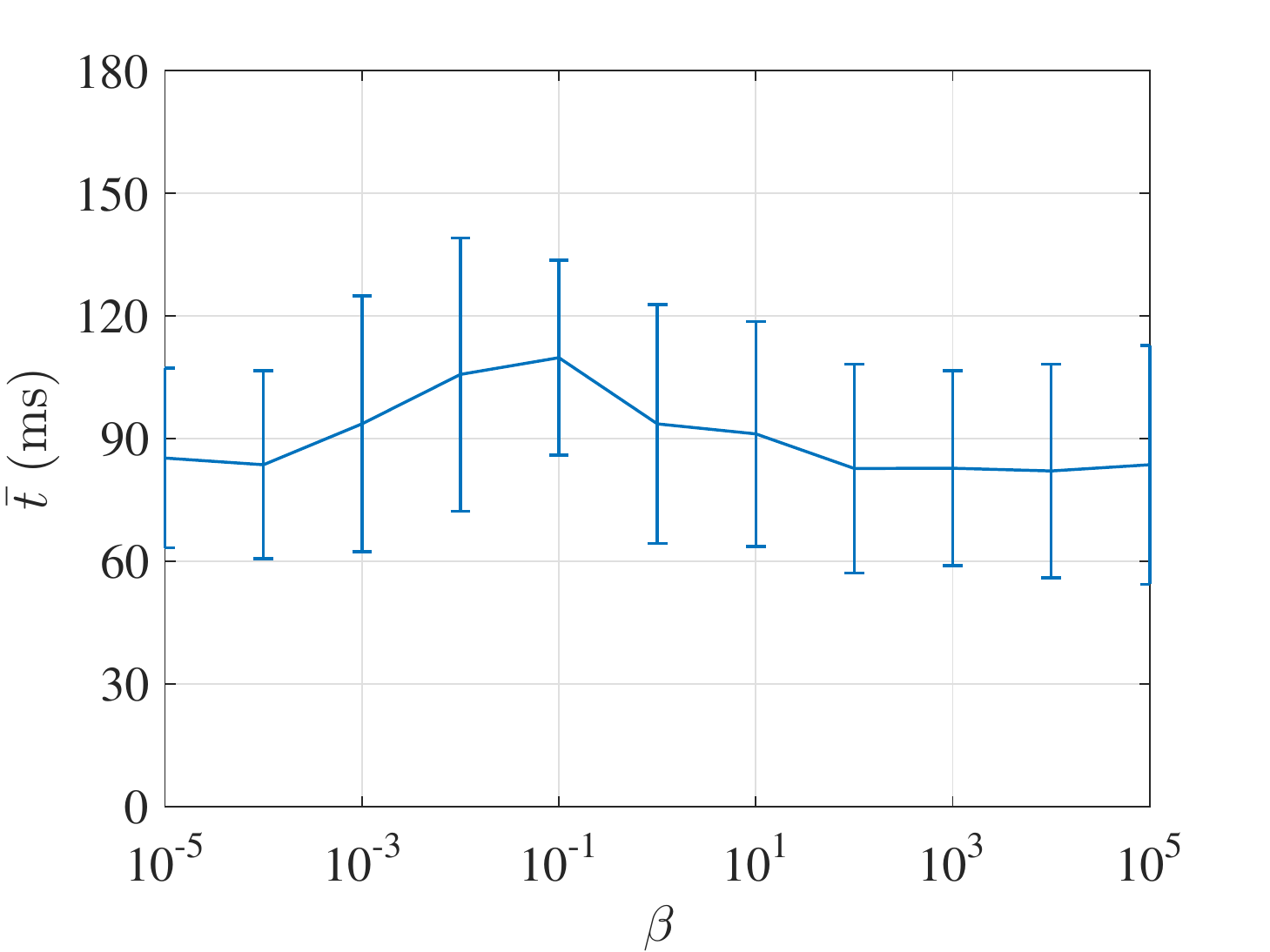}}
	\caption{Comparison of CPU time consumed by {the} ADP* algorithm {when used} with different values of $\lambda$ and $\beta$}\label{FigCPU}
\end{figure}

\subsection{Experimental results of a {realistically sized} multi-specialty patient admission control problem}\label{Exp3}
{In order to validate the proposed algorithms' capability to solve realistically sized problems, we consider a realistic multi-specialty patient admission control problem that is derived from \shortciteA{min2010scheduling} and \shortciteA{neyshabouri2017two} (with some reasonable modifications). The configuration of specialties is presented in Table \ref{TabConfMulti}.} {The regular OR capacity $B_j$ listed in the last column of Table \ref{TabConfMulti} is computed according to the MSS presented in Figure 1 of \shortciteA{min2010scheduling}. The MSS specifies the assignments of 32 surgical blocks to 9 specialties over a one-week period. Then, given that the regular time length of each surgical block is 8 hours, the values of $B_j$ can be computed. The} unit costs are $[{c_b},{c_d},c_o,{c_e}]=[50,200,1000,1000]$ and {the} regular SICU capacity is 105 bed-days per week (15 recovery beds). {The} estimated availability rates for ORs and recovery beds are $\rho_1=\rho_2=0.6$. According to Table \ref{TabConfMulti}, we can calculate that the size of the state space is as large as $2.14\times10^{176}$. Given that this multi-specialty problem is on a much larger scale than the {previously solved} test problems, the length of the simulation is shortened to 100 weeks. ADP* is employed as the solution technique and a myopic policy is computed as the benchmark. The parameters of the ADP* algorithm are {arbitrarily set as $\lambda=0.5$, $\beta=1$}, $N=25$, and $\epsilon=0.01$. 

\begin{table}[!htb]
	\centering
	\small
	\caption{Problem configuration {for} the multi-specialty test problem}\label{TabConfMulti}
	\begin{threeparttable}
	\begin{tabular}{cccrrrrccccr}
		\hline
		Specialty & $j$ & $v_j$ & $u$ & \multicolumn{1}{c}{$W_{ju}$} & \multicolumn{1}{c}{$\bar{n}_{ju}$} & \multicolumn{1}{c}{$\max{\tilde{n}_{ju}}$} & \multicolumn{1}{c}{$\bar{d}_j$\tnote{*}} & $\sigma(d_j)\tnote{*}$ & $\bar{l}_j$\tnote{**} & $\sigma(l_j)$\tnote{**} & \multicolumn{1}{c}{$B_j$\tnote{*}} \\
		\hline
		\vspace{0.12cm}
		ENT & 1 & 1 & 1 & 20 & 10.00 & 25 & 1.23 & 0.38 & 0.10 & 0.10 & 48.00 \\
		\vspace{-0.08cm}
		OBGYN & 2 & 2 & 1 & 15 & 4.00 & 15 & 1.43 & 0.44 & 2.00 & 2.00 & 24.00 \\
		\vspace{0.12cm}
		& &  & 3 & 6 & 0.50 & 4 &&&&& \\
		\vspace{-0.08cm}
		ORTHO & 3 & 2 & 1 & 15 & 10.00 & 25 & 1.78 & 0.54 & 1.50 & 1.50 & 48.00 \\
		\vspace{0.12cm}
		& &  & 3 & 6 & 2.00 & 10 &&&&& \\
		\vspace{0.12cm}
		NEURO & 4 & 5 & 1 & 8 & 2.50 & 12 & 2.67 & 1.65 & 2.00 & 2.00 & 8.00 \\
		\vspace{-0.08cm}
		GEN & 5 & 1 & 1 & 20 & 9.00 & 20 & 1.55 & 0.67 & 0.05 & 0.05 & 64.00 \\
		\vspace{0.12cm}
		& &  & 2 & 15 & 2.00 & 10 &&&&& \\
		\vspace{0.12cm}
		OPHTH & 6 & 2 & 1 & 15 & 1.50 & 8 & 0.63 & 0.10 & 0.05 & 0.05 & 32.00 \\
		\vspace{-0.08cm}
		VASCULAR & 7 & 4 & 1 & 10 & 1.00 & 6 & 2.00 & 1.03 & 3.50 & 3.50 & 16.00 \\
		\vspace{-0.08cm}
		& && 2 & 5 & 2.50 & 12 &&&&& \\
		\vspace{0.12cm}
		& && 4 & 2 & 0.50 & 4 &&&&& \\
		\vspace{-0.08cm}
		CARDIAC & 8 & 5 & 1 & 8 & 0.25 & 3 & 4.00 & 2.95 & 2.00 & 2.00 & 8.00 \\
		\vspace{-0.08cm}
		& && 2 & 3 & 1.25 & 7 &&&&& \\
		\vspace{0.12cm}
		& && 6 & 1 & 0.50 & 4 &&&&& \\
		\vspace{-0.08cm}
		UROLOGY & 9 & 3 & 1 & 12 & 2.00 & 10 & 1.07 & 0.75 & 0.80 & 0.80 & 8.00 \\
		& &  & 2 & 6 & 0.50 & 4 &&&&& \\
		\hline
	\end{tabular}
	\begin{tablenotes}
		\item * unit: hours; ** unit: days
	\end{tablenotes}
	\end{threeparttable}
\end{table}
{Tables \ref{tabMultiWT} and \ref{tabMultiTot}} present the experimental results for the multi-specialty problem in detail. {Table \ref{tabMultiWT}} shows that, in comparison with the myopic policy ($\gamma=0.00$), {the ADP* policy ($\gamma=0.99$)} significantly shortens patients’ waiting times, but leads to a slight increase in the overuse of ORs. In the specialties where the OR capacities are {relatively} sufficient and { patients do not spend much} time in the ORs and SICU, such as ENT and OPHTH, {patients wait for} one week at most, and there {are} few instances where ORs are overused. In comparison, in the specialties NEURO, CARDIAC and VASCULAR, where {the OR capacities are limited} and expectations of surgery duration and LOS are relatively large, many surgeries are delayed for more than one week and there is {greater} overuse of the ORs. From {Table \ref{tabMultiTot}} it can be seen that, compared to the myopic policy, ADP* reduces the {total} cost by {two thirds}, saves 9.03\% of CPU time, and increases the overuse of ORs by 1.378 hours per week on average. 

\begin{table}[!htb]
	\centering
	
	\scriptsize
	\caption{Experimental results of the multi-specialty problem: patients' waiting time and overtime of ORs}\label{tabMultiWT}
	\begin{adjustbox}{center}
	\addtolength\tabcolsep{-0.3em}
	\begin{tabular}{@{\extracolsep{3pt}}ccccccccccccccccccc}
		\hline
		\multirow{2}{*}{$\gamma$} & \multirow{2}{*}{Variable} & \multirow{2}{*}{ENT} & \multicolumn{2}{c}{OBGYN} & \multicolumn{2}{c}{ORTHO} & \multirow{2}{*}{NEURO} & \multicolumn{2}{c}{GEN} & \multirow{2}{*}{OPHTH} & \multicolumn{3}{c}{VASCULAR} & \multicolumn{3}{c}{CARDIAC} & \multicolumn{2}{c}{UROLOGY} \\\cline{4-5}\cline{6-7}\cline{9-10}\cline{12-14}\cline{15-17}\cline{18-19}
		&  &  & $u=1$ & $u=3$ & $u=1$ & $u=3$ &  & $u=1$ & $u=2$ &  & $u=1$ & $u=2$ & $u=4$ & $u=1$ & $u=2$ & $u=6$ & $u=1$ & $u=2$ \\\hline
		\vspace{-0.08cm}
		0.00 & $\bar{\omega}_{ju}$ & 1.000 & 1.740 & 1.000 & 1.427 & 1.000 & 3.900 & 1.000 & 1.000 & 1.000 & 2.700 & 1.511 & 1.000 & 6.063 & 2.955 & 1.000 & 1.218 & 1.000 \\
		\vspace{-0.08cm}
		& $\sigma(\omega_{ju})$ & 0.000 & 0.417 & 0.000 & 0.280 & 0.000 & 0.709 & 0.000 & 0.000 & 0.000 & 1.750 & 0.342 & 0.000 & 0.371 & 0.043 & 0.000 & 0.189 & 0.000 \\
		\vspace{-0.08cm}
		& $\bar{o}_j$ & 0.000	& \multicolumn{2}{c}{0.052} & \multicolumn{2}{c}{0.015} & 1.097 &	\multicolumn{2}{c}{0.086} &	0.000 &	\multicolumn{3}{c}{0.052} &	\multicolumn{3}{c}{0.369} &	 \multicolumn{2}{c}{0.000}  \\
		\vspace{0.2cm}
		& $\sigma(o_j)$ & 0.000 & \multicolumn{2}{c}{0.284} & \multicolumn{2}{c}{0.145} & 2.062 &	\multicolumn{2}{c}{0.604} &	0.000 &	\multicolumn{3}{c}{0.291} &	\multicolumn{3}{c}{0.941} &	 \multicolumn{2}{c}{0.000} \\
		\vspace{-0.08cm}
		0.99 & $\bar{\omega}_{ju}$ & 1.000 & 1.358 & 1.000 & 1.213 & 1.000 & 2.294 & 1.000 & 1.000 & 1.000 & 1.509 & 1.000 & 1.000 & 1.971 & 1.015 & 1.000 & 1.000 & 1.000 \\
		\vspace{-0.08cm}
		& $\sigma(\omega_{ju})$ & 0.000 & 0.436 & 0.000 & 0.167 & 0.000 & 1.270 & 0.000 & 0.000 & 0.000 & 0.341 & 0.000 & 0.000 & 0.885 & 0.015 & 0.000 & 0.000 & 0.000 \\
		\vspace{-0.08cm}
		& $\bar{o}_j$ & 0.000	& \multicolumn{2}{c}{0.004} & \multicolumn{2}{c}{0.216} & 1.304 &	\multicolumn{2}{c}{0.044} &	0.000 &	\multicolumn{3}{c}{0.740} &	\multicolumn{3}{c}{0.743} &	 \multicolumn{2}{c}{0.000}  \\
		& $\sigma(o_j)$ & 0.000 & \multicolumn{2}{c}{0.039} & \multicolumn{2}{c}{0.903} & 1.989 &	\multicolumn{2}{c}{0.229} &	0.000 &	\multicolumn{3}{c}{1.707} &	\multicolumn{3}{c}{7.738} &	 \multicolumn{2}{c}{0.000} \\
		\hline 
	\end{tabular}
	\end{adjustbox}

	\centering
	
	\small
	\caption{Experimental results of the multi-specialty problem: overuse of OR and SICU, cost, and CPU time}\label{tabMultiTot}
	\begin{tabular}{ccccccccr}
		\hline
		$\gamma$ & $\bar{o}$ & $\sigma(o)$ & $\bar{e}$ & $\sigma(e)$ & $\bar{c}$ & $\sigma(c)$ & $\bar{t}$ & \multicolumn{1}{c}{$\sigma(t)$} \\
		\hline	
		0.00 & 1.672 & 2.381 & 7.913 & 9.027 & 63715 & 21485 & 729,490 & 1078,424 \\
		0.99 & 3.050 & 3.271 & 8.690 & 9.518 & 21012 & 11487 & 663,610 & 736,206 \\
		\hline
	\end{tabular}
\end{table}

Figure \ref{FigMulti} compares the evolutions of the key performance measures under {the two} policies {(myopic and ADP*)}. Figure \ref{FigSubC} clearly shows that the ADP* policy leads to {lower costs over most of the weeks shown}. Moreover, since the patients' waiting times are shortened {under the ADP* policy}, the costs incurred by performing and delaying surgeries {(i.e., the patient-related costs)} are much lower and there are {significantly fewer} patients on the waiting list, as shown {in} Figures \ref{FigSubCw} and \ref{FigSubS}. In addition, we can observe from Figure \ref{FigSubT} that ADP* consumes more CPU time at the beginning of the simulation, because the parameter vector $\Theta$ is still far from {convergence}, whereas computing the myopic policy requires more CPU time at the end of the simulation, since the size of the waiting list and the number of actions to be evaluated continually increase over time.
\begin{figure}[!htb]
	\centering
	\subfigure[]{\label{FigSubC}\includegraphics[width=2.5in]{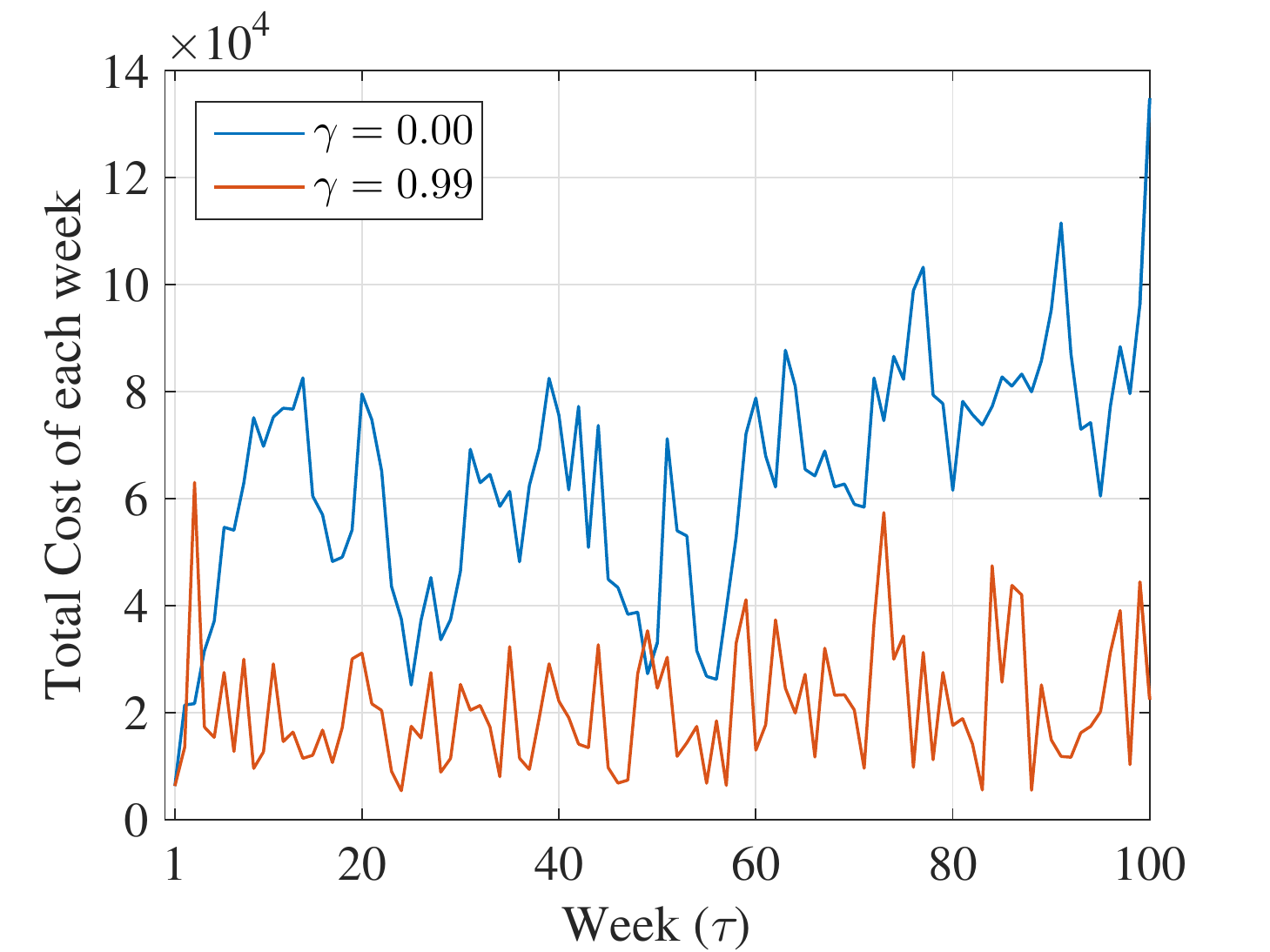}}
	\subfigure[]{\label{FigSubCw}\includegraphics[width=2.5in]{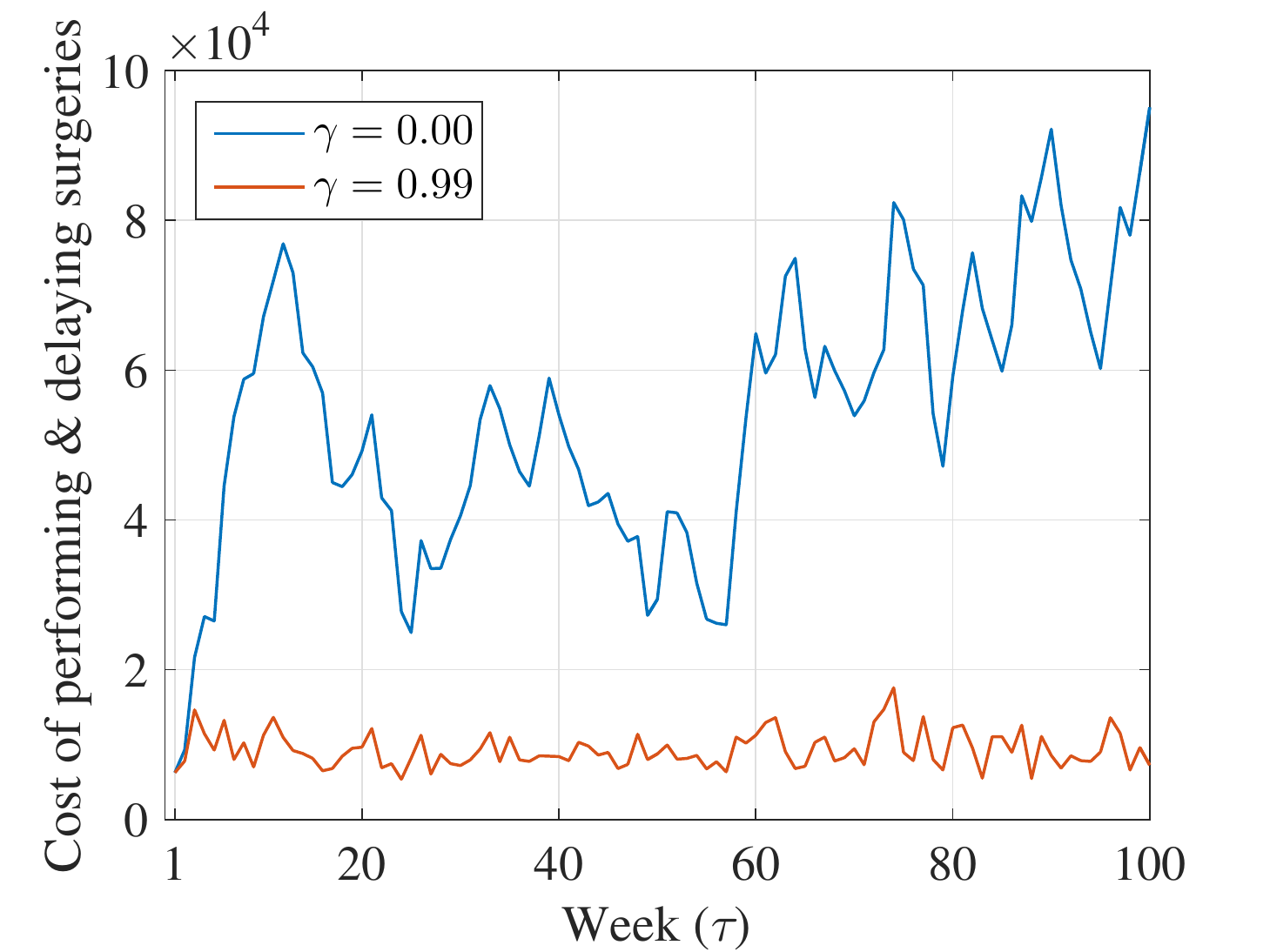}}
	\subfigure[]{\label{FigSubS}\includegraphics[width=2.5in]{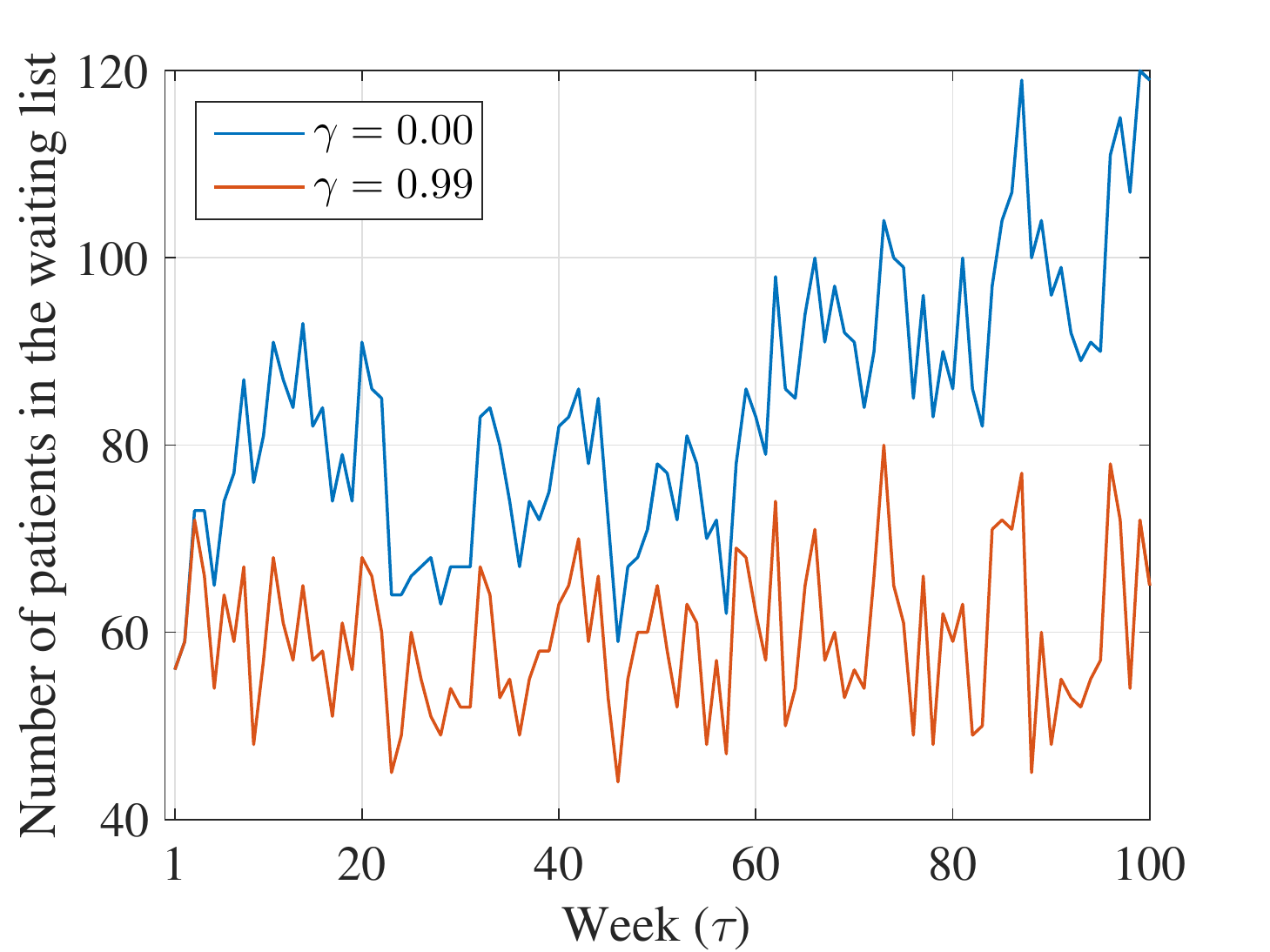}}
	\subfigure[]{\label{FigSubT}\includegraphics[width=2.5in]{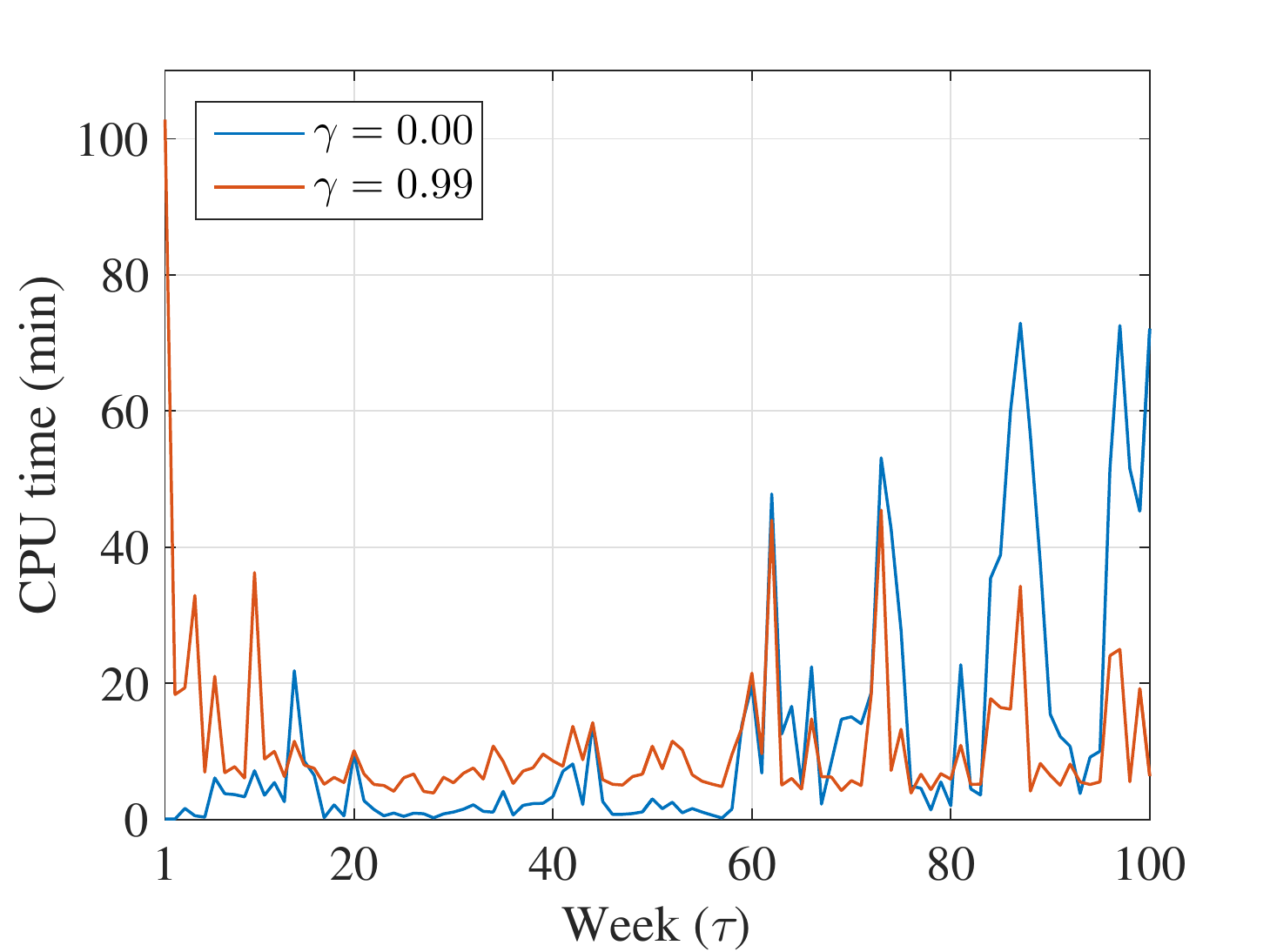}}
	\caption{Comparison of two policies for the multi-specialty {patient admission control} problem}\label{FigMulti}
\end{figure}

\section{Conclusion}\label{SecCon}
This work deals with the {admission control of elective patients}. It considers {uncertainty, dynamic patient priority {scores}, and multiple} capacity constraints. Sequential decisions are made at the end of each week that determine which patients on a waiting list will be treated in the following week. To guarantee equity and efficiency, each patient on the waiting list is assigned a dynamic priority {score} according to his/her urgency {coefficient}, actual waiting time{, and the relative importance of his/her specialty}. The studied {patient admission control} problem is formulated as an infinite-horizon MDP model. The objective is to minimize a cost function that assesses the costs of performing and delaying surgeries as well as the {expected} costs incurred by the over-utilizations of the ORs and SICU. Given that the problem scales of real-life situations are usually very large, solving the MDP model with traditional DP algorithms is computationally intractable. Therefore, we adopt two measures to tackle the curses of dimensionality: first, we perform an intensive structural analysis for the MDP model, reducing the number of actions to be evaluated (and thus drastically reducing the dimensionality of the action space); second, we develop a novel ADP algorithm based on RLS–TD($\lambda$) learning to {tackle} the dimensionalities of the state space and the outcome space. 

{In the numerical experiments}, multiple test problems are solved {to} validate the efficiency and accuracy of the proposed algorithms. {The experimental results of a small-sized test problem show that{, in comparison with conventional DP-based algorithms,} the ADP algorithm consumes significantly less CPU time and only leads to a slight increase in the total costs, {regardless of the changes in problem parameters. Further, our sensitivity analyses on the parameters of the ADP algorithm}, performed on a large-sized CABG test problem, reveal that the patients' waiting times and the total costs decrease in $\lambda$ and $\beta$, while the over-utilizations of ORs and SICU increase. For all the values of $\lambda$ and $\beta$, the resulting ADP policy significantly outperforms the myopic policy in terms of waiting times and total costs. Following this, we solve a multi-specialty patient admission control problem using the proposed ADP algorithm and validate its capability to tackle realistically sized problems. Our experimental results also show that Algorithm \ref{GreedyAction}, proposed in Section \ref{SecAna}, drastically reduces the number of actions to be evaluated and improves the computational efficiency of all the employed algorithms.} 

{In the future, three aspects of this work can be extended for further research. First, {to model patient admission control problems more realistically, stochastic durations and violations of the maximum recommended waiting times can be explicitly incorporated into the MDP model. As the model properties proved in this paper may no longer hold true and the computational burden will increase, further structural analyses and more efficient solution approaches should be studied and applied to the extended MDP model.} Second, we find in this work that the RLS–TD($\lambda$)-based ADP algorithm tends to increase the over-utilization of surgical resources and reduce patients' waiting times. However, the exact impact of the ADP algorithm on the nature of the resulting policy is unclear, since the interactions between the MDP model and the ADP algorithm are highly complex. Hence, an in-depth theoretical analysis {in this respect} can be studied as part of future research.} {Third, the MDP model {for patient admission control} can easily be combined with a mathematical programming model which optimizes the intra-week surgery-to-block assignments. For example, \shortciteA{zhang2019two} have addressed a relatively simple advance surgery scheduling problem (single specialty, single OR) using MDP {and} stochastic programming. {They} show that the resulting surgery schedules significantly outperform those of a pure stochastic programming model. Obviously, the MDP model {proposed in this paper} has big potential to improve the long-term quality of advance surgery scheduling, as the {future costs} are mostly ignored in the commonly used pure mathematical programming models.}

{
\section*{Acknowledgements}
The authors gratefully acknowledge the financial support granted by the China Scholarship Council (CSC, Grant No. 201604490106).}

\bibliographystyle{apacite}
\bibliography{ManuscriptCOR}

\begin{thebibliography}{}

\bibitem [\protect \citeauthoryear {%
Addis%
, Carello%
, Grosso%
\BCBL {}\ \BBA {} T{\`a}nfani%
}{%
Addis%
\ \protect \BOthers {.}}{%
{\protect \APACyear {2016}}%
}]{%
addis2016operating}
\APACinsertmetastar {%
addis2016operating}%
\begin{APACrefauthors}%
Addis, B.%
, Carello, G.%
, Grosso, A.%
\BCBL {}\ \BBA {} T{\`a}nfani, E.%
\end{APACrefauthors}%
\unskip\
\newblock
\APACrefYearMonthDay{2016}{}{}.
\newblock
{\BBOQ}\APACrefatitle {Operating room scheduling and rescheduling: a rolling
  horizon approach} {Operating room scheduling and rescheduling: a rolling
  horizon approach}.{\BBCQ}
\newblock
\APACjournalVolNumPages{Flexible Services and Manufacturing
  Journal}{28}{1-2}{206--232}.
\PrintBackRefs{\CurrentBib}

\bibitem [\protect \citeauthoryear {%
Agnetis%
\ \protect \BOthers {.}}{%
Agnetis%
\ \protect \BOthers {.}}{%
{\protect \APACyear {2012}}%
}]{%
agnetis2012long}
\APACinsertmetastar {%
agnetis2012long}%
\begin{APACrefauthors}%
Agnetis, A.%
, Coppi, A.%
, Corsini, M.%
, Dellino, G.%
, Meloni, C.%
\BCBL {}\ \BBA {} Pranzo, M.%
\end{APACrefauthors}%
\unskip\
\newblock
\APACrefYearMonthDay{2012}{}{}.
\newblock
{\BBOQ}\APACrefatitle {Long term evaluation of operating theater planning
  policies} {Long term evaluation of operating theater planning
  policies}.{\BBCQ}
\newblock
\APACjournalVolNumPages{Operations Research for Health Care}{1}{4}{95--104}.
\PrintBackRefs{\CurrentBib}

\bibitem [\protect \citeauthoryear {%
Agnetis%
\ \protect \BOthers {.}}{%
Agnetis%
\ \protect \BOthers {.}}{%
{\protect \APACyear {2014}}%
}]{%
agnetis2014decomposition}
\APACinsertmetastar {%
agnetis2014decomposition}%
\begin{APACrefauthors}%
Agnetis, A.%
, Coppi, A.%
, Corsini, M.%
, Dellino, G.%
, Meloni, C.%
\BCBL {}\ \BBA {} Pranzo, M.%
\end{APACrefauthors}%
\unskip\
\newblock
\APACrefYearMonthDay{2014}{}{}.
\newblock
{\BBOQ}\APACrefatitle {A decomposition approach for the combined master
  surgical schedule and surgical case assignment problems} {A decomposition
  approach for the combined master surgical schedule and surgical case
  assignment problems}.{\BBCQ}
\newblock
\APACjournalVolNumPages{Health care management science}{17}{1}{49--59}.
\PrintBackRefs{\CurrentBib}

\bibitem [\protect \citeauthoryear {%
Aringhieri%
, Landa%
, Soriano%
, T{\`a}nfani%
\BCBL {}\ \BBA {} Testi%
}{%
Aringhieri%
\ \protect \BOthers {.}}{%
{\protect \APACyear {2015}}%
}]{%
aringhieri2015two}
\APACinsertmetastar {%
aringhieri2015two}%
\begin{APACrefauthors}%
Aringhieri, R.%
, Landa, P.%
, Soriano, P.%
, T{\`a}nfani, E.%
\BCBL {}\ \BBA {} Testi, A.%
\end{APACrefauthors}%
\unskip\
\newblock
\APACrefYearMonthDay{2015}{}{}.
\newblock
{\BBOQ}\APACrefatitle {A two level metaheuristic for the operating room
  scheduling and assignment problem} {A two level metaheuristic for the
  operating room scheduling and assignment problem}.{\BBCQ}
\newblock
\APACjournalVolNumPages{Computers \& Operations Research}{54}{}{21--34}.
\PrintBackRefs{\CurrentBib}

\bibitem [\protect \citeauthoryear {%
Astaraky%
\ \BBA {} Patrick%
}{%
Astaraky%
\ \BBA {} Patrick%
}{%
{\protect \APACyear {2015}}%
}]{%
astaraky2015simulation}
\APACinsertmetastar {%
astaraky2015simulation}%
\begin{APACrefauthors}%
Astaraky, D.%
\BCBT {}\ \BBA {} Patrick, J.%
\end{APACrefauthors}%
\unskip\
\newblock
\APACrefYearMonthDay{2015}{}{}.
\newblock
{\BBOQ}\APACrefatitle {A simulation based approximate dynamic programming
  approach to multi-class, multi-resource surgical scheduling} {A simulation
  based approximate dynamic programming approach to multi-class, multi-resource
  surgical scheduling}.{\BBCQ}
\newblock
\APACjournalVolNumPages{European Journal of Operational
  Research}{245}{1}{309--319}.
\PrintBackRefs{\CurrentBib}

\bibitem [\protect \citeauthoryear {%
{Australian Institute of Health and Welfare}%
}{%
{Australian Institute of Health and Welfare}%
}{%
{\protect \APACyear {2018}}%
}]{%
AIHW}
\APACinsertmetastar {%
AIHW}%
\begin{APACrefauthors}%
{Australian Institute of Health and Welfare}.%
\end{APACrefauthors}%
\unskip\
\newblock
\APACrefYearMonthDay{2018}{}{}.
\newblock
\APACrefbtitle {Australia's health 2018.} {Australia's health 2018.}
\newblock
\APAChowpublished
  {\url{https://www.aihw.gov.au/getmedia/7c42913d-295f-4bc9-9c24-4e44eff4a04a/aihw-aus-221.pdf.aspx?inline=true}
  (retrieved May 17, 2019)}.
\PrintBackRefs{\CurrentBib}

\bibitem [\protect \citeauthoryear {%
Banditori%
, Cappanera%
\BCBL {}\ \BBA {} Visintin%
}{%
Banditori%
\ \protect \BOthers {.}}{%
{\protect \APACyear {2013}}%
}]{%
banditori2013combined}
\APACinsertmetastar {%
banditori2013combined}%
\begin{APACrefauthors}%
Banditori, C.%
, Cappanera, P.%
\BCBL {}\ \BBA {} Visintin, F.%
\end{APACrefauthors}%
\unskip\
\newblock
\APACrefYearMonthDay{2013}{}{}.
\newblock
{\BBOQ}\APACrefatitle {A combined optimization--simulation approach to the
  master surgical scheduling problem} {A combined optimization--simulation
  approach to the master surgical scheduling problem}.{\BBCQ}
\newblock
\APACjournalVolNumPages{IMA Journal of Management
  Mathematics}{24}{2}{155--187}.
\PrintBackRefs{\CurrentBib}

\bibitem [\protect \citeauthoryear {%
Barto%
, Bradtke%
\BCBL {}\ \BBA {} Singh%
}{%
Barto%
\ \protect \BOthers {.}}{%
{\protect \APACyear {1995}}%
}]{%
barto1995learning}
\APACinsertmetastar {%
barto1995learning}%
\begin{APACrefauthors}%
Barto, A\BPBI G.%
, Bradtke, S\BPBI J.%
\BCBL {}\ \BBA {} Singh, S\BPBI P.%
\end{APACrefauthors}%
\unskip\
\newblock
\APACrefYearMonthDay{1995}{}{}.
\newblock
{\BBOQ}\APACrefatitle {Learning to act using real-time dynamic programming}
  {Learning to act using real-time dynamic programming}.{\BBCQ}
\newblock
\APACjournalVolNumPages{Artificial Intelligence}{72}{1-2}{81--138}.
\PrintBackRefs{\CurrentBib}

\bibitem [\protect \citeauthoryear {%
Bonet%
\ \BBA {} Geffner%
}{%
Bonet%
\ \BBA {} Geffner%
}{%
{\protect \APACyear {2003}}%
}]{%
bonet2003labeled}
\APACinsertmetastar {%
bonet2003labeled}%
\begin{APACrefauthors}%
Bonet, B.%
\BCBT {}\ \BBA {} Geffner, H.%
\end{APACrefauthors}%
\unskip\
\newblock
\APACrefYearMonthDay{2003}{}{}.
\newblock
{\BBOQ}\APACrefatitle {Labeled {RTDP}: Improving the Convergence of Real-Time
  Dynamic Programming.} {Labeled {RTDP}: Improving the convergence of real-time
  dynamic programming.}{\BBCQ}
\newblock
\BIn{} \APACrefbtitle {{Proceedings of Thirteenth International Conference on
  Automated Planning and Scheduling}} {{Proceedings of Thirteenth International
  Conference on Automated Planning and Scheduling}}\ (\BVOL~3, \BPGS\ 12--21).
\PrintBackRefs{\CurrentBib}

\bibitem [\protect \citeauthoryear {%
Boyan%
}{%
Boyan%
}{%
{\protect \APACyear {2002}}%
}]{%
boyan2002technical}
\APACinsertmetastar {%
boyan2002technical}%
\begin{APACrefauthors}%
Boyan, J\BPBI A.%
\end{APACrefauthors}%
\unskip\
\newblock
\APACrefYearMonthDay{2002}{}{}.
\newblock
{\BBOQ}\APACrefatitle {Technical update: Least-squares temporal difference
  learning} {Technical update: Least-squares temporal difference
  learning}.{\BBCQ}
\newblock
\APACjournalVolNumPages{Machine learning}{49}{2-3}{233--246}.
\PrintBackRefs{\CurrentBib}

\bibitem [\protect \citeauthoryear {%
{Centers for Medicare \& Medicaid Services}%
}{%
{Centers for Medicare \& Medicaid Services}%
}{%
{\protect \APACyear {2018}}%
}]{%
CMS}
\APACinsertmetastar {%
CMS}%
\begin{APACrefauthors}%
{Centers for Medicare \& Medicaid Services}.%
\end{APACrefauthors}%
\unskip\
\newblock
\APACrefYearMonthDay{2018}{}{}.
\newblock
\APACrefbtitle {National Health Expenditures 2017 Highlights.} {National health
  expenditures 2017 highlights.}
\newblock
\APAChowpublished
  {\url{https://www.cms.gov/Research-Statistics-Data-and-Systems/Statistics-Trends-and-Reports/NationalHealthExpendData/Downloads/highlights.pdf}
  (retrieved May 17, 2019)}.
\PrintBackRefs{\CurrentBib}

\bibitem [\protect \citeauthoryear {%
Denton%
, Miller%
, Balasubramanian%
\BCBL {}\ \BBA {} Huschka%
}{%
Denton%
\ \protect \BOthers {.}}{%
{\protect \APACyear {2010}}%
}]{%
denton2010optimal}
\APACinsertmetastar {%
denton2010optimal}%
\begin{APACrefauthors}%
Denton, B.%
, Miller, A.%
, Balasubramanian, H.%
\BCBL {}\ \BBA {} Huschka, T.%
\end{APACrefauthors}%
\unskip\
\newblock
\APACrefYearMonthDay{2010}{}{}.
\newblock
{\BBOQ}\APACrefatitle {Optimal allocation of surgery blocks to operating rooms
  under uncertainty} {Optimal allocation of surgery blocks to operating rooms
  under uncertainty}.{\BBCQ}
\newblock
\APACjournalVolNumPages{Operations research}{58}{4-part-1}{802--816}.
\PrintBackRefs{\CurrentBib}

\bibitem [\protect \citeauthoryear {%
Denton%
, Viapiano%
\BCBL {}\ \BBA {} Vogl%
}{%
Denton%
\ \protect \BOthers {.}}{%
{\protect \APACyear {2007}}%
}]{%
denton2007optimization}
\APACinsertmetastar {%
denton2007optimization}%
\begin{APACrefauthors}%
Denton, B.%
, Viapiano, J.%
\BCBL {}\ \BBA {} Vogl, A.%
\end{APACrefauthors}%
\unskip\
\newblock
\APACrefYearMonthDay{2007}{}{}.
\newblock
{\BBOQ}\APACrefatitle {Optimization of surgery sequencing and scheduling
  decisions under uncertainty} {Optimization of surgery sequencing and
  scheduling decisions under uncertainty}.{\BBCQ}
\newblock
\APACjournalVolNumPages{Health care management science}{10}{1}{13--24}.
\PrintBackRefs{\CurrentBib}

\bibitem [\protect \citeauthoryear {%
Dios%
, Molina-Pariente%
, Fernandez-Viagas%
, Andrade-Pineda%
\BCBL {}\ \BBA {} Framinan%
}{%
Dios%
\ \protect \BOthers {.}}{%
{\protect \APACyear {2015}}%
}]{%
dios2015decision}
\APACinsertmetastar {%
dios2015decision}%
\begin{APACrefauthors}%
Dios, M.%
, Molina-Pariente, J\BPBI M.%
, Fernandez-Viagas, V.%
, Andrade-Pineda, J\BPBI L.%
\BCBL {}\ \BBA {} Framinan, J\BPBI M.%
\end{APACrefauthors}%
\unskip\
\newblock
\APACrefYearMonthDay{2015}{}{}.
\newblock
{\BBOQ}\APACrefatitle {A decision support system for operating room scheduling}
  {A decision support system for operating room scheduling}.{\BBCQ}
\newblock
\APACjournalVolNumPages{Computers \& Industrial Engineering}{88}{}{430--443}.
\PrintBackRefs{\CurrentBib}

\bibitem [\protect \citeauthoryear {%
Duma%
\ \BBA {} Aringhieri%
}{%
Duma%
\ \BBA {} Aringhieri%
}{%
{\protect \APACyear {2019}}%
}]{%
duma2019management}
\APACinsertmetastar {%
duma2019management}%
\begin{APACrefauthors}%
Duma, D.%
\BCBT {}\ \BBA {} Aringhieri, R.%
\end{APACrefauthors}%
\unskip\
\newblock
\APACrefYearMonthDay{2019}{}{}.
\newblock
{\BBOQ}\APACrefatitle {The management of non-elective patients: shared vs.
  dedicated policies} {The management of non-elective patients: shared vs.
  dedicated policies}.{\BBCQ}
\newblock
\APACjournalVolNumPages{Omega}{83}{}{199--212}.
\PrintBackRefs{\CurrentBib}

\bibitem [\protect \citeauthoryear {%
Fei%
, Chu%
\BCBL {}\ \BBA {} Meskens%
}{%
Fei%
, Chu%
\BCBL {}\ \BBA {} Meskens%
}{%
{\protect \APACyear {2009}}%
}]{%
fei2009solving}
\APACinsertmetastar {%
fei2009solving}%
\begin{APACrefauthors}%
Fei, H.%
, Chu, C.%
\BCBL {}\ \BBA {} Meskens, N.%
\end{APACrefauthors}%
\unskip\
\newblock
\APACrefYearMonthDay{2009}{}{}.
\newblock
{\BBOQ}\APACrefatitle {Solving a tactical operating room planning problem by a
  column-generation-based heuristic procedure with four criteria} {Solving a
  tactical operating room planning problem by a column-generation-based
  heuristic procedure with four criteria}.{\BBCQ}
\newblock
\APACjournalVolNumPages{Annals of Operations Research}{166}{1}{91}.
\PrintBackRefs{\CurrentBib}

\bibitem [\protect \citeauthoryear {%
Fei%
, Chu%
, Meskens%
\BCBL {}\ \BBA {} Artiba%
}{%
Fei%
\ \protect \BOthers {.}}{%
{\protect \APACyear {2008}}%
}]{%
fei2008solving}
\APACinsertmetastar {%
fei2008solving}%
\begin{APACrefauthors}%
Fei, H.%
, Chu, C.%
, Meskens, N.%
\BCBL {}\ \BBA {} Artiba, A.%
\end{APACrefauthors}%
\unskip\
\newblock
\APACrefYearMonthDay{2008}{}{}.
\newblock
{\BBOQ}\APACrefatitle {Solving surgical cases assignment problem by a
  branch-and-price approach} {Solving surgical cases assignment problem by a
  branch-and-price approach}.{\BBCQ}
\newblock
\APACjournalVolNumPages{International Journal of Production
  Economics}{112}{1}{96--108}.
\PrintBackRefs{\CurrentBib}

\bibitem [\protect \citeauthoryear {%
Fei%
, Meskens%
\BCBL {}\ \BBA {} Chu%
}{%
Fei%
\ \protect \BOthers {.}}{%
{\protect \APACyear {2010}}%
}]{%
fei2010planning}
\APACinsertmetastar {%
fei2010planning}%
\begin{APACrefauthors}%
Fei, H.%
, Meskens, N.%
\BCBL {}\ \BBA {} Chu, C.%
\end{APACrefauthors}%
\unskip\
\newblock
\APACrefYearMonthDay{2010}{}{}.
\newblock
{\BBOQ}\APACrefatitle {A planning and scheduling problem for an operating
  theatre using an open scheduling strategy} {A planning and scheduling problem
  for an operating theatre using an open scheduling strategy}.{\BBCQ}
\newblock
\APACjournalVolNumPages{Computers \& Industrial Engineering}{58}{2}{221--230}.
\PrintBackRefs{\CurrentBib}

\bibitem [\protect \citeauthoryear {%
Fei%
, Meskens%
, Combes%
\BCBL {}\ \BBA {} Chu%
}{%
Fei%
, Meskens%
\BCBL {}\ \protect \BOthers {.}}{%
{\protect \APACyear {2009}}%
}]{%
fei2009endoscopy}
\APACinsertmetastar {%
fei2009endoscopy}%
\begin{APACrefauthors}%
Fei, H.%
, Meskens, N.%
, Combes, C.%
\BCBL {}\ \BBA {} Chu, C.%
\end{APACrefauthors}%
\unskip\
\newblock
\APACrefYearMonthDay{2009}{}{}.
\newblock
{\BBOQ}\APACrefatitle {The endoscopy scheduling problem: A case study with two
  specialised operating rooms} {The endoscopy scheduling problem: A case study
  with two specialised operating rooms}.{\BBCQ}
\newblock
\APACjournalVolNumPages{International Journal of Production
  Economics}{120}{2}{452--462}.
\PrintBackRefs{\CurrentBib}

\bibitem [\protect \citeauthoryear {%
Ferrand%
, Magazine%
\BCBL {}\ \BBA {} Rao%
}{%
Ferrand%
\ \protect \BOthers {.}}{%
{\protect \APACyear {2010}}%
}]{%
ferrand2010comparing}
\APACinsertmetastar {%
ferrand2010comparing}%
\begin{APACrefauthors}%
Ferrand, Y.%
, Magazine, M.%
\BCBL {}\ \BBA {} Rao, U.%
\end{APACrefauthors}%
\unskip\
\newblock
\APACrefYearMonthDay{2010}{}{}.
\newblock
{\BBOQ}\APACrefatitle {Comparing two operating-room-allocation policies for
  elective and emergency surgeries} {Comparing two operating-room-allocation
  policies for elective and emergency surgeries}.{\BBCQ}
\newblock
\BIn{} \APACrefbtitle {{Proceedings of the 2010 Winter Simulation Conference}}
  {{Proceedings of the 2010 Winter Simulation Conference}}\ (\BPGS\
  2364--2374).
\PrintBackRefs{\CurrentBib}

\bibitem [\protect \citeauthoryear {%
Ferrand%
, Magazine%
\BCBL {}\ \BBA {} Rao%
}{%
Ferrand%
\ \protect \BOthers {.}}{%
{\protect \APACyear {2014}}%
}]{%
ferrand2014partially}
\APACinsertmetastar {%
ferrand2014partially}%
\begin{APACrefauthors}%
Ferrand, Y.%
, Magazine, M.%
\BCBL {}\ \BBA {} Rao, U.%
\end{APACrefauthors}%
\unskip\
\newblock
\APACrefYearMonthDay{2014}{}{}.
\newblock
{\BBOQ}\APACrefatitle {Partially flexible operating rooms for elective and
  emergency surgeries} {Partially flexible operating rooms for elective and
  emergency surgeries}.{\BBCQ}
\newblock
\APACjournalVolNumPages{Decision Sciences}{45}{5}{819--847}.
\PrintBackRefs{\CurrentBib}

\bibitem [\protect \citeauthoryear {%
Gerchak%
, Gupta%
\BCBL {}\ \BBA {} Henig%
}{%
Gerchak%
\ \protect \BOthers {.}}{%
{\protect \APACyear {1996}}%
}]{%
gerchak1996reservation}
\APACinsertmetastar {%
gerchak1996reservation}%
\begin{APACrefauthors}%
Gerchak, Y.%
, Gupta, D.%
\BCBL {}\ \BBA {} Henig, M.%
\end{APACrefauthors}%
\unskip\
\newblock
\APACrefYearMonthDay{1996}{}{}.
\newblock
{\BBOQ}\APACrefatitle {Reservation planning for elective surgery under
  uncertain demand for emergency surgery} {Reservation planning for elective
  surgery under uncertain demand for emergency surgery}.{\BBCQ}
\newblock
\APACjournalVolNumPages{Management Science}{42}{3}{321--334}.
\PrintBackRefs{\CurrentBib}

\bibitem [\protect \citeauthoryear {%
Gocgun%
\ \BBA {} Ghate%
}{%
Gocgun%
\ \BBA {} Ghate%
}{%
{\protect \APACyear {2012}}%
}]{%
gocgun2012lagrangian}
\APACinsertmetastar {%
gocgun2012lagrangian}%
\begin{APACrefauthors}%
Gocgun, Y.%
\BCBT {}\ \BBA {} Ghate, A.%
\end{APACrefauthors}%
\unskip\
\newblock
\APACrefYearMonthDay{2012}{}{}.
\newblock
{\BBOQ}\APACrefatitle {Lagrangian relaxation and constraint generation for
  allocation and advanced scheduling} {Lagrangian relaxation and constraint
  generation for allocation and advanced scheduling}.{\BBCQ}
\newblock
\APACjournalVolNumPages{Computers \& Operations Research}{39}{10}{2323--2336}.
\PrintBackRefs{\CurrentBib}

\bibitem [\protect \citeauthoryear {%
Green%
, Savin%
\BCBL {}\ \BBA {} Wang%
}{%
Green%
\ \protect \BOthers {.}}{%
{\protect \APACyear {2006}}%
}]{%
green2006managing}
\APACinsertmetastar {%
green2006managing}%
\begin{APACrefauthors}%
Green, L\BPBI V.%
, Savin, S.%
\BCBL {}\ \BBA {} Wang, B.%
\end{APACrefauthors}%
\unskip\
\newblock
\APACrefYearMonthDay{2006}{}{}.
\newblock
{\BBOQ}\APACrefatitle {Managing patient service in a diagnostic medical
  facility} {Managing patient service in a diagnostic medical facility}.{\BBCQ}
\newblock
\APACjournalVolNumPages{Operations Research}{54}{1}{11--25}.
\PrintBackRefs{\CurrentBib}

\bibitem [\protect \citeauthoryear {%
Guerriero%
\ \BBA {} Guido%
}{%
Guerriero%
\ \BBA {} Guido%
}{%
{\protect \APACyear {2011}}%
}]{%
guerriero2011operational}
\APACinsertmetastar {%
guerriero2011operational}%
\begin{APACrefauthors}%
Guerriero, F.%
\BCBT {}\ \BBA {} Guido, R.%
\end{APACrefauthors}%
\unskip\
\newblock
\APACrefYearMonthDay{2011}{}{}.
\newblock
{\BBOQ}\APACrefatitle {Operational research in the management of the operating
  theatre: A survey} {Operational research in the management of the operating
  theatre: A survey}.{\BBCQ}
\newblock
\APACjournalVolNumPages{Health Care Management Science}{14}{1}{89--114}.
\PrintBackRefs{\CurrentBib}

\bibitem [\protect \citeauthoryear {%
Guido%
\ \BBA {} Conforti%
}{%
Guido%
\ \BBA {} Conforti%
}{%
{\protect \APACyear {2017}}%
}]{%
guido2017hybrid}
\APACinsertmetastar {%
guido2017hybrid}%
\begin{APACrefauthors}%
Guido, R.%
\BCBT {}\ \BBA {} Conforti, D.%
\end{APACrefauthors}%
\unskip\
\newblock
\APACrefYearMonthDay{2017}{}{}.
\newblock
{\BBOQ}\APACrefatitle {A hybrid genetic approach for solving an integrated
  multi-objective operating room planning and scheduling problem} {A hybrid
  genetic approach for solving an integrated multi-objective operating room
  planning and scheduling problem}.{\BBCQ}
\newblock
\APACjournalVolNumPages{Computers \& Operations Research}{87}{}{270--282}.
\PrintBackRefs{\CurrentBib}

\bibitem [\protect \citeauthoryear {%
Hashemi~Doulabi%
, Rousseau%
\BCBL {}\ \BBA {} Pesant%
}{%
Hashemi~Doulabi%
\ \protect \BOthers {.}}{%
{\protect \APACyear {2016}}%
}]{%
hashemi2016constraint}
\APACinsertmetastar {%
hashemi2016constraint}%
\begin{APACrefauthors}%
Hashemi~Doulabi, S\BPBI H.%
, Rousseau, L\BHBI M.%
\BCBL {}\ \BBA {} Pesant, G.%
\end{APACrefauthors}%
\unskip\
\newblock
\APACrefYearMonthDay{2016}{}{}.
\newblock
{\BBOQ}\APACrefatitle {A constraint-programming-based branch-and-price-and-cut
  approach for operating room planning and scheduling} {A
  constraint-programming-based branch-and-price-and-cut approach for operating
  room planning and scheduling}.{\BBCQ}
\newblock
\APACjournalVolNumPages{INFORMS Journal on Computing}{28}{3}{432--448}.
\PrintBackRefs{\CurrentBib}

\bibitem [\protect \citeauthoryear {%
Hosseini%
\ \BBA {} Taaffe%
}{%
Hosseini%
\ \BBA {} Taaffe%
}{%
{\protect \APACyear {2014}}%
}]{%
hosseini2014evaluation}
\APACinsertmetastar {%
hosseini2014evaluation}%
\begin{APACrefauthors}%
Hosseini, N.%
\BCBT {}\ \BBA {} Taaffe, K.%
\end{APACrefauthors}%
\unskip\
\newblock
\APACrefYearMonthDay{2014}{}{}.
\newblock
{\BBOQ}\APACrefatitle {Evaluation of optimal scheduling policy for
  accommodating elective and non-elective surgery via simulation} {Evaluation
  of optimal scheduling policy for accommodating elective and non-elective
  surgery via simulation}.{\BBCQ}
\newblock
\BIn{} \APACrefbtitle {{Proceedings of the 2014 Winter Simulation Conference}}
  {{Proceedings of the 2014 Winter Simulation Conference}}\ (\BPGS\
  1377--1386).
\PrintBackRefs{\CurrentBib}

\bibitem [\protect \citeauthoryear {%
Huh%
, Liu%
\BCBL {}\ \BBA {} Truong%
}{%
Huh%
\ \protect \BOthers {.}}{%
{\protect \APACyear {2013}}%
}]{%
huh2013multiresource}
\APACinsertmetastar {%
huh2013multiresource}%
\begin{APACrefauthors}%
Huh, W\BPBI T.%
, Liu, N.%
\BCBL {}\ \BBA {} Truong, V\BHBI A.%
\end{APACrefauthors}%
\unskip\
\newblock
\APACrefYearMonthDay{2013}{}{}.
\newblock
{\BBOQ}\APACrefatitle {Multiresource allocation scheduling in dynamic
  environments} {Multiresource allocation scheduling in dynamic
  environments}.{\BBCQ}
\newblock
\APACjournalVolNumPages{Manufacturing \& Service Operations
  Management}{15}{2}{280--291}.
\PrintBackRefs{\CurrentBib}

\bibitem [\protect \citeauthoryear {%
Jebali%
\ \BBA {} Diabat%
}{%
Jebali%
\ \BBA {} Diabat%
}{%
{\protect \APACyear {2015}}%
}]{%
jebali2015stochastic}
\APACinsertmetastar {%
jebali2015stochastic}%
\begin{APACrefauthors}%
Jebali, A.%
\BCBT {}\ \BBA {} Diabat, A.%
\end{APACrefauthors}%
\unskip\
\newblock
\APACrefYearMonthDay{2015}{}{}.
\newblock
{\BBOQ}\APACrefatitle {A stochastic model for operating room planning under
  capacity constraints} {A stochastic model for operating room planning under
  capacity constraints}.{\BBCQ}
\newblock
\APACjournalVolNumPages{International Journal of Production
  Research}{53}{24}{7252--7270}.
\PrintBackRefs{\CurrentBib}

\bibitem [\protect \citeauthoryear {%
Jebali%
\ \BBA {} Diabat%
}{%
Jebali%
\ \BBA {} Diabat%
}{%
{\protect \APACyear {2017}}%
}]{%
jebali2017chance}
\APACinsertmetastar {%
jebali2017chance}%
\begin{APACrefauthors}%
Jebali, A.%
\BCBT {}\ \BBA {} Diabat, A.%
\end{APACrefauthors}%
\unskip\
\newblock
\APACrefYearMonthDay{2017}{}{}.
\newblock
{\BBOQ}\APACrefatitle {A Chance-constrained operating room planning with
  elective and emergency cases under downstream capacity constraints} {A
  chance-constrained operating room planning with elective and emergency cases
  under downstream capacity constraints}.{\BBCQ}
\newblock
\APACjournalVolNumPages{Computers \& Industrial Engineering}{114}{}{329--344}.
\PrintBackRefs{\CurrentBib}

\bibitem [\protect \citeauthoryear {%
Jonnalagadda%
, Walrond%
, Hariharan%
, Walrond%
\BCBL {}\ \BBA {} Prasad%
}{%
Jonnalagadda%
\ \protect \BOthers {.}}{%
{\protect \APACyear {2005}}%
}]{%
jonnalagadda2005evaluation}
\APACinsertmetastar {%
jonnalagadda2005evaluation}%
\begin{APACrefauthors}%
Jonnalagadda, R.%
, Walrond, E.%
, Hariharan, S.%
, Walrond, M.%
\BCBL {}\ \BBA {} Prasad, C.%
\end{APACrefauthors}%
\unskip\
\newblock
\APACrefYearMonthDay{2005}{}{}.
\newblock
{\BBOQ}\APACrefatitle {Evaluation of the reasons for cancellations and delays
  of surgical procedures in a developing country} {Evaluation of the reasons
  for cancellations and delays of surgical procedures in a developing
  country}.{\BBCQ}
\newblock
\APACjournalVolNumPages{International journal of clinical
  practice}{59}{6}{716--720}.
\PrintBackRefs{\CurrentBib}

\bibitem [\protect \citeauthoryear {%
Koppka%
, Wiesche%
, Schacht%
\BCBL {}\ \BBA {} Werners%
}{%
Koppka%
\ \protect \BOthers {.}}{%
{\protect \APACyear {2018}}%
}]{%
koppka2018optimal}
\APACinsertmetastar {%
koppka2018optimal}%
\begin{APACrefauthors}%
Koppka, L.%
, Wiesche, L.%
, Schacht, M.%
\BCBL {}\ \BBA {} Werners, B.%
\end{APACrefauthors}%
\unskip\
\newblock
\APACrefYearMonthDay{2018}{}{}.
\newblock
{\BBOQ}\APACrefatitle {Optimal distribution of operating hours over operating
  rooms using probabilities} {Optimal distribution of operating hours over
  operating rooms using probabilities}.{\BBCQ}
\newblock
\APACjournalVolNumPages{European Journal of Operational
  Research}{267}{3}{1156--1171}.
\PrintBackRefs{\CurrentBib}

\bibitem [\protect \citeauthoryear {%
Lamiri%
, Xie%
, Dolgui%
\BCBL {}\ \BBA {} Grimaud%
}{%
Lamiri%
, Xie%
, Dolgui%
\BCBL {}\ \BBA {} Grimaud%
}{%
{\protect \APACyear {2008}}%
}]{%
lamiri2008stochastic}
\APACinsertmetastar {%
lamiri2008stochastic}%
\begin{APACrefauthors}%
Lamiri, M.%
, Xie, X.%
, Dolgui, A.%
\BCBL {}\ \BBA {} Grimaud, F.%
\end{APACrefauthors}%
\unskip\
\newblock
\APACrefYearMonthDay{2008}{}{}.
\newblock
{\BBOQ}\APACrefatitle {A stochastic model for operating room planning with
  elective and emergency demand for surgery} {A stochastic model for operating
  room planning with elective and emergency demand for surgery}.{\BBCQ}
\newblock
\APACjournalVolNumPages{European Journal of Operational
  Research}{185}{3}{1026--1037}.
\PrintBackRefs{\CurrentBib}

\bibitem [\protect \citeauthoryear {%
Lamiri%
, Xie%
\BCBL {}\ \BBA {} Zhang%
}{%
Lamiri%
, Xie%
\BCBL {}\ \BBA {} Zhang%
}{%
{\protect \APACyear {2008}}%
}]{%
lamiri2008column}
\APACinsertmetastar {%
lamiri2008column}%
\begin{APACrefauthors}%
Lamiri, M.%
, Xie, X.%
\BCBL {}\ \BBA {} Zhang, S.%
\end{APACrefauthors}%
\unskip\
\newblock
\APACrefYearMonthDay{2008}{}{}.
\newblock
{\BBOQ}\APACrefatitle {Column generation approach to operating theater planning
  with elective and emergency patients} {Column generation approach to
  operating theater planning with elective and emergency patients}.{\BBCQ}
\newblock
\APACjournalVolNumPages{{IIE} Transactions}{40}{9}{838--852}.
\PrintBackRefs{\CurrentBib}

\bibitem [\protect \citeauthoryear {%
Las~Hayas%
\ \protect \BOthers {.}}{%
Las~Hayas%
\ \protect \BOthers {.}}{%
{\protect \APACyear {2010}}%
}]{%
las2010can}
\APACinsertmetastar {%
las2010can}%
\begin{APACrefauthors}%
Las~Hayas, C.%
, Gonz{\'a}lez, N.%
, Aguirre, U.%
, Blasco, J\BPBI A.%
, Elizalde, B.%
, Perea, E.%
\BDBL {}others%
\end{APACrefauthors}%
\unskip\
\newblock
\APACrefYearMonthDay{2010}{}{}.
\newblock
{\BBOQ}\APACrefatitle {Can an appropriateness evaluation tool be used to
  prioritize patients on a waiting list for cataract extraction?} {Can an
  appropriateness evaluation tool be used to prioritize patients on a waiting
  list for cataract extraction?}{\BBCQ}
\newblock
\APACjournalVolNumPages{Health policy}{95}{2-3}{194--203}.
\PrintBackRefs{\CurrentBib}

\bibitem [\protect \citeauthoryear {%
Liu%
, Ziya%
\BCBL {}\ \BBA {} Kulkarni%
}{%
Liu%
\ \protect \BOthers {.}}{%
{\protect \APACyear {2010}}%
}]{%
liu2010dynamic}
\APACinsertmetastar {%
liu2010dynamic}%
\begin{APACrefauthors}%
Liu, N.%
, Ziya, S.%
\BCBL {}\ \BBA {} Kulkarni, V\BPBI G.%
\end{APACrefauthors}%
\unskip\
\newblock
\APACrefYearMonthDay{2010}{}{}.
\newblock
{\BBOQ}\APACrefatitle {Dynamic scheduling of outpatient appointments under
  patient no-shows and cancellations} {Dynamic scheduling of outpatient
  appointments under patient no-shows and cancellations}.{\BBCQ}
\newblock
\APACjournalVolNumPages{Manufacturing \& Service Operations
  Management}{12}{2}{347--364}.
\PrintBackRefs{\CurrentBib}

\bibitem [\protect \citeauthoryear {%
Ljung%
\ \BBA {} S{\"o}derstr{\"o}m%
}{%
Ljung%
\ \BBA {} S{\"o}derstr{\"o}m%
}{%
{\protect \APACyear {1983}}%
}]{%
ljung1983theory}
\APACinsertmetastar {%
ljung1983theory}%
\begin{APACrefauthors}%
Ljung, L.%
\BCBT {}\ \BBA {} S{\"o}derstr{\"o}m, T.%
\end{APACrefauthors}%
\unskip\
\newblock
\APACrefYear{1983}.
\newblock
\APACrefbtitle {Theory and practice of recursive identification} {Theory and
  practice of recursive identification}.
\newblock
\APACaddressPublisher{Cambridge, MA}{MIT Press}.
\PrintBackRefs{\CurrentBib}

\bibitem [\protect \citeauthoryear {%
Marques%
\ \BBA {} Captivo%
}{%
Marques%
\ \BBA {} Captivo%
}{%
{\protect \APACyear {2017}}%
}]{%
marques2017different}
\APACinsertmetastar {%
marques2017different}%
\begin{APACrefauthors}%
Marques, I.%
\BCBT {}\ \BBA {} Captivo, M\BPBI E.%
\end{APACrefauthors}%
\unskip\
\newblock
\APACrefYearMonthDay{2017}{}{}.
\newblock
{\BBOQ}\APACrefatitle {Different stakeholders’ perspectives for a surgical
  case assignment problem: Deterministic and robust approaches} {Different
  stakeholders’ perspectives for a surgical case assignment problem:
  Deterministic and robust approaches}.{\BBCQ}
\newblock
\APACjournalVolNumPages{European Journal of Operational
  Research}{261}{1}{260--278}.
\PrintBackRefs{\CurrentBib}

\bibitem [\protect \citeauthoryear {%
Mausam%
\ \BBA {} Kolobov%
}{%
Mausam%
\ \BBA {} Kolobov%
}{%
{\protect \APACyear {2012}}%
}]{%
kolobov2012planning}
\APACinsertmetastar {%
kolobov2012planning}%
\begin{APACrefauthors}%
Mausam%
\BCBT {}\ \BBA {} Kolobov, A.%
\end{APACrefauthors}%
\unskip\
\newblock
\APACrefYearMonthDay{2012}{}{}.
\newblock
{\BBOQ}\APACrefatitle {Planning with {Markov} decision processes: An {AI}
  perspective} {Planning with {Markov} decision processes: An {AI}
  perspective}.{\BBCQ}
\newblock
\APACjournalVolNumPages{Synthesis Lectures on Artificial Intelligence and
  Machine Learning}{6}{1}{1--210}.
\PrintBackRefs{\CurrentBib}

\bibitem [\protect \citeauthoryear {%
McMahan%
, Likhachev%
\BCBL {}\ \BBA {} Gordon%
}{%
McMahan%
\ \protect \BOthers {.}}{%
{\protect \APACyear {2005}}%
}]{%
mcmahan2005bounded}
\APACinsertmetastar {%
mcmahan2005bounded}%
\begin{APACrefauthors}%
McMahan, H\BPBI B.%
, Likhachev, M.%
\BCBL {}\ \BBA {} Gordon, G\BPBI J.%
\end{APACrefauthors}%
\unskip\
\newblock
\APACrefYearMonthDay{2005}{}{}.
\newblock
{\BBOQ}\APACrefatitle {Bounded real-time dynamic programming: {RTDP} with
  monotone upper bounds and performance guarantees} {Bounded real-time dynamic
  programming: {RTDP} with monotone upper bounds and performance
  guarantees}.{\BBCQ}
\newblock
\BIn{} \APACrefbtitle {{Proceedings of the 22nd International Conference on
  Machine Learning}} {{Proceedings of the 22nd International Conference on
  Machine Learning}}\ (\BPGS\ 569--576).
\PrintBackRefs{\CurrentBib}

\bibitem [\protect \citeauthoryear {%
Min%
\ \BBA {} Yih%
}{%
Min%
\ \BBA {} Yih%
}{%
{\protect \APACyear {2010}}%
{\protect \APACexlab {{\protect \BCnt {1}}}}}]{%
min2010elective}
\APACinsertmetastar {%
min2010elective}%
\begin{APACrefauthors}%
Min, D.%
\BCBT {}\ \BBA {} Yih, Y.%
\end{APACrefauthors}%
\unskip\
\newblock
\APACrefYearMonthDay{2010{\protect \BCnt {1}}}{}{}.
\newblock
{\BBOQ}\APACrefatitle {An elective surgery scheduling problem considering
  patient priority} {An elective surgery scheduling problem considering patient
  priority}.{\BBCQ}
\newblock
\APACjournalVolNumPages{Computers \& Operations Research}{37}{6}{1091--1099}.
\PrintBackRefs{\CurrentBib}

\bibitem [\protect \citeauthoryear {%
Min%
\ \BBA {} Yih%
}{%
Min%
\ \BBA {} Yih%
}{%
{\protect \APACyear {2010}}%
{\protect \APACexlab {{\protect \BCnt {2}}}}}]{%
min2010scheduling}
\APACinsertmetastar {%
min2010scheduling}%
\begin{APACrefauthors}%
Min, D.%
\BCBT {}\ \BBA {} Yih, Y.%
\end{APACrefauthors}%
\unskip\
\newblock
\APACrefYearMonthDay{2010{\protect \BCnt {2}}}{}{}.
\newblock
{\BBOQ}\APACrefatitle {Scheduling elective surgery under uncertainty and
  downstream capacity constraints} {Scheduling elective surgery under
  uncertainty and downstream capacity constraints}.{\BBCQ}
\newblock
\APACjournalVolNumPages{European Journal of Operational
  Research}{206}{3}{642--652}.
\PrintBackRefs{\CurrentBib}

\bibitem [\protect \citeauthoryear {%
Min%
\ \BBA {} Yih%
}{%
Min%
\ \BBA {} Yih%
}{%
{\protect \APACyear {2014}}%
}]{%
min2014managing}
\APACinsertmetastar {%
min2014managing}%
\begin{APACrefauthors}%
Min, D.%
\BCBT {}\ \BBA {} Yih, Y.%
\end{APACrefauthors}%
\unskip\
\newblock
\APACrefYearMonthDay{2014}{}{}.
\newblock
{\BBOQ}\APACrefatitle {Managing a patient waiting list with time-dependent
  priority and adverse events} {Managing a patient waiting list with
  time-dependent priority and adverse events}.{\BBCQ}
\newblock
\APACjournalVolNumPages{RAIRO-Operations Research}{48}{1}{53--74}.
\PrintBackRefs{\CurrentBib}

\bibitem [\protect \citeauthoryear {%
Molina-Pariente%
, Fernandez-Viagas%
\BCBL {}\ \BBA {} Framinan%
}{%
Molina-Pariente%
, Fernandez-Viagas%
\BCBL {}\ \BBA {} Framinan%
}{%
{\protect \APACyear {2015}}%
}]{%
molina2015integrated}
\APACinsertmetastar {%
molina2015integrated}%
\begin{APACrefauthors}%
Molina-Pariente, J\BPBI M.%
, Fernandez-Viagas, V.%
\BCBL {}\ \BBA {} Framinan, J\BPBI M.%
\end{APACrefauthors}%
\unskip\
\newblock
\APACrefYearMonthDay{2015}{}{}.
\newblock
{\BBOQ}\APACrefatitle {Integrated operating room planning and scheduling
  problem with assistant surgeon dependent surgery durations} {Integrated
  operating room planning and scheduling problem with assistant surgeon
  dependent surgery durations}.{\BBCQ}
\newblock
\APACjournalVolNumPages{Computers \& Industrial Engineering}{82}{}{8--20}.
\PrintBackRefs{\CurrentBib}

\bibitem [\protect \citeauthoryear {%
Molina-Pariente%
, Hans%
\BCBL {}\ \BBA {} Framinan%
}{%
Molina-Pariente%
\ \protect \BOthers {.}}{%
{\protect \APACyear {2018}}%
}]{%
molina2018stochastic}
\APACinsertmetastar {%
molina2018stochastic}%
\begin{APACrefauthors}%
Molina-Pariente, J\BPBI M.%
, Hans, E\BPBI W.%
\BCBL {}\ \BBA {} Framinan, J\BPBI M.%
\end{APACrefauthors}%
\unskip\
\newblock
\APACrefYearMonthDay{2018}{}{}.
\newblock
{\BBOQ}\APACrefatitle {A stochastic approach for solving the operating room
  scheduling problem} {A stochastic approach for solving the operating room
  scheduling problem}.{\BBCQ}
\newblock
\APACjournalVolNumPages{Flexible services and manufacturing
  journal}{30}{1-2}{224--251}.
\PrintBackRefs{\CurrentBib}

\bibitem [\protect \citeauthoryear {%
Molina-Pariente%
, Hans%
, Framinan%
\BCBL {}\ \BBA {} Gomez-Cia%
}{%
Molina-Pariente%
, Hans%
\BCBL {}\ \protect \BOthers {.}}{%
{\protect \APACyear {2015}}%
}]{%
molina2015new}
\APACinsertmetastar {%
molina2015new}%
\begin{APACrefauthors}%
Molina-Pariente, J\BPBI M.%
, Hans, E\BPBI W.%
, Framinan, J\BPBI M.%
\BCBL {}\ \BBA {} Gomez-Cia, T.%
\end{APACrefauthors}%
\unskip\
\newblock
\APACrefYearMonthDay{2015}{}{}.
\newblock
{\BBOQ}\APACrefatitle {New heuristics for planning operating rooms} {New
  heuristics for planning operating rooms}.{\BBCQ}
\newblock
\APACjournalVolNumPages{Computers \& Industrial Engineering}{90}{}{429--443}.
\PrintBackRefs{\CurrentBib}

\bibitem [\protect \citeauthoryear {%
Montoya%
\ \protect \BOthers {.}}{%
Montoya%
\ \protect \BOthers {.}}{%
{\protect \APACyear {2014}}%
}]{%
montoya2014study}
\APACinsertmetastar {%
montoya2014study}%
\begin{APACrefauthors}%
Montoya, S\BPBI B.%
, Gonz{\'a}lez, M\BPBI S.%
, L{\'o}pez, S\BPBI F.%
, Mu{\~n}oz, J\BPBI D.%
, Gaibar, A\BPBI G.%
\BCBL {}\ \BBA {} Rodr{\'\i}guez, J\BPBI R\BPBI E.%
\end{APACrefauthors}%
\unskip\
\newblock
\APACrefYearMonthDay{2014}{}{}.
\newblock
{\BBOQ}\APACrefatitle {Study to develop a waiting list prioritization score for
  varicose vein surgery} {Study to develop a waiting list prioritization score
  for varicose vein surgery}.{\BBCQ}
\newblock
\APACjournalVolNumPages{Annals of Vascular Surgery}{28}{2}{306--312}.
\PrintBackRefs{\CurrentBib}

\bibitem [\protect \citeauthoryear {%
Moosavi%
\ \BBA {} Ebrahimnejad%
}{%
Moosavi%
\ \BBA {} Ebrahimnejad%
}{%
{\protect \APACyear {2018}}%
}]{%
moosavi2018scheduling}
\APACinsertmetastar {%
moosavi2018scheduling}%
\begin{APACrefauthors}%
Moosavi, A.%
\BCBT {}\ \BBA {} Ebrahimnejad, S.%
\end{APACrefauthors}%
\unskip\
\newblock
\APACrefYearMonthDay{2018}{}{}.
\newblock
{\BBOQ}\APACrefatitle {Scheduling of elective patients considering upstream and
  downstream units and emergency demand using robust optimization} {Scheduling
  of elective patients considering upstream and downstream units and emergency
  demand using robust optimization}.{\BBCQ}
\newblock
\APACjournalVolNumPages{Computers \& Industrial Engineering}{120}{}{216--233}.
\PrintBackRefs{\CurrentBib}

\bibitem [\protect \citeauthoryear {%
Mullen%
}{%
Mullen%
}{%
{\protect \APACyear {2003}}%
}]{%
mullen2003prioritising}
\APACinsertmetastar {%
mullen2003prioritising}%
\begin{APACrefauthors}%
Mullen, P\BPBI M.%
\end{APACrefauthors}%
\unskip\
\newblock
\APACrefYearMonthDay{2003}{}{}.
\newblock
{\BBOQ}\APACrefatitle {Prioritising waiting lists: how and why?} {Prioritising
  waiting lists: how and why?}{\BBCQ}
\newblock
\APACjournalVolNumPages{European Journal of Operational
  Research}{150}{1}{32--45}.
\PrintBackRefs{\CurrentBib}

\bibitem [\protect \citeauthoryear {%
{National Bureau of Statistics of China}%
}{%
{National Bureau of Statistics of China}%
}{%
{\protect \APACyear {2019}}%
}]{%
NBSC}
\APACinsertmetastar {%
NBSC}%
\begin{APACrefauthors}%
{National Bureau of Statistics of China}.%
\end{APACrefauthors}%
\unskip\
\newblock
\APACrefYearMonthDay{2019}{}{}.
\newblock
\APACrefbtitle {China Statistical Yearbook 2018.} {China statistical yearbook
  2018.}
\newblock
\APAChowpublished {\url{http://www.stats.gov.cn/tjsj/ndsj/2018/indexeh.htm}
  (retrieved May 17, 2019)}.
\PrintBackRefs{\CurrentBib}

\bibitem [\protect \citeauthoryear {%
Neyshabouri%
\ \BBA {} Berg%
}{%
Neyshabouri%
\ \BBA {} Berg%
}{%
{\protect \APACyear {2017}}%
}]{%
neyshabouri2017two}
\APACinsertmetastar {%
neyshabouri2017two}%
\begin{APACrefauthors}%
Neyshabouri, S.%
\BCBT {}\ \BBA {} Berg, B\BPBI P.%
\end{APACrefauthors}%
\unskip\
\newblock
\APACrefYearMonthDay{2017}{}{}.
\newblock
{\BBOQ}\APACrefatitle {Two-stage robust optimization approach to elective
  surgery and downstream capacity planning} {Two-stage robust optimization
  approach to elective surgery and downstream capacity planning}.{\BBCQ}
\newblock
\APACjournalVolNumPages{European Journal of Operational
  Research}{260}{1}{21--40}.
\PrintBackRefs{\CurrentBib}

\bibitem [\protect \citeauthoryear {%
Oliveira%
, B{\'e}langer%
, Marques%
\BCBL {}\ \BBA {} Ruiz%
}{%
Oliveira%
\ \protect \BOthers {.}}{%
{\protect \APACyear {2020}}%
}]{%
oliveira2020assessing}
\APACinsertmetastar {%
oliveira2020assessing}%
\begin{APACrefauthors}%
Oliveira, M.%
, B{\'e}langer, V.%
, Marques, I.%
\BCBL {}\ \BBA {} Ruiz, A.%
\end{APACrefauthors}%
\unskip\
\newblock
\APACrefYearMonthDay{2020}{}{}.
\newblock
{\BBOQ}\APACrefatitle {Assessing the impact of patient prioritization on
  operating room schedules} {Assessing the impact of patient prioritization on
  operating room schedules}.{\BBCQ}
\newblock
\APACjournalVolNumPages{Operations Research for Health Care}{24}{}{100232}.
\PrintBackRefs{\CurrentBib}

\bibitem [\protect \citeauthoryear {%
Pang%
, Xie%
, Song%
\BCBL {}\ \BBA {} Luo%
}{%
Pang%
\ \protect \BOthers {.}}{%
{\protect \APACyear {2018}}%
}]{%
pang2018surgery}
\APACinsertmetastar {%
pang2018surgery}%
\begin{APACrefauthors}%
Pang, B.%
, Xie, X.%
, Song, Y.%
\BCBL {}\ \BBA {} Luo, L.%
\end{APACrefauthors}%
\unskip\
\newblock
\APACrefYearMonthDay{2018}{}{}.
\newblock
{\BBOQ}\APACrefatitle {Surgery Scheduling Under Case Cancellation and Surgery
  Duration Uncertainty} {Surgery scheduling under case cancellation and surgery
  duration uncertainty}.{\BBCQ}
\newblock
\APACjournalVolNumPages{IEEE Transactions on Automation Science and
  Engineering}{16}{1}{74--86}.
\PrintBackRefs{\CurrentBib}

\bibitem [\protect \citeauthoryear {%
Patrick%
, Puterman%
\BCBL {}\ \BBA {} Queyranne%
}{%
Patrick%
\ \protect \BOthers {.}}{%
{\protect \APACyear {2008}}%
}]{%
patrick2008dynamic}
\APACinsertmetastar {%
patrick2008dynamic}%
\begin{APACrefauthors}%
Patrick, J.%
, Puterman, M\BPBI L.%
\BCBL {}\ \BBA {} Queyranne, M.%
\end{APACrefauthors}%
\unskip\
\newblock
\APACrefYearMonthDay{2008}{}{}.
\newblock
{\BBOQ}\APACrefatitle {Dynamic multipriority patient scheduling for a
  diagnostic resource} {Dynamic multipriority patient scheduling for a
  diagnostic resource}.{\BBCQ}
\newblock
\APACjournalVolNumPages{Operations research}{56}{6}{1507--1525}.
\PrintBackRefs{\CurrentBib}

\bibitem [\protect \citeauthoryear {%
Powell%
}{%
Powell%
}{%
{\protect \APACyear {2011}}%
}]{%
powell2007approximate}
\APACinsertmetastar {%
powell2007approximate}%
\begin{APACrefauthors}%
Powell, W\BPBI B.%
\end{APACrefauthors}%
\unskip\
\newblock
\APACrefYear{2011}.
\newblock
\APACrefbtitle {Approximate Dynamic Programming: Solving the curses of
  dimensionality} {Approximate dynamic programming: Solving the curses of
  dimensionality}\ (\PrintOrdinal{2nd}\ \BEd).
\newblock
\APACaddressPublisher{Hoboken, New Jersey}{John Wiley \& Sons, Inc.}
\PrintBackRefs{\CurrentBib}

\bibitem [\protect \citeauthoryear {%
Powell%
\ \BBA {} Meisel%
}{%
Powell%
\ \BBA {} Meisel%
}{%
{\protect \APACyear {2015}}%
}]{%
powell2015tutorial}
\APACinsertmetastar {%
powell2015tutorial}%
\begin{APACrefauthors}%
Powell, W\BPBI B.%
\BCBT {}\ \BBA {} Meisel, S.%
\end{APACrefauthors}%
\unskip\
\newblock
\APACrefYearMonthDay{2015}{}{}.
\newblock
{\BBOQ}\APACrefatitle {{Tutorial on stochastic optimization in energy—Part
  II: An energy storage illustration}} {{Tutorial on stochastic optimization in
  energy—Part II: An energy storage illustration}}.{\BBCQ}
\newblock
\APACjournalVolNumPages{IEEE Transactions on Power Systems}{31}{2}{1468--1475}.
\PrintBackRefs{\CurrentBib}

\bibitem [\protect \citeauthoryear {%
Puterman%
}{%
Puterman%
}{%
{\protect \APACyear {1994}}%
}]{%
puterman1994markov}
\APACinsertmetastar {%
puterman1994markov}%
\begin{APACrefauthors}%
Puterman, M\BPBI L.%
\end{APACrefauthors}%
\unskip\
\newblock
\APACrefYear{1994}.
\newblock
\APACrefbtitle {Markov Decision Processes: Discrete Stochastic Dynamic
  Programming} {Markov decision processes: Discrete stochastic dynamic
  programming}.
\newblock
\APACaddressPublisher{New York}{John Wiley \& Sons, Inc.}
\PrintBackRefs{\CurrentBib}

\bibitem [\protect \citeauthoryear {%
Rachuba%
\ \BBA {} Werners%
}{%
Rachuba%
\ \BBA {} Werners%
}{%
{\protect \APACyear {2017}}%
}]{%
rachuba2017fuzzy}
\APACinsertmetastar {%
rachuba2017fuzzy}%
\begin{APACrefauthors}%
Rachuba, S.%
\BCBT {}\ \BBA {} Werners, B.%
\end{APACrefauthors}%
\unskip\
\newblock
\APACrefYearMonthDay{2017}{}{}.
\newblock
{\BBOQ}\APACrefatitle {A fuzzy multi-criteria approach for robust operating
  room schedules} {A fuzzy multi-criteria approach for robust operating room
  schedules}.{\BBCQ}
\newblock
\APACjournalVolNumPages{Annals of Operations Research}{251}{1-2}{325--350}.
\PrintBackRefs{\CurrentBib}

\bibitem [\protect \citeauthoryear {%
Rath%
, Rajaram%
\BCBL {}\ \BBA {} Mahajan%
}{%
Rath%
\ \protect \BOthers {.}}{%
{\protect \APACyear {2017}}%
}]{%
rath2017integrated}
\APACinsertmetastar {%
rath2017integrated}%
\begin{APACrefauthors}%
Rath, S.%
, Rajaram, K.%
\BCBL {}\ \BBA {} Mahajan, A.%
\end{APACrefauthors}%
\unskip\
\newblock
\APACrefYearMonthDay{2017}{}{}.
\newblock
{\BBOQ}\APACrefatitle {Integrated Anesthesiologist and Room Scheduling for
  Surgeries: Methodology and Application} {Integrated anesthesiologist and room
  scheduling for surgeries: Methodology and application}.{\BBCQ}
\newblock
\APACjournalVolNumPages{Operations Research}{65}{6}{1460--1478}.
\PrintBackRefs{\CurrentBib}

\bibitem [\protect \citeauthoryear {%
Roshanaei%
, Luong%
, Aleman%
\BCBL {}\ \BBA {} Urbach%
}{%
Roshanaei%
\ \protect \BOthers {.}}{%
{\protect \APACyear {2017}}%
{\protect \APACexlab {{\protect \BCnt {2}}}}}]{%
roshanaei2017propagating}
\APACinsertmetastar {%
roshanaei2017propagating}%
\begin{APACrefauthors}%
Roshanaei, V.%
, Luong, C.%
, Aleman, D\BPBI M.%
\BCBL {}\ \BBA {} Urbach, D.%
\end{APACrefauthors}%
\unskip\
\newblock
\APACrefYearMonthDay{2017{\protect \BCnt {2}}}{}{}.
\newblock
{\BBOQ}\APACrefatitle {Propagating logic-based Benders’ decomposition
  approaches for distributed operating room scheduling} {Propagating
  logic-based benders’ decomposition approaches for distributed operating
  room scheduling}.{\BBCQ}
\newblock
\APACjournalVolNumPages{European Journal of Operational
  Research}{257}{2}{439--455}.
\PrintBackRefs{\CurrentBib}

\bibitem [\protect \citeauthoryear {%
Roshanaei%
, Luong%
, Aleman%
\BCBL {}\ \BBA {} Urbach%
}{%
Roshanaei%
\ \protect \BOthers {.}}{%
{\protect \APACyear {2017}}%
{\protect \APACexlab {{\protect \BCnt {1}}}}}]{%
roshanaei2017collaborative}
\APACinsertmetastar {%
roshanaei2017collaborative}%
\begin{APACrefauthors}%
Roshanaei, V.%
, Luong, C.%
, Aleman, D\BPBI M.%
\BCBL {}\ \BBA {} Urbach, D\BPBI R.%
\end{APACrefauthors}%
\unskip\
\newblock
\APACrefYearMonthDay{2017{\protect \BCnt {1}}}{}{}.
\newblock
{\BBOQ}\APACrefatitle {Collaborative operating room planning and scheduling}
  {Collaborative operating room planning and scheduling}.{\BBCQ}
\newblock
\APACjournalVolNumPages{INFORMS Journal on Computing}{29}{3}{558--580}.
\PrintBackRefs{\CurrentBib}

\bibitem [\protect \citeauthoryear {%
Roshanaei%
, Luong%
, Aleman%
\BCBL {}\ \BBA {} Urbach%
}{%
Roshanaei%
\ \protect \BOthers {.}}{%
{\protect \APACyear {2020}}%
}]{%
roshanaei2020reformulation}
\APACinsertmetastar {%
roshanaei2020reformulation}%
\begin{APACrefauthors}%
Roshanaei, V.%
, Luong, C.%
, Aleman, D\BPBI M.%
\BCBL {}\ \BBA {} Urbach, D\BPBI R.%
\end{APACrefauthors}%
\unskip\
\newblock
\APACrefYearMonthDay{2020}{}{}.
\newblock
{\BBOQ}\APACrefatitle {Reformulation, linearization, and decomposition
  techniques for balanced distributed operating room scheduling}
  {Reformulation, linearization, and decomposition techniques for balanced
  distributed operating room scheduling}.{\BBCQ}
\newblock
\APACjournalVolNumPages{Omega}{93}{}{102043}.
\PrintBackRefs{\CurrentBib}

\bibitem [\protect \citeauthoryear {%
Samudra%
\ \protect \BOthers {.}}{%
Samudra%
\ \protect \BOthers {.}}{%
{\protect \APACyear {2016}}%
}]{%
samudra2016scheduling}
\APACinsertmetastar {%
samudra2016scheduling}%
\begin{APACrefauthors}%
Samudra, M.%
, Van~Riet, C.%
, Demeulemeester, E.%
, Cardoen, B.%
, Vansteenkiste, N.%
\BCBL {}\ \BBA {} Rademakers, F\BPBI E.%
\end{APACrefauthors}%
\unskip\
\newblock
\APACrefYearMonthDay{2016}{}{}.
\newblock
{\BBOQ}\APACrefatitle {Scheduling operating rooms: Achievements, challenges and
  pitfalls} {Scheduling operating rooms: Achievements, challenges and
  pitfalls}.{\BBCQ}
\newblock
\APACjournalVolNumPages{Journal of Scheduling}{19}{5}{493--525}.
\PrintBackRefs{\CurrentBib}

\bibitem [\protect \citeauthoryear {%
Sanner%
, Goetschalckx%
, Driessens%
\BCBL {}\ \BBA {} Shani%
}{%
Sanner%
\ \protect \BOthers {.}}{%
{\protect \APACyear {2009}}%
}]{%
Sanner2009VPI}
\APACinsertmetastar {%
Sanner2009VPI}%
\begin{APACrefauthors}%
Sanner, S.%
, Goetschalckx, R.%
, Driessens, K.%
\BCBL {}\ \BBA {} Shani, G.%
\end{APACrefauthors}%
\unskip\
\newblock
\APACrefYearMonthDay{2009}{}{}.
\newblock
{\BBOQ}\APACrefatitle {Bayesian Real-time Dynamic Programming} {Bayesian
  real-time dynamic programming}.{\BBCQ}
\newblock
\BIn{} \APACrefbtitle {{Proceedings of the 21st International Joint Conference
  on Artifical Intelligence}} {{Proceedings of the 21st International Joint
  Conference on Artifical Intelligence}}\ (\BPGS\ 1784--1789).
\PrintBackRefs{\CurrentBib}

\bibitem [\protect \citeauthoryear {%
Smith%
\ \BBA {} Simmons%
}{%
Smith%
\ \BBA {} Simmons%
}{%
{\protect \APACyear {2006}}%
}]{%
Smith2006Focused}
\APACinsertmetastar {%
Smith2006Focused}%
\begin{APACrefauthors}%
Smith, T.%
\BCBT {}\ \BBA {} Simmons, R.%
\end{APACrefauthors}%
\unskip\
\newblock
\APACrefYearMonthDay{2006}{}{}.
\newblock
{\BBOQ}\APACrefatitle {Focused Real-time Dynamic Programming for {MDPs}:
  Squeezing More out of a Heuristic} {Focused real-time dynamic programming for
  {MDPs}: Squeezing more out of a heuristic}.{\BBCQ}
\newblock
\BIn{} \APACrefbtitle {{Proceedings of the 21st National Conference on
  Artificial Intelligence}} {{Proceedings of the 21st National Conference on
  Artificial Intelligence}}\ (\BVOL~2, \BPGS\ 1227--1232).
\PrintBackRefs{\CurrentBib}

\bibitem [\protect \citeauthoryear {%
Sobolev%
\ \BBA {} Kuramoto%
}{%
Sobolev%
\ \BBA {} Kuramoto%
}{%
{\protect \APACyear {2008}}%
}]{%
sobolev2008analysis}
\APACinsertmetastar {%
sobolev2008analysis}%
\begin{APACrefauthors}%
Sobolev, B.%
\BCBT {}\ \BBA {} Kuramoto, L.%
\end{APACrefauthors}%
\unskip\
\newblock
\APACrefYear{2008}.
\newblock
\APACrefbtitle {Analysis of waiting-time data in health services research}
  {Analysis of waiting-time data in health services research}.
\newblock
\APACaddressPublisher{New York}{Springer Science \& Business Media}.
\PrintBackRefs{\CurrentBib}

\bibitem [\protect \citeauthoryear {%
Sutton%
}{%
Sutton%
}{%
{\protect \APACyear {1988}}%
}]{%
sutton1988learning}
\APACinsertmetastar {%
sutton1988learning}%
\begin{APACrefauthors}%
Sutton, R\BPBI S.%
\end{APACrefauthors}%
\unskip\
\newblock
\APACrefYearMonthDay{1988}{}{}.
\newblock
{\BBOQ}\APACrefatitle {Learning to predict by the methods of temporal
  differences} {Learning to predict by the methods of temporal
  differences}.{\BBCQ}
\newblock
\APACjournalVolNumPages{Machine learning}{3}{1}{9--44}.
\PrintBackRefs{\CurrentBib}

\bibitem [\protect \citeauthoryear {%
Tancrez%
, Roland%
, Cordier%
\BCBL {}\ \BBA {} Riane%
}{%
Tancrez%
\ \protect \BOthers {.}}{%
{\protect \APACyear {2013}}%
}]{%
tancrez2013assessing}
\APACinsertmetastar {%
tancrez2013assessing}%
\begin{APACrefauthors}%
Tancrez, J\BHBI S.%
, Roland, B.%
, Cordier, J\BHBI P.%
\BCBL {}\ \BBA {} Riane, F.%
\end{APACrefauthors}%
\unskip\
\newblock
\APACrefYearMonthDay{2013}{}{}.
\newblock
{\BBOQ}\APACrefatitle {Assessing the impact of stochasticity for operating
  theater sizing} {Assessing the impact of stochasticity for operating theater
  sizing}.{\BBCQ}
\newblock
\APACjournalVolNumPages{Decision Support Systems}{55}{2}{616--628}.
\PrintBackRefs{\CurrentBib}

\bibitem [\protect \citeauthoryear {%
T{\`a}nfani%
\ \BBA {} Testi%
}{%
T{\`a}nfani%
\ \BBA {} Testi%
}{%
{\protect \APACyear {2010}}%
}]{%
tanfani2010pre}
\APACinsertmetastar {%
tanfani2010pre}%
\begin{APACrefauthors}%
T{\`a}nfani, E.%
\BCBT {}\ \BBA {} Testi, A.%
\end{APACrefauthors}%
\unskip\
\newblock
\APACrefYearMonthDay{2010}{}{}.
\newblock
{\BBOQ}\APACrefatitle {A pre-assignment heuristic algorithm for the master
  surgical schedule problem (MSSP)} {A pre-assignment heuristic algorithm for
  the master surgical schedule problem (mssp)}.{\BBCQ}
\newblock
\APACjournalVolNumPages{Annals of Operations Research}{178}{1}{105--119}.
\PrintBackRefs{\CurrentBib}

\bibitem [\protect \citeauthoryear {%
Testi%
\ \BBA {} T{\`a}nfani%
}{%
Testi%
\ \BBA {} T{\`a}nfani%
}{%
{\protect \APACyear {2009}}%
}]{%
testi2009tactical}
\APACinsertmetastar {%
testi2009tactical}%
\begin{APACrefauthors}%
Testi, A.%
\BCBT {}\ \BBA {} T{\`a}nfani, E.%
\end{APACrefauthors}%
\unskip\
\newblock
\APACrefYearMonthDay{2009}{}{}.
\newblock
{\BBOQ}\APACrefatitle {Tactical and operational decisions for operating room
  planning: Efficiency and welfare implications} {Tactical and operational
  decisions for operating room planning: Efficiency and welfare
  implications}.{\BBCQ}
\newblock
\APACjournalVolNumPages{Health Care Management Science}{12}{4}{363}.
\PrintBackRefs{\CurrentBib}

\bibitem [\protect \citeauthoryear {%
Testi%
, Tanfani%
\BCBL {}\ \BBA {} Torre%
}{%
Testi%
\ \protect \BOthers {.}}{%
{\protect \APACyear {2007}}%
}]{%
testi2007three}
\APACinsertmetastar {%
testi2007three}%
\begin{APACrefauthors}%
Testi, A.%
, Tanfani, E.%
\BCBL {}\ \BBA {} Torre, G.%
\end{APACrefauthors}%
\unskip\
\newblock
\APACrefYearMonthDay{2007}{}{}.
\newblock
{\BBOQ}\APACrefatitle {A three-phase approach for operating theatre schedules}
  {A three-phase approach for operating theatre schedules}.{\BBCQ}
\newblock
\APACjournalVolNumPages{Health Care Management Science}{10}{2}{163--172}.
\PrintBackRefs{\CurrentBib}

\bibitem [\protect \citeauthoryear {%
Testi%
, Tanfani%
, Valente%
, Ansaldo%
\BCBL {}\ \BBA {} Torre%
}{%
Testi%
\ \protect \BOthers {.}}{%
{\protect \APACyear {2008}}%
}]{%
testi2008prioritizing}
\APACinsertmetastar {%
testi2008prioritizing}%
\begin{APACrefauthors}%
Testi, A.%
, Tanfani, E.%
, Valente, R.%
, Ansaldo, G.%
\BCBL {}\ \BBA {} Torre, G.%
\end{APACrefauthors}%
\unskip\
\newblock
\APACrefYearMonthDay{2008}{}{}.
\newblock
{\BBOQ}\APACrefatitle {Prioritizing surgical waiting lists} {Prioritizing
  surgical waiting lists}.{\BBCQ}
\newblock
\APACjournalVolNumPages{Journal of Evaluation in Clinical
  Practice}{14}{1}{59--64}.
\PrintBackRefs{\CurrentBib}

\bibitem [\protect \citeauthoryear {%
Truong%
}{%
Truong%
}{%
{\protect \APACyear {2015}}%
}]{%
truong2015optimal}
\APACinsertmetastar {%
truong2015optimal}%
\begin{APACrefauthors}%
Truong, V\BHBI A.%
\end{APACrefauthors}%
\unskip\
\newblock
\APACrefYearMonthDay{2015}{}{}.
\newblock
{\BBOQ}\APACrefatitle {Optimal advance scheduling} {Optimal advance
  scheduling}.{\BBCQ}
\newblock
\APACjournalVolNumPages{Management Science}{61}{7}{1584--1597}.
\PrintBackRefs{\CurrentBib}

\bibitem [\protect \citeauthoryear {%
{Tsitsiklis}%
\ \BBA {} {Van Roy}%
}{%
{Tsitsiklis}%
\ \BBA {} {Van Roy}%
}{%
{\protect \APACyear {1997}}%
}]{%
tsitsiklis1997analysis}
\APACinsertmetastar {%
tsitsiklis1997analysis}%
\begin{APACrefauthors}%
{Tsitsiklis}, J\BPBI N.%
\BCBT {}\ \BBA {} {Van Roy}, B.%
\end{APACrefauthors}%
\unskip\
\newblock
\APACrefYearMonthDay{1997}{May}{}.
\newblock
{\BBOQ}\APACrefatitle {An analysis of temporal-difference learning with
  function approximation} {An analysis of temporal-difference learning with
  function approximation}.{\BBCQ}
\newblock
\APACjournalVolNumPages{IEEE Transactions on Automatic
  Control}{42}{5}{674-690}.
\newblock
\begin{APACrefDOI} \doi{10.1109/9.580874} \end{APACrefDOI}
\PrintBackRefs{\CurrentBib}

\bibitem [\protect \citeauthoryear {%
Ulmer%
\ \BBA {} Thomas%
}{%
Ulmer%
\ \BBA {} Thomas%
}{%
{\protect \APACyear {2020}}%
}]{%
ulmer2020meso}
\APACinsertmetastar {%
ulmer2020meso}%
\begin{APACrefauthors}%
Ulmer, M\BPBI W.%
\BCBT {}\ \BBA {} Thomas, B\BPBI W.%
\end{APACrefauthors}%
\unskip\
\newblock
\APACrefYearMonthDay{2020}{}{}.
\newblock
{\BBOQ}\APACrefatitle {Meso-parametric value function approximation for dynamic
  customer acceptances in delivery routing} {Meso-parametric value function
  approximation for dynamic customer acceptances in delivery routing}.{\BBCQ}
\newblock
\APACjournalVolNumPages{European Journal of Operational
  Research}{285}{1}{183--195}.
\PrintBackRefs{\CurrentBib}

\bibitem [\protect \citeauthoryear {%
Utzolino%
, Kaffarnik%
, Keck%
, Berlet%
\BCBL {}\ \BBA {} Hopt%
}{%
Utzolino%
\ \protect \BOthers {.}}{%
{\protect \APACyear {2010}}%
}]{%
utzolino2010unplanned}
\APACinsertmetastar {%
utzolino2010unplanned}%
\begin{APACrefauthors}%
Utzolino, S.%
, Kaffarnik, M.%
, Keck, T.%
, Berlet, M.%
\BCBL {}\ \BBA {} Hopt, U\BPBI T.%
\end{APACrefauthors}%
\unskip\
\newblock
\APACrefYearMonthDay{2010}{}{}.
\newblock
{\BBOQ}\APACrefatitle {Unplanned discharges from a surgical intensive care
  unit: readmissions and mortality} {Unplanned discharges from a surgical
  intensive care unit: readmissions and mortality}.{\BBCQ}
\newblock
\APACjournalVolNumPages{Journal of critical care}{25}{3}{375--381}.
\PrintBackRefs{\CurrentBib}

\bibitem [\protect \citeauthoryear {%
Valente%
\ \protect \BOthers {.}}{%
Valente%
\ \protect \BOthers {.}}{%
{\protect \APACyear {2009}}%
}]{%
valente2009model}
\APACinsertmetastar {%
valente2009model}%
\begin{APACrefauthors}%
Valente, R.%
, Testi, A.%
, Tanfani, E.%
, Fato, M.%
, Porro, I.%
, Santo, M.%
\BDBL {}Ansaldo, G.%
\end{APACrefauthors}%
\unskip\
\newblock
\APACrefYearMonthDay{2009}{}{}.
\newblock
{\BBOQ}\APACrefatitle {A model to prioritize access to elective surgery on the
  basis of clinical urgency and waiting time} {A model to prioritize access to
  elective surgery on the basis of clinical urgency and waiting time}.{\BBCQ}
\newblock
\APACjournalVolNumPages{BMC Health Services Research}{9}{1}{1}.
\PrintBackRefs{\CurrentBib}

\bibitem [\protect \citeauthoryear {%
Vancroonenburg%
, De~Causmaecker%
\BCBL {}\ \BBA {} Berghe%
}{%
Vancroonenburg%
\ \protect \BOthers {.}}{%
{\protect \APACyear {2019}}%
}]{%
vancroonenburg2019chance}
\APACinsertmetastar {%
vancroonenburg2019chance}%
\begin{APACrefauthors}%
Vancroonenburg, W.%
, De~Causmaecker, P.%
\BCBL {}\ \BBA {} Berghe, G\BPBI V.%
\end{APACrefauthors}%
\unskip\
\newblock
\APACrefYearMonthDay{2019}{}{}.
\newblock
{\BBOQ}\APACrefatitle {Chance-constrained admission scheduling of elective
  surgical patients in a dynamic, uncertain setting} {Chance-constrained
  admission scheduling of elective surgical patients in a dynamic, uncertain
  setting}.{\BBCQ}
\newblock
\APACjournalVolNumPages{Operations Research for Health Care}{22}{}{100196}.
\PrintBackRefs{\CurrentBib}

\bibitem [\protect \citeauthoryear {%
Van~Riet%
\ \BBA {} Demeulemeester%
}{%
Van~Riet%
\ \BBA {} Demeulemeester%
}{%
{\protect \APACyear {2015}}%
}]{%
van2015trade}
\APACinsertmetastar {%
van2015trade}%
\begin{APACrefauthors}%
Van~Riet, C.%
\BCBT {}\ \BBA {} Demeulemeester, E.%
\end{APACrefauthors}%
\unskip\
\newblock
\APACrefYearMonthDay{2015}{}{}.
\newblock
{\BBOQ}\APACrefatitle {Trade-offs in operating room planning for electives and
  emergencies: A review} {Trade-offs in operating room planning for electives
  and emergencies: A review}.{\BBCQ}
\newblock
\APACjournalVolNumPages{Operations Research for Health Care}{7}{}{52--69}.
\PrintBackRefs{\CurrentBib}

\bibitem [\protect \citeauthoryear {%
Voelkel%
, Sachs%
\BCBL {}\ \BBA {} Thonemann%
}{%
Voelkel%
\ \protect \BOthers {.}}{%
{\protect \APACyear {2020}}%
}]{%
voelkel2020aggregation}
\APACinsertmetastar {%
voelkel2020aggregation}%
\begin{APACrefauthors}%
Voelkel, M\BPBI A.%
, Sachs, A\BHBI L.%
\BCBL {}\ \BBA {} Thonemann, U\BPBI W.%
\end{APACrefauthors}%
\unskip\
\newblock
\APACrefYearMonthDay{2020}{}{}.
\newblock
{\BBOQ}\APACrefatitle {An aggregation-based approximate dynamic programming
  approach for the periodic review model with random yield} {An
  aggregation-based approximate dynamic programming approach for the periodic
  review model with random yield}.{\BBCQ}
\newblock
\APACjournalVolNumPages{European Journal of Operational
  Research}{281}{2}{286--298}.
\PrintBackRefs{\CurrentBib}

\bibitem [\protect \citeauthoryear {%
S.~Wang%
, Roshanaei%
, Aleman%
\BCBL {}\ \BBA {} Urbach%
}{%
S.~Wang%
\ \protect \BOthers {.}}{%
{\protect \APACyear {2016}}%
}]{%
wang2016discrete}
\APACinsertmetastar {%
wang2016discrete}%
\begin{APACrefauthors}%
Wang, S.%
, Roshanaei, V.%
, Aleman, D.%
\BCBL {}\ \BBA {} Urbach, D.%
\end{APACrefauthors}%
\unskip\
\newblock
\APACrefYearMonthDay{2016}{}{}.
\newblock
{\BBOQ}\APACrefatitle {A discrete event simulation evaluation of distributed
  operating room scheduling} {A discrete event simulation evaluation of
  distributed operating room scheduling}.{\BBCQ}
\newblock
\APACjournalVolNumPages{IIE Transactions on Healthcare Systems
  Engineering}{6}{4}{236--245}.
\PrintBackRefs{\CurrentBib}

\bibitem [\protect \citeauthoryear {%
Y.~Wang%
, Tang%
\BCBL {}\ \BBA {} Fung%
}{%
Y.~Wang%
\ \protect \BOthers {.}}{%
{\protect \APACyear {2014}}%
}]{%
wang2014column}
\APACinsertmetastar {%
wang2014column}%
\begin{APACrefauthors}%
Wang, Y.%
, Tang, J.%
\BCBL {}\ \BBA {} Fung, R\BPBI Y.%
\end{APACrefauthors}%
\unskip\
\newblock
\APACrefYearMonthDay{2014}{}{}.
\newblock
{\BBOQ}\APACrefatitle {A column-generation-based heuristic algorithm for
  solving operating theater planning problem under stochastic demand and
  surgery cancellation risk} {A column-generation-based heuristic algorithm for
  solving operating theater planning problem under stochastic demand and
  surgery cancellation risk}.{\BBCQ}
\newblock
\APACjournalVolNumPages{International Journal of Production
  Economics}{158}{}{28--36}.
\PrintBackRefs{\CurrentBib}

\bibitem [\protect \citeauthoryear {%
White%
}{%
White%
}{%
{\protect \APACyear {2001}}%
}]{%
White2001}
\APACinsertmetastar {%
White2001}%
\begin{APACrefauthors}%
White, C\BPBI C.%
\end{APACrefauthors}%
\unskip\
\newblock
\APACrefYearMonthDay{2001}{}{}.
\newblock
{\BBOQ}\APACrefatitle {Markov decision processes} {Markov decision
  processes}.{\BBCQ}
\newblock
\BIn{} \APACrefbtitle {Encyclopedia of Operations Research and Management
  Science} {Encyclopedia of operations research and management science}\
  (\BPGS\ 484--486).
\newblock
\APACaddressPublisher{Boston, MA}{Springer US}.
\PrintBackRefs{\CurrentBib}

\bibitem [\protect \citeauthoryear {%
Xiao%
, van Jaarsveld%
, Dong%
\BCBL {}\ \BBA {} Van De~Klundert%
}{%
Xiao%
\ \protect \BOthers {.}}{%
{\protect \APACyear {2016}}%
}]{%
xiao2016stochastic}
\APACinsertmetastar {%
xiao2016stochastic}%
\begin{APACrefauthors}%
Xiao, G.%
, van Jaarsveld, W.%
, Dong, M.%
\BCBL {}\ \BBA {} Van De~Klundert, J.%
\end{APACrefauthors}%
\unskip\
\newblock
\APACrefYearMonthDay{2016}{}{}.
\newblock
{\BBOQ}\APACrefatitle {Stochastic programming analysis and solutions to
  schedule overcrowded operating rooms in China} {Stochastic programming
  analysis and solutions to schedule overcrowded operating rooms in
  china}.{\BBCQ}
\newblock
\APACjournalVolNumPages{Computers \& Operations Research}{74}{}{78--91}.
\PrintBackRefs{\CurrentBib}

\bibitem [\protect \citeauthoryear {%
Xu%
, He%
\BCBL {}\ \BBA {} Hu%
}{%
Xu%
\ \protect \BOthers {.}}{%
{\protect \APACyear {2002}}%
}]{%
xu2002efficient}
\APACinsertmetastar {%
xu2002efficient}%
\begin{APACrefauthors}%
Xu, X.%
, He, H\BHBI g.%
\BCBL {}\ \BBA {} Hu, D.%
\end{APACrefauthors}%
\unskip\
\newblock
\APACrefYearMonthDay{2002}{}{}.
\newblock
{\BBOQ}\APACrefatitle {Efficient reinforcement learning using recursive
  least-squares methods} {Efficient reinforcement learning using recursive
  least-squares methods}.{\BBCQ}
\newblock
\APACjournalVolNumPages{Journal of Artificial Intelligence
  Research}{16}{}{259--292}.
\PrintBackRefs{\CurrentBib}

\bibitem [\protect \citeauthoryear {%
Yin%
, Dridi%
\BCBL {}\ \BBA {} Moudni%
}{%
Yin%
\ \protect \BOthers {.}}{%
{\protect \APACyear {2017}}%
}]{%
yin2017recursive}
\APACinsertmetastar {%
yin2017recursive}%
\begin{APACrefauthors}%
Yin, B.%
, Dridi, M.%
\BCBL {}\ \BBA {} Moudni, A\BPBI E.%
\end{APACrefauthors}%
\unskip\
\newblock
\APACrefYearMonthDay{2017}{{\APACmonth{06}}}{}.
\newblock
{\BBOQ}\APACrefatitle {Recursive least-squares temporal difference learning for
  adaptive traffic signal control at intersection} {Recursive least-squares
  temporal difference learning for adaptive traffic signal control at
  intersection}.{\BBCQ}
\newblock
\APACjournalVolNumPages{Neural Computing and Applications}{}{}{1--16}.
\PrintBackRefs{\CurrentBib}

\bibitem [\protect \citeauthoryear {%
J.~Zhang%
, Dridi%
\BCBL {}\ \BBA {} El~Moudni%
}{%
J.~Zhang%
\ \protect \BOthers {.}}{%
{\protect \APACyear {2019}}%
{\protect \APACexlab {{\protect \BCnt {1}}}}}]{%
zhang2019markov}
\APACinsertmetastar {%
zhang2019markov}%
\begin{APACrefauthors}%
Zhang, J.%
, Dridi, M.%
\BCBL {}\ \BBA {} El~Moudni, A.%
\end{APACrefauthors}%
\unskip\
\newblock
\APACrefYearMonthDay{2019{\protect \BCnt {1}}}{}{}.
\newblock
{\BBOQ}\APACrefatitle {A Markov decision model with dead ends for operating
  room planning considering dynamic patient priority} {A markov decision model
  with dead ends for operating room planning considering dynamic patient
  priority}.{\BBCQ}
\newblock
\APACjournalVolNumPages{RAIRO-Operations Research}{53}{5}{1819--1841}.
\PrintBackRefs{\CurrentBib}

\bibitem [\protect \citeauthoryear {%
J.~Zhang%
, Dridi%
\BCBL {}\ \BBA {} El~Moudni%
}{%
J.~Zhang%
\ \protect \BOthers {.}}{%
{\protect \APACyear {2019}}%
{\protect \APACexlab {{\protect \BCnt {2}}}}}]{%
zhang2019two}
\APACinsertmetastar {%
zhang2019two}%
\begin{APACrefauthors}%
Zhang, J.%
, Dridi, M.%
\BCBL {}\ \BBA {} El~Moudni, A.%
\end{APACrefauthors}%
\unskip\
\newblock
\APACrefYearMonthDay{2019{\protect \BCnt {2}}}{}{}.
\newblock
{\BBOQ}\APACrefatitle {A two-level optimization model for elective surgery
  scheduling with downstream capacity constraints} {A two-level optimization
  model for elective surgery scheduling with downstream capacity
  constraints}.{\BBCQ}
\newblock
\APACjournalVolNumPages{European Journal of Operational
  Research}{276}{2}{602--613}.
\PrintBackRefs{\CurrentBib}

\bibitem [\protect \citeauthoryear {%
J.~Zhang%
, Dridi%
\BCBL {}\ \BBA {} El~Moudni%
}{%
J.~Zhang%
\ \protect \BOthers {.}}{%
{\protect \APACyear {2020}}%
}]{%
zhang2020column}
\APACinsertmetastar {%
zhang2020column}%
\begin{APACrefauthors}%
Zhang, J.%
, Dridi, M.%
\BCBL {}\ \BBA {} El~Moudni, A.%
\end{APACrefauthors}%
\unskip\
\newblock
\APACrefYearMonthDay{2020}{}{}.
\newblock
{\BBOQ}\APACrefatitle {Column-generation-based heuristic approaches to
  stochastic surgery scheduling with downstream capacity constraints}
  {Column-generation-based heuristic approaches to stochastic surgery
  scheduling with downstream capacity constraints}.{\BBCQ}
\newblock
\APACjournalVolNumPages{International Journal of Production
  Economics}{}{}{107764}.
\PrintBackRefs{\CurrentBib}

\bibitem [\protect \citeauthoryear {%
Y.~Zhang%
, Wang%
, Tang%
\BCBL {}\ \BBA {} Lim%
}{%
Y.~Zhang%
\ \protect \BOthers {.}}{%
{\protect \APACyear {2020}}%
}]{%
zhang2020mitigating}
\APACinsertmetastar {%
zhang2020mitigating}%
\begin{APACrefauthors}%
Zhang, Y.%
, Wang, Y.%
, Tang, J.%
\BCBL {}\ \BBA {} Lim, A.%
\end{APACrefauthors}%
\unskip\
\newblock
\APACrefYearMonthDay{2020}{}{}.
\newblock
{\BBOQ}\APACrefatitle {Mitigating overtime risk in tactical surgical
  scheduling} {Mitigating overtime risk in tactical surgical
  scheduling}.{\BBCQ}
\newblock
\APACjournalVolNumPages{Omega}{93}{}{102024}.
\PrintBackRefs{\CurrentBib}

\bibitem [\protect \citeauthoryear {%
Zhu%
, Fan%
, Yang%
, Pei%
\BCBL {}\ \BBA {} Pardalos%
}{%
Zhu%
\ \protect \BOthers {.}}{%
{\protect \APACyear {2019}}%
}]{%
zhu2019operating}
\APACinsertmetastar {%
zhu2019operating}%
\begin{APACrefauthors}%
Zhu, S.%
, Fan, W.%
, Yang, S.%
, Pei, J.%
\BCBL {}\ \BBA {} Pardalos, P\BPBI M.%
\end{APACrefauthors}%
\unskip\
\newblock
\APACrefYearMonthDay{2019}{}{}.
\newblock
{\BBOQ}\APACrefatitle {Operating room planning and surgical case scheduling: a
  review of literature} {Operating room planning and surgical case scheduling:
  a review of literature}.{\BBCQ}
\newblock
\APACjournalVolNumPages{Journal of Combinatorial
  Optimization}{37}{3}{757--805}.
\PrintBackRefs{\CurrentBib}

\end{thebibliography}
\numberwithin{equation}{section}
\numberwithin{table}{section}
\begin{appendices}
	\section{Proof of Lemma \ref{Lemmamin}}\label{ProofLemmamin}
	\begin{proof}
		Let $x_1=\mathop{\arg\min}f(x)$ and $x_2=\mathop{\arg\min}g(x)$, then $g(x_1)\geqslant g(x_2)$, hence
		\begin{equation}
		\min f(x)-\min g(x)=f(x_1)-g(x_2)\geqslant f(x_1)-g(x_1)\geqslant \min[f(x)-g(x)]
		\end{equation}
		thus the lemma is proved.
	\end{proof}
	\section{Proof of Proposition \ref{ConvexC}}\label{ProofConvexC}
	\begin{proof}
		\begin{itemize}
			\item [(i)] If $a_0(s+\Delta_{j'u'w'})\in A(s)$, then by Lemma \ref{Lemmamin},
			\begin{equation}
			\begin{split}
			C_0(s+\Delta_{j'u'w'})-C_0(s)&=\min_{a\in A(s)}C(s+\Delta_{j'u'w'},a)-\min_{a\in A(s)}C(s,a)\\
			&\geqslant \min_{a\in A(s)}[C(s+\Delta_{j'u'w'},a)-C(s,a)]={c_d}v_{j'}u'w'>0
			\end{split}
			\end{equation}\par 
			Otherwise, $a_0(s+\Delta_{j'u'w'})\!\in\! A(s+\Delta_{j'u'w'})\setminus A(s)$. Let $s\!=\!\{n_{juw}\}$, $a_0(s)\!=\!\{m^0_{juw}\}$ and $a_0(s+\Delta_{j'u'w'})\!=\!\{m'_{juw}\}$, then we have $m'_{j'u'w'}=n_{j'u'w'}+1$ and $a_0(s+\Delta_{j'u'w'})-\Delta_{j'u'w'}\in A(s)$. Since $a_0(s)$ minimizes $C(s,a)$, then
			\begin{equation}
			\begin{split}
			&C_0(s)\!-\!C[s,a_0(s\!+\!\Delta_{j'u'w'})\!-\!\Delta_{j'u'w'}]\\
			=&C_0(s)\!-\!C_p[s,a_0(s\!+\!\Delta_{j'u'w'})]\!+\!({c_b}\!-\!{c_d})v_{j'}u'w'\!-\!C_h[a_0(s\!+\!\Delta_{j'u'w'})\!-\!\Delta_{j'u'w'}]\leqslant0
			\end{split}
			\end{equation}
			As a result,
			\begin{equation}
			C_p[s,a_0(s\!+\!\Delta_{j'u'w'})]\!-\!C_0(s)\!\geqslant\!({c_b}\!-\!{c_d})v_{j'}u'w'\!-\!C_h[a_0(s\!+\!\Delta_{j'u'w'})\!-\!\Delta_{j'u'w'}]
			\end{equation}
			Given that $C_h[a_0(s+\Delta_{j'u'w'})]\geqslant C_h[a_0(s+\Delta_{j'u'w'})-\Delta_{j'u'w'}]$, we have
			\begin{equation}
			\begin{split}
			&C_0(s+\Delta_{j'u'w'})-C_0(s)\\
			=&C_p[s+\Delta_{j'u'w'},a_0(s+\Delta_{j'u'w'})]+C_h[a_0(s+\Delta_{j'u'w'})]-C_0(s)\\
			=&C_p[s,a_0(s+\Delta_{j'u'w'})]+{c_d}v_{j'}u'w'+C_h[a_0(s+\Delta_{j'u'w'})]-C_0(s)\\
			\geqslant&{c_b}v_{j'}u'w'+C_h[a_0(s+\Delta_{j'u'w'})]-C_h[a_0(s+\Delta_{j'u'w'})-\Delta_{j'u'w'}]>0
			\end{split}
			\end{equation}
			As $C_0(s+\Delta_{j'u'w'})-C_0(s)>0$ holds in all the possible cases, $C_0(s)$ is increasing in $s$. 
			\item[(ii)] By Definition \ref{pp}, if $a=\{m_{juw}\}$ and $a'=\{m'_{juw}\}$ are comparable, then $\sum_{u=1}^{U_j}\sum_{w=1}^{W_{ju}}m_{juw}=\sum_{u=1}^{U_j}\sum_{w=1}^{W_{ju}}m'_{juw}$ holds for any specialty $j$. Further, according to the model descriptions presented in Section \ref{SecModel}, patients from the same specialty have the same expectations of surgery duration and LOS. Hence by the cost function (\ref{CostDef}), $C_h(a)=C_h(a')$.
			\item[(iii)] Let $\sigma=s-\Delta_{j'u'w'}=s'-\Delta_{j'u''w''}$ and $u'w'<u''w''$, then $P(s)<P(s')$. Let $\sigma=\{\sigma_{juw}\}$, $a_0(s)=\{m^s_{juw}\}$ and $a_0(s')=\{m^{s'}_{juw}\}$, then we discuss all the possible cases: \par		
			If $a_0(s')\in A(s')\setminus A(\sigma)$, then $m^{s'}_{j'u''w''}=\sigma_{j'u''w''}+1$ and $m^{s'}_{j'u'w'}\leqslant\sigma_{j'u'w'}$, thus $a'=a_0(s')-\Delta_{j'u''w''}+\Delta_{j'u'w'}\in A(s)$. Since $C_h(a')=C_h[a_0(s')]$ by (ii) and $C_0(s)=\min_{a\in A(s)}C(s,a)$, then
			\begin{equation}
			\begin{split}
			&C_0(s')-C_0(s)=C_p[s',a_0(s')]+C_h[a_0(s')]-C_0(s)\\
			=&C_p(s,a')+{c_b}v_{j'}(u''w''-u'w')+C_h(a')-C_0(s)>C(s,a')-C_0(s)\geqslant0
			\end{split}
			\end{equation}\par
			Otherwise, $a_0(s')\in A(\sigma)$. We suppose that $a_0(s)\in A(s)\setminus A(\sigma)$, then $m^s_{j'u'w'}=\sigma_{j'u'w'}+1$ and $m^s_{j'u''w''}\leqslant\sigma_{j'u''w''}$, thus $a'=a_0(s)-\Delta_{j'u'w'}+\Delta_{j'u''w''}\in A(s')$. Since $C_0(s')=\min_{a\in A(s')}C(s',a)$, and $C_h(a')=C_h[a_0(s)]$ by (i), then
			\begin{equation}
			\begin{split}
			C[s,a_0(s)]-C[s,a_0(s')]&=C(s',a')+({c_b}-{c_d})v_{j'}(u'w'-u''w'')-C[s',a_0(s')]\\
			&>C(s',a')-C_0(s')\geqslant0
			\end{split}
			\end{equation}
			As $C[s,a_0(s)]-C[s,a_0(s')]>0$ contradicts the fact that $a_0(s)$ minimizes $C(s,a)$, $a_0(s)\in A(s)\setminus A(\sigma)$ does not hold, hence $a_0(s)\in A(\sigma)$ holds when $a_0(s')\in A(\sigma)$. Then by Lemma \ref{Lemmamin}, we have
			\begin{equation}
			C_0(s')-C_0(s)\geqslant\min_{a\in A(\sigma)}[C(s',a)-C(s,a)] ={c_d}v_{j'}(u''w''-u'w')>0
			\end{equation}\par
			To summarize, since $C_0(s')-C_0(s)>0$ holds in all the possible cases, then $C_0(s)$ is increasing in $P(s)$. \qedhere
		\end{itemize}
	\end{proof}
	\section{Proof of Proposition \ref{V}}\label{ProofV}
	\begin{proof}
		\begin{itemize}
			\item [(i)] Since $V_1^{\pi^*}(s)=V^{\pi^*}(s)$, we first prove that $V_n^{\pi^*}(s)$ is increasing in $s$ for $n=1,2,...,N$. The proof is given by backward mathematical induction:\\
			For $n=N$, since $V_{N+1}^{\pi^*}(s')=0$ holds for any $s'\in S$, then $V_N^{\pi^*}(s)=\min_{a\in A(s)}C(s,a)=C_0(s)$. $C_0(s)$ is increasing in $s$ by (i) of Proposition \ref{ConvexC}, so is $V_N^{\pi^*}(s)$.\\
			For $n=k<N$, suppose that $V_{k+1}^{\pi^*}(s)$ is increasing in $s$, then we distinguish the following cases:\par
			If $\pi^*(s+\Delta_{j'u'w'})\in A(s)$, then by Lemma \ref{Lemmamin},
			\begin{equation}
			\begin{split}
			&V_k^{\pi^*}(s+\Delta_{j'u'w'})-V_k^{\pi^*}(s)
			\geqslant\min_{a\in A(s)}[Q^{\pi^*}_k(s+\Delta_{j'u'w'},a)-Q^{\pi^*}_k(s,a)]\\
			=&\min_{a\in A(s)}\left\{{c_d}v_{j'}u'w'\!+\!\gamma\!\sum_{\Psi=\pmb{0}}^{+\infty}\!P_{\Psi}[V_{k+1}^{\pi^*}(G^{s+\Delta_{j'u'w'}}_a\!+\!\Psi)\!-\!V_{k+1}^{\pi^*}(G^s_a\!+\!\Psi)]\right\}\!>\!0
			\end{split}
			\end{equation}\par
			Otherwise, $\pi^*(s+\Delta_{j'u'w'})\!\in\! A(s+\Delta_{j'u'w'})\setminus A(s)$. Let $s\!=\!\{n_{juw}\}$, $\pi^*(s)\!=\!\{m^*_{juw}\}$ and $\pi^*(s+\Delta_{j'u'w'})\!=\!\{m^+_{juw}\}$, then we have $m^+_{j'u'w'}=n_{j'u'w'}+1$ and $\pi^-=\pi^*(s+\Delta_{j'u'w'})-\Delta_{j'u'w'}\in A(s)$, thus
			\begin{equation}
			\begin{split}
				&V^{\pi^*}_k(s+\Delta_{j'u'w'})-V^{\pi^*}_k(s)\\
				=&C[s+\Delta_{j'u'w'},\pi^-\!+\Delta_{j'u'w'}]+\gamma\sum_{\Psi=\pmb{0}}^{+\infty}P_{\Psi}V^{\pi^*}_{k+1}(G^{s+\Delta_{j'u'w'}}_{\pi^-+\Delta_{j'u'w'}}+\Psi)-V^{\pi^*}_k(s)\\
				=&C[s,\pi^-]+{c_b}v_{j'}u'w'+C_h(\pi^-+\Delta_{j'u'w'})-C_h(\pi^-)+\gamma\sum_{\Psi=\pmb{0}}^{+\infty}P_{\Psi}V^{\pi^*}_{k+1}(G^{s}_{\pi^-}+\Psi)-V^{\pi^*}_k(s)\\
				>&C[s,\pi^-]+\gamma\sum_{\Psi=\pmb{0}}^{+\infty}P_{\Psi}V^{\pi^*}_{k+1}(G^{s}_{\pi^-}+\Psi)-V^{\pi^*}_k(s)=Q^{\pi^*}_k(s,\pi^-)-V_k^{\pi^*}(s)\geqslant0
			\end{split}
			\end{equation}
			As $V^{\pi^*}_k(s+\Delta_{j'u'w'})>V^{\pi^*}_k(s)$ holds in all the two possible cases, the induction hypothesis is satisfied, thus $V^{\pi^*}_n(s)$ is increasing in $s$ for $n=1,2,...,N$. Therefore, $V^{\pi^*}(s)=V^{\pi^*}_1(s)$ is increasing in $s$.
			\item [(ii)] We employ backward mathematical induction to prove that $V_n^{\pi^*}(s)$ is increasing in $P(s)$ for $n=1,2,...,N$:\par
			For $n=N$, since $\forall s'\in S$: $V_{N+1}^{\pi^*}(s')=0$, then $V_N^{\pi^*}(s)=\min_{a\in A(s)}C(s,a)=C_0(s)$. $C_0(s)$ is increasing in $P(s)$ by (iii) of Proposition \ref{ConvexC}, so is $V_N^{\pi^*}(s)$. \par
			For $n=k<N$, suppose that $V_{k+1}^{\pi^*}(s)$ is increasing in $P(s)$. Let $\sigma=s-\Delta_{j'u'w'}=s'-\Delta_{j'u''w''}$ and $u'w'<u''w''$, then $P(s)<P(s')$. Similar to the proof of (iii) of Proposition \ref{ConvexC}, the following cases are considered:\par
			If $\pi^*(s')\in A(s')\setminus A(\sigma)$, then $a'=\pi^*(s')-\Delta_{j'u''w''}+\Delta_{j'u'w'}\in A(s)$, and $C_h(a')=C_h[\pi^*(s')]$ by (ii) of Proposition \ref{ConvexC}, hence
			\begin{equation}
			\begin{split}
			&V_k^{\pi^*}(s')-V_k^{\pi^*}(s)=C[s',\pi^*(s')]+\gamma\sum_{\Psi=\pmb{0}}^{+\infty}P_{\Psi}V_{k+1}^{\pi^*}(G^{s'}_{\pi^*(s')}+\Psi)-V_k^{\pi^*}(s)\\
			=&C(s,a')+{c_b}v_{j'}(u''w''-u'w')+\gamma\sum_{\Psi=\pmb{0}}^{+\infty}P_{\Psi}V_{k+1}^{\pi^*}(G^{s}_{a'}+\Psi)-V_k^{\pi^*}(s)>Q^{\pi^*}_k(s,a')-V_k^{\pi^*}(s)\geqslant0
			\end{split}
			\end{equation}\par
			Otherwise, $\pi^*(s')\in A(\sigma)$. Suppose that $\pi^*(s)\in A(s)\setminus A(\sigma)$, then $a'=\pi^*(s)-\Delta_{j'u'w'}+\Delta_{j'u''w''}\in A(s')$. Because $P(G^s_{\pi^*(s')}+\Psi)<P(G^{s'}_{\pi^*(s')}+\Psi)$, then $V_{k+1}^{\pi^*}(G^s_{\pi^*(s')}+\Psi)<V_{k+1}^{\pi^*}(G^{s'}_{\pi^*(s')}+\Psi)$. Given that $C_h[\pi^*(s)]=C_h(a')$ by (ii) of Proposition \ref{ConvexC}, we have
			\begin{equation}
			\begin{split}
			&V_k^{\pi^*}(s)-Q^{\pi^*}_k[s,\pi^*(s')]=C[s,\pi^*(s)]\!-\!C[s,\pi^*(s')]\!+\!\gamma\!\sum_{\Psi=\pmb{0}}^{+\infty}\!\!P_{\Psi}[V_{k+1}^{\pi^*}(G^s_{\pi^*\!(s)}\!+\!\Psi)\!-\!V_{k+1}^{\pi^*}(G^s_{\pi^*\!(s')}\!+\!\Psi)]\\
			>&C(s',a')\!-\!C[s',\pi^*(s')]\!+\!({c_b}\!-\!{c_d})v_{j'}(u'w'\!-\!u''w'')\!+\gamma\!\sum_{\Psi=\pmb{0}}^{+\infty}\!\!P_{\Psi}[V_{k+1}^{\pi^*}(G^{s'}_{a'}\!+\!\Psi)\!-\!V_{k+1}^{\pi^*}(G^{s'}_{\pi^*(s')}\!+\!\Psi)]\\
			>&Q^{\pi^*}(s',a')-V_k^{\pi^*}(s')\geqslant0
			\end{split}
			\end{equation}
			As $V_k^{\pi^*}(s)-Q^{\pi^*}_k[s,\pi^*(s')]>0$ contradicts the fact that $\pi^*(s)$ minimizes $Q^{\pi^*}_k(s,a)$, $\pi^*(s)\in A(s)\setminus A(\sigma)$ does not hold, hence $\pi^*(s)\in A(\sigma)$. Then by Lemma \ref{Lemmamin},
			\begin{equation}
			\begin{split}
			&V_k^{\pi^*}(s')-V_k^{\pi^*}(s)\geqslant\min_{a\in A(\sigma)}[Q^{\pi^*}(s',a)-Q^{\pi^*}(s,a)]\\
			=&\!\min_{a\in A(\sigma)}\!\left\{\!{c_d}v_{j'}(u''w''\!-\!u'w')\!+\!\gamma\!\sum_{\Psi=\pmb{0}}^{+\infty}\!P_{\Psi}[V_{k+1}^{\pi^*}(G^{s'}_a\!+\!\Psi)\!-\!V_{k+1}^{\pi^*}(G^s_a\!+\!\Psi)]\!\right\}\!>\!0
			\end{split}
			\end{equation}
			Since $V_k^{\pi^*}(s')-V_k^{\pi^*}(s)>0$ holds in all the possible cases, the induction hypothesis is satisfied, thus $V_n^{\pi^*}(s)$ is increasing in $P(s)$ for $n=1,2,...,N$, hence $V^{\pi^*}(s)=V_1^{\pi^*}(s)$ is increasing in $P(s)$. \qedhere
		\end{itemize}
	\end{proof}
	\section{Proof of Proposition \ref{pi}}\label{Proofpi}
	\begin{proof}
		Let $s=\{n_{juw}\}$, $\pi^*(s)=\{m^*_{juw}\}$ and $({c_d}-{c_b})v_{j'}u'w'>c_o\bar{d}_{j'}+{c_e}\bar{l}_{j'}$. Suppose that $n_{j'u'w'}-m^*_{j'u'w'}=x>0$, then $a'=\pi^*(s)+x\Delta_{j'u'w'}\in A(s)$. By (\ref{CostDef}), we know that $C_h(a')-C_h[\pi^*(s)]\leqslant (c_o\bar{d}_{j'}+{c_e}\bar{l}_{j'})x$, then
		\begin{equation}
		\begin{split}
		C[s,\pi^*(s)]-C(s,a')
		=&({c_d}-{c_b})v_{j'}u'w'x+C_h[\pi^*(s)]-C_h(a')\\
		\geqslant&[({c_d}-{c_b})v_{j'}u'w'-c_o\bar{d}_{j'}-{c_e}\bar{l}_{j'}]x>0
		\end{split}
		\end{equation}
		Besides, since $G^s_{\pi^*(s)}+\Psi>G^s_{a'}+\Psi$, and $V^{\pi^*}(s)$ is increasing in $s$ by (i) of Proposition \ref{V}, then $V^{\pi^*}(G^s_{\pi^*(s)}+\Psi)-V^{\pi^*}(G^s_{a'}+\Psi)>0$. Therefore, 
		\begin{equation}
		Q^{\pi^*}[s,\pi^*(s)]-Q^{\pi^*}(s,a')=C[s,\pi^*(s)]-C(s,a')+\gamma\sum_{\Psi=\pmb{0}}^{+\infty}P_\Psi[V^{\pi^*}(G^s_{\pi^*(s)}+\Psi)-V^{\pi^*}(G^s_{a'}+\Psi)]>0
		\end{equation}
		Since $Q^{\pi^*}[s,\pi^*(s)]-Q^{\pi^*}(s,a')>0$ contradicts the fact that $\pi^*(s)$ minimizes $Q^{\pi^*}(s,a)$, then $n_{j'u'w'}-m^*_{j'u'w'}=x>0$ does not hold, proving the proposition.\qedhere
	\end{proof}
	\section{Proof of Proposition \ref{PropVQ}}\label{ProofPropVQ}
	\begin{proof}
		Let $a=a'+\Delta_{j'u'w'}-\Delta_{j'u''w''}$ and $u'w'<u''w''$, then $P(a)<P(a')$ and $P(G^s_a+\Psi)>P(G^s_{a'}+\Psi)$, thus $V^{\pi^*}(G^s_a+\Psi)>V^{\pi^*}(G^s_{a'}+\Psi)$ holds by (ii) of Proposition \ref{V}, hence
		\begin{equation}
		Q^{\pi^*}(s,a)-Q^{\pi^*}(s,a')=({c_b}-{c_d})v_{j'}(u'w'-u''w'')+\gamma\sum_{\Psi=\pmb{0}}^{+\infty}P_{\Psi}[V^{\pi^*}(G^s_a+\Psi)-V^{\pi^*}(G^s_{a'}+\Psi)]>0
		\end{equation}
		As $Q^{\pi^*}(s,a)-Q^{\pi^*}(s,a')>0$ holds for $P(a)<P(a')$, the proposition is proved.
	\end{proof}

\end{appendices}
\end{document}